\documentclass[11pt]{amsart}

\input{mathPreambuleReport}

\usepackage[T1]{fontenc}       
\usepackage[sc]{mathpazo}          
\usepackage[mathscr]{eucal}
\usepackage[scaled=0.86]{berasans} 
\usepackage[scaled=1]{inconsolata} 

\usepackage{fontspec}
\newfontfamily\cyrfont{cmunrm.otf}

\usepackage{tensor}
\usepackage[]{natbib}
\usepackage{tcolorbox}
\usepackage{mdframed}

\usepackage{tikz}
\usetikzlibrary{patterns,decorations.markings,calc,arrows.meta}

\definecolor{moved}{RGB}{0,100,180}
\definecolor{note}{RGB}{80,120,60}

\renewcommand{\phi}{\varphi}

\usepackage{titlesec}

\titleformat{\section}
  {\normalfont\Large\bfseries}{\S\thesection}{1em}{}
\titleformat{\subsection}
  {\normalfont\large\bfseries}{\S\thesubsection}{1em}{}
\titleformat{\subsubsection}
  {\normalfont\normalsize\bfseries}{\S\thesubsubsection}{1em}{}

\newcounter{partctr}
\renewcommand{\thepart}{\Roman{partctr}}
\newcommand{\Part}[1]{%
  \clearpage
  \refstepcounter{partctr}%
  \addcontentsline{toc}{part}{Part \thepart.\ #1}%
  \vspace*{3ex}%
  {\centering\LARGE\bfseries Part \thepart\par\medskip\Large #1\par}%
  \vspace{4ex}%
}

\newtheorem*{mainthm}{Theorem}

\title{
Singularities of matrix semicircles
}
\bigskip
\author{
Vladislav Kargin
}
\thanks{email:
vkargin@binghamton.edu; current address: 4400 Vestal Pkwy East, Department of Mathematics, Binghamton University, Binghamton, 13902-6000, USA}

\begin{document}

\maketitle

\begin{abstract}
Let $S=\sum_{i=1}^r A_i\otimes s_i$ be a matrix semicircular element, with
Hermitian coefficients $A_i\in M_n(\bC)$ and free standard semicircular
generators $s_i$. Its scalar spectral density $f$ is governed, through
Speicher's equation (a matrix Dyson equation), by the completely positive
covariance map $\eta_S(X)=\sum_i A_iXA_i$. We treat the singular regime, in
which the underlying Hermitian pencil $\sum_i A_i x_i$ is full but not
semisimple and $f$ is unbounded at the origin, in contrast to the bounded,
real-analytic density of the regular case.

We prove three results. (i) The leading singularity exponent at $0$ is
invariant under congruence $A_i\mapsto bA_ib^{*}$ of the pencil ($b$
invertible) -- or, more generally, under symmetric scaling of the covariance map.
(ii) For binary elements ($r=2$) we obtain a complete classification: in
Lancaster--Rodman canonical form every indecomposable cell is of one of three
types, and $f(x)\sim c\,|x|^{-(n^*-1)/(n^*+1)}$ as $x\to0$ with an explicit
constant $c$, where the exponent depends only on the size of the largest
Jordan block of the pencil (the effective chain length $n^*$) and not on the
coupling. Together with~(i) and the behaviour over direct sums, this
classifies the singularities of all full binary Hermitian pencils. (iii) The
spectral classification is strictly coarser than the algebraic one: a
Type~III cell of size $2m$ with non-real parameter $\beta$ and the direct sum
of two Type~II cells of size $m$ with parameter $|\beta|$ have identical
scalar densities, yet their covariance maps are not symmetrically scalable to
one another -- the scalar spectrum cannot detect the phase of $\beta$.

Each type calls for a different method: a reduction of Speicher's equation to
an autonomous discrete Painlev\'e~I (McMillan) map (Type~I), a
Lyapunov--Schmidt reduction at the branch point (Type~II), and a gauge
reduction by a diagonal unitary (Type~III).
\end{abstract}

\newpage

\contentsmargin{2.55em}
\dottedcontents{section}[3.8em]{}{2.3em}{1pc}
\dottedcontents{subsection}[6.1em]{}{3.2em}{1pc}
\tableofcontents  

\section{Introduction} 
\label{sec:introduction}

A \emph{matrix semicircular variable} is a random variable of the form
$S = \sum_{i=1}^{r} A_i \otimes s_i$, where $A_1,\ldots,A_r$ are 
Hermitian matrices in $M_n(\bC)$ and $s_1,\ldots,s_r$ are free 
standard semicircular variables in a non-commutative probability 
space $(\mathcal{A},\varphi)$. Such variables arise as the 
asymptotic limits of large random matrices --- band matrices, 
block matrices, and many structured ensembles --- and their 
spectral distributions encode the interaction between the 
algebraic structure of the matrix coefficients and the free 
probabilistic behavior of the semicircular generators.

The spectral properties of~$S$ are determined by its 
\emph{covariance map} $\eta_S(X) = \sum_{i} A_i X A_i$, a 
completely positive self-adjoint linear map on $M_n(\bC)$, via 
a matrix-valued quadratic equation (Speicher's equation).

The behavior of the spectral density $f_S$ at the origin is governed by a sharp
dichotomy, and the analytic difficulty is concentrated on its
singular side. A companion paper~\cite{kargin2026_algebraic}
located the dividing line: $f_S$ is bounded and real-analytic
at~$0$ exactly when the underlying pencil is LR-semisimple,
equivalently when the covariance map is symmetrically
DS-scalable. (Here \emph{DS-scalability} refers to scalability
of a single map to a doubly-stochastic normal form, a notion
distinct from the symmetric scalability \emph{relation} of
Definition~\ref{defi_sym_scalability}.) That theorem identifies
\emph{which} pencils are singular but is silent on \emph{what
kind} of singularity they produce, and it is the latter question
that carries the analytic content. We take it up here. Once
LR-semisimplicity fails, the regular-point arguments that force
analyticity in the smooth case break down: the operator-valued
Cauchy transform acquires a branch point at the origin, and
locating its leading behavior becomes a problem in the local
analysis of Speicher's equation that the implicit function
theorem no longer settles. We resolve this problem completely
for binary pencils, and we find that the singularity one
recovers retains strictly less information than the pencil that
produced it.

We assume throughout that the pencil is full, so 
that the spectral distribution has no atom at zero by
\cite{hoffmann_mai_speicher2024}. Under this assumption, the 
singularities of the density can only be algebraic branch 
points: $f_S(x) \sim c\,|x|^{\alpha}$ as $x \to 0$, where 
$\alpha > -1$ is a rational number (this follows from the 
algebraicity of the Cauchy transform, 
Proposition~\ref{prop:entrywise-algebraicity}).

Our aim is to answer two questions:
\begin{enumerate}[label=(\alph*)]
\item Does the singularity type change under natural 
transformations of the covariance map?
\item Can we classify the singularities explicitly?
\end{enumerate}
The answers lead to three main results.

\subsection*{Result A: Singularity exponents are invariant 
under scaling of the covariance map}

\begin{defi} \label{defi_sym_scalability}
Two covariance maps $\eta_A$ and $\eta_B$ on $M_n(\bC)$ are 
\emph{symmetrically scalable} to each other if there exists an 
invertible matrix $b \in M_n(\bC)$ such that 
$\eta_B(X) = b\,\eta_A(b^\ast X b)\, b^\ast$ for all $X \in M_n(\bC)$.
\end{defi}
This is the covariance-map analogue of the congruence 
transformation $A_i \mapsto b A_i b^\ast$ at the pencil level; it 
preserves the class of completely positive self-adjoint maps 
and arises naturally in the theory of operator scaling 
\cite{gurvits2004, ggow2020}.

\medskip
\begin{mdframed}
\begin{mainthm}[A. Invariance of the leading singularity 
exponent under covariance-map scaling]
\label{thm:intro_invariance}
Let $S_A$ and $S_B$ be two full self-adjoint matrix semicircles, 
with covariance maps $\eta_A$ and $\eta_B$.
Assume that $\eta_A$ and $\eta_B$ are symmetrically scalable. 
Let $f_A$ and $f_B$ denote the scalar densities of $S_A$ and 
$S_B$. Suppose that $f_A$ has an algebraic singularity at~$0$ 
with leading exponent $-\alpha$, where $\alpha \in [0,1)$.
Then $f_B$ has an algebraic singularity at~$0$ with the same 
leading exponent $-\alpha$. In other words, the leading 
singularity exponent at~$0$ is invariant under symmetric 
scalability of covariance maps.
\end{mainthm}
\end{mdframed}
\medskip

\noindent
The proof 
(Theorem~\ref{thm:covariance-scaling-singularity-exponent}
\S\ref{sec:invariance}) proceeds by sandwiching the Poisson transform 
$\mathcal{P}_{S_B}(\varepsilon)$ between rescaled copies of 
$\mathcal{P}_{S_A}(\varepsilon)$, and using two Tauberian 
lemmas to pass between density asymptotics and Poisson 
asymptotics.

Result~A answers question~(a): the singularity exponent 
depends on the covariance map only through its symmetric 
scaling class. In particular, for Hermitian binary pencils, the 
exponent depends only on the Lancaster--Rodman canonical form 
\cite{lancaster_rodman2005}. This motivates computing the 
exponent for each canonical cell, which is the content of 
Result~B.

\subsection*{Result B: Classification of singularities for 
binary matrix semicircles}

A binary matrix semicircle has $r = 2$ generators:
$S = A_1 \otimes s_1 + A_2 \otimes s_2$, and the
associated Hermitian pencil $A = A_1 x_1 + A_2 x_2$.
By the Lancaster--Rodman classification
\cite{lancaster_rodman2005}, every such pencil is
congruent to a direct sum of indecomposable cells of
three canonical types --- I, II, and III --- 
distinguished by whether the associated eigenvalue is
infinite, real, or a conjugate non-real pair 
(see \S\ref{sec:LaRo_cells} for the precise
definitions). We prove the following for cells of
size $n \ge 2$ (Type~I, II) or $m \ge 2$ (Type~III).

\medskip
\begin{mdframed}
\begin{mainthm}[B. Singularities of indecomposable binary 
matrix semicircles]
\label{thm:intro-classification}
Let $A = A_1 x_1 + A_2 x_2$ be an indecomposable Hermitian 
binary pencil in Lancaster--Rodman canonical form, and let 
$S = A_1 \otimes s_1 + A_2 \otimes s_2$ be the associated 
matrix semicircular element. Define the \emph{effective chain 
length} $n^*$ and \emph{coupling modulus} $\gamma$ by
\[
\renewcommand{\arraystretch}{1.3}
\begin{array}{lccc}
\hline
\textup{Cell type} & \textup{Size} & n^* & \gamma \\
\hline
\textup{Type I:}\  F_n x_1+G_n x_2
        & n & n & 0 \\[2pt]
\textup{Type II:}\  F_n x_1+(\alpha F_n+G_n)x_2,\ 
        \alpha\in\bR
        & n & n & |\alpha| \\[2pt]
\textup{Type III:}\ 
        F_{2m}x_1+\bigl(\begin{smallmatrix}0&\beta F_m+G_m\\
        \bar\beta F_m+G_m&0\end{smallmatrix}\bigr)x_2,\ 
        \beta\in\bC\setminus\bR
        & 2m & m & |\beta| \\
\hline
\end{array}
\]
Set $c := 1 + \gamma^2$. Then the spectral density satisfies
\begin{equation}\label{eq:intro-density}
f(x)
\;=\;
\frac{\sin\!\bigl(\pi/(n^*{+}1)\bigr)}{n^*\,\pi}\;
c^{-\frac{2n^*-1}{n^*+1}}\;
|x|^{-\frac{n^*-1}{n^*+1}}
\;+\;
o\!\left(|x|^{-\frac{n^*-1}{n^*+1}}\right)
\end{equation}
as $|x| \to 0$.
\end{mainthm}
\end{mdframed}
\medskip

For instance, the simplest non-trivial Type~I cell ($n=2$)
produces a cusp singularity $f(x) \sim 
\frac{\sqrt{3}}{4\pi}|x|^{-1/3}$ (Example~\ref{exa:cusp_singularity}).

\noindent
Several features of this classification deserve comment.
The singularity exponent $(n^* - 1)/(n^* + 1)$ depends only 
on the effective chain length and is independent of the 
coupling parameter~$\gamma$; by Result~A, it is also invariant 
under symmetric scaling of the covariance map. Equivalently, with $d:=n^*-1$ the defect of semisimplicity (the size of the largest
Jordan block of the pencil, minus one), the exponent is $d/(d+2)$, which is the form of the Kr\"uger--Renfrew
singularity degree~(\cite{MR4901646}).\footnote{Our exponents $\beta_k$ are
those of \cite[Lem.~3.1]{MR4901646} up to a relabeling of indices and the sign
convention there ($v_i\sim\eta^{-f_i}$ versus $y_k\sim u^{\beta_k}$): as
multisets $\{\beta_k\}=\{-f_i\}$, the most singular value being
$\beta_1=-(n-1)/(n+1)=-\sigma$, so the leading exponent agrees with
\cite[Thm.~2.8]{MR4901646}.} (with the algebraic defect $d$ in the role of
their combinatorial chain length) --- the defect filtered through the
quadratic term of Speicher's equation; Remark~\ref{rem:Nstar-geometry}
gives the underlying indefinite-metric picture.

Type~I is the 
special case $\gamma = 0$ of Type~II (with $c = 1$). For 
Type~III cells of matrix size~$2m$, the singularity is governed 
by the half-size chain length~$m$, not by~$2m$; this reflects 
a gauge reduction 
(Theorem~\ref{thm:typeIII-scalar-law-caseII}) that identifies 
the Type~III scalar density with a Type~II density of half the 
matrix size.

By Lancaster–Rodman every binary Hermitian pencil is congruent to a direct sum of canonical cells; since the singularity at the origin is invariant under congruence (Theorem A) and, over a direct sum, is governed by the most singular summand (\S\ref{sec:direct_sums}), the cell-by-cell classification of Theorem B is in fact a complete classification of singularities for full binary Hermitian pencils — the exponent of a general pencil being the largest occurring among the cells of its canonical form.

The proofs of Result~B employ different techniques for each 
type:
\begin{itemize}[nosep]
\item \emph{Type~I} (\S\ref{sec:type_I}): the 
covariance map preserves the diagonal subalgebra, reducing Speicher's equation to a discrete dynamical system --- a chain of 
coupled equations equivalent to the McMillan map (autonomous 
discrete Painlev\'e~I). A biquadratic integral of motion 
constrains consecutive pairs, and a cancellation mechanism in 
the two-step recurrence pins down the Puiseux exponents.

\item \emph{Type~II} 
(\S\ref{sec:singularities_type_2}--\S\ref{sec:proof_thm_case_II}): 
the diagonal subalgebra is no longer preserved. A rescaling 
$W_\alpha(u) = D(t)M_\alpha(t)D(t)$ with $t = u^{1/(n+1)}$ 
transforms Speicher's equation into a system that is regular 
at $t = 0$. A Lyapunov--Schmidt reduction and the holomorphic 
implicit function theorem produce a convergent expansion 
of~$M_\alpha(t)$.

\item \emph{Type~III} (\S\ref{sec:type_III}): a gauge reduction by a 
diagonal unitary shows that the Type~III Speicher equation with 
parameter $\beta \in \bC \setminus \bR$ decomposes into two 
copies of the Type~II equation with real parameter $|\beta|$.
\end{itemize}

\subsection*{Result C: The spectral classification is strictly 
coarser than the algebraic one}

Result~A shows that symmetrically scalable covariance maps 
produce densities with the same singularity exponent. The 
gauge reduction in the Type~III analysis proves something 
stronger: the densities are not merely similar but 
\emph{identical}. One might then ask whether equal densities 
force the covariance maps to be symmetrically scalable.
The answer is no.

\medskip
\begin{mdframed}
\begin{mainthm}[C. Spectral coincidence without algebraic 
equivalence]
\label{thm:intro-gap}
Let $\beta \in \bC \setminus \bR$. 
A \TypeIII{} cell of size $2m$ and parameter $\beta$ has the same scalar spectral density as the direct sum of two identical \TypeII{} cells of size $m$ and parameter $|\beta|$, but  their covariance maps are not symmetrically scalable to 
each other.
\end{mainthm}
\end{mdframed}
\medskip

\noindent
The spectral coincidence is
Theorem~\ref{thm:typeIII-scalar-law-caseII}; the algebraic
non-equivalence is Theorem~\ref{thm:typeIII-not-scalable} (with the
plane invariant of Proposition~\ref{prop:typeIII-definite-kraus}).

The obstruction 
is visible at the level of Hermitian Kraus planes: the Type~III 
plane $\mathrm{span}_\bR\{A_1, A_2\}$ consists entirely of 
invertible matrices (apart from~$0$), while any direct sum of 
Type~II planes contains nonzero singular elements. Since the 
property ``the Hermitian Kraus plane intersects the singular 
locus only at the origin'' is preserved under congruence, the 
two covariance maps lie in different symmetric scaling classes.

Result~C shows that the Lancaster--Rodman classification of 
Hermitian matrix pencils is strictly finer than what the scalar 
spectral density can detect. The covariance map retains 
algebraic information --- specifically, the phase of the 
coupling parameter $\beta$ in Type~III cells --- that is 
invisible to the spectral measure. Thus, while symmetric 
scaling of the covariance map cannot change the singularity type 
(Result~A), non-equivalent covariance maps may nonetheless 
produce identical densities.
\subsection*{Independent contributions and context}

Although the question settled here is the natural counterpart of
the one in~\cite{kargin2026_algebraic}, its answer stands on its
own. We isolate four contributions, then place them against the
closest prior work.

\emph{A complete and explicit classification.}
For every full binary Hermitian pencil, Theorem~B returns the
leading singularity exponent together with its constant in closed
form~\eqref{eq:intro-density}, the exponent read off from the
Jordan structure of the pencil through the defect $d=n^*-1$. The
classification is thus an exact dictionary between the
Lancaster--Rodman canonical form of an indefinite pencil and the
local profile of the associated spectral measure---the analytic
singularity realized as an invariant of indefinite linear
algebra. At the level of the question the two papers are
complementary: \cite{kargin2026_algebraic} decides \emph{when} a
singularity occurs, the present paper says \emph{what} it is, and
together they settle the local picture at the origin.

\emph{Invariance of the exponent (Theorem~A).}
The leading exponent depends on the covariance map only through
its symmetric-scaling class; equivalently, it is constant on the
congruence orbit of the pencil. This makes the singularity degree
an invariant of the operator-scaling action~\cite{gurvits2004,
ggow2020}, computable from any convenient representative, and it
is the mechanism by which the cell-by-cell computation of
Theorem~B extends to arbitrary pencils.

\emph{A bridge to discrete integrable systems.}
The Type~I analysis identifies the restriction of Speicher's
equation to the diagonal subalgebra with the autonomous discrete
Painlev\'e~I, or McMillan, map; the Puiseux exponents are then
forced by a conserved biquadratic and a cancellation in the
associated two-step recurrence. This places the local spectral
analysis of matrix semicircles, in the most degenerate cases, in
contact with the theory of QRT maps and discrete Painlev\'e
equations.

\emph{A spectral phenomenon with no algebraic counterpart
(Theorem~C).}
The scalar spectral measure cannot detect the phase of the
Type~III coupling~$\beta$: two covariance maps in distinct
symmetric-scaling classes produce \emph{identical} densities. The
Lancaster--Rodman classification is therefore strictly finer than
any invariant of the scalar spectrum, and Theorem~C pins down
exactly the algebraic data the spectrum forgets. This is a
self-contained negative result, independent
of~\cite{kargin2026_algebraic}.

\medskip
\emph{Relation to prior work.}
Under a flatness, or uniform-primitivity, hypothesis on the
self-energy operator, the now-standard regularity theory of the
matrix Dyson equation forces a \emph{bounded} density, with
non-analyticities confined to square-root edges and cubic-root
cusps~\cite{AjaErdKru2017, ErdPCMI2019, AltErdKru2020}. Dropping
flatness while keeping the self-energy \emph{commutative}
---structured matrices with a variance profile---~\cite{MR4901646} introduced the \emph{singularity degree}
$\sigma=\ell/(\ell+2)$, governed by the zero pattern of the
profile $S$ through a trichotomy: total support gives a bounded
density, support without total support a blow-up $|x|^{-\sigma}$,
and absence of support an atom (here $\ell$ is the length of the
longest increasing chain in a relation read off $S$). This
trichotomy is the commutative shadow of the dichotomy organising
the present series: total support, support, and no support
correspond to LR-semisimple (the bounded analytic case
of~\cite{kargin2026_algebraic}), full but not semisimple (our
blow-up), and non-full (an atom, excluded here by fullness).

\emph{The non-commutative setting.}
The covariance maps studied here need not be commutative. Type~I
preserves the diagonal subalgebra, so its scalar law coincides
with the variance-profile law of $S=\eta_S|_{\Delta}$, with
$S_{ij}=|(A_1)_{ij}|^2+|(A_2)_{ij}|^2$; there the chain exponent
$(n{-}1)/(n{+}1)$ is exactly the Kr\"uger--Renfrew degree with
$\ell=n-1$. For Types~II and~III the covariance map preserves no maximal abelian
subalgebra (Lemma~\ref{lem:no-MASA}), so their scalar laws are
variance-profile laws in no basis, the diagonal reduction
underlying~\cite{MR4901646} no longer applies, and the analysis
proceeds at the matrix level; yet the
same law $\sigma=d/(d+2)$ persists, with the combinatorial $\ell$
replaced by the algebraic defect $d=n^*-1$ of the
Lancaster--Rodman form. The resulting blow-ups form an infinite
discrete family approaching, without reaching, the non-integrable
threshold $|x|^{-1}$ as the largest Jordan block grows---behaviour
the flatness hypothesis excludes. The sharp constant, the
scaling-invariance (Theorem~A), and the spectral/algebraic gap
(Theorem~C) have no counterpart in~\cite{MR4901646}.

\medskip
\noindent
The length of the paper reflects this breadth rather than
incremental elaboration. The three canonical types are not
variants of a single argument but require qualitatively different
machinery---a discrete-integrable reduction (Type~I), a
Lyapunov--Schmidt analysis of a degenerate fixed-point equation
(Type~II, the longest and most delicate of the three), and a
gauge reduction by a diagonal unitary (Type~III)---each carried
out in full.

\subsection*{Organization of the paper}

The paper is divided into two parts.

\emph{Part~I} (\S\S\ref{sec:background}--\ref{sec:direct_sums}) develops the general 
theory. Section~\ref{sec:background} collects background material: matrix 
semicircles, covariance maps, Speicher's equation, 
algebraicity of the Cauchy transform, and a motivating example 
(the $n = 2$ cusp singularity). 
Section~\ref{sec:invariance} proves the invariance of the singularity 
exponent under symmetric scaling of covariance maps (Result~A) 
via the Poisson comparison method. Section~\ref{sec:direct_sums} analyzes 
direct sums, showing that the density of a decomposable 
semicircle is a weighted average of the component densities, 
with the singularity governed by the most singular summand.

\emph{Part~II} (\S\S\ref{sec:type_I}--\ref{sec:classification-synthesis}) classifies the
singularities of binary matrix semicircles (Result~B).
Section~\ref{sec:type_I} treats Type~I cells via the chain system and its
Painlev\'e structure. Sections~\ref{sec:type_II}--\ref{sec:LS} treat Type~II cells
via rescaling, Lyapunov--Schmidt reduction, and the implicit function theorem.
Section~\ref{sec:type_III} treats Type~III cells via gauge reduction to Type~II, and
establishes Result~C. Section~\ref{sec:classification-synthesis} assembles the cell
computations into a complete classification for arbitrary regular Hermitian binary
pencils: via congruence-invariance (Result~A) and the direct-sum reduction, the leading
singularity is governed by the largest effective chain length $N^*$ of the
Lancaster--Rodman form --- the density being regular at the origin exactly when the
pencil is LR-semisimple --- and the geometric meaning of $N^*$ is recorded there.

The appendix collects the deferred proofs: the algebraicity of the scalar Cauchy
transform (Appendix~\ref{sec:algebraicity}); the two Tauberian lemmas relating Puiseux
density asymptotics to Poisson asymptotics (Appendix~\ref{sec:tauberian}); and the
computations supporting the Type~II analysis (Appendix~\ref{sec:type2-computations}) ---
real-analyticity of the solution of Speicher's equation in the coupling parameter, the
negative-definiteness of the Lyapunov--Schmidt linearization~$\BB_\alpha$, the
reflection identity, and the explicit shifted base point.

\Part{General principles}


\section{Matrix semicircles and algebraic singularities} 
\label{sec:background}
Let $(\mathcal{A}, \varphi)$ be a non-commutative probability 
space, where $\varphi$ is a faithful tracial state. Self-adjoint 
elements $s_1, \ldots, s_r \in \mathcal{A}$ are called 
\emph{free standard semicircular variables} if they are freely 
independent and each $s_i$ has the semicircular distribution 
$d\mu(x) = \frac{1}{2\pi}\sqrt{4 - x^2}\,dx$ on $[-2, 2]$.

\begin{defi}
A \emph{matrix semicircular variable} is an element 
$S \in M_n(\mathbb{C}) \otimes \mathcal{A}$ of the form
\begin{equation}
\label{equ_defi_semicircle}
S = \sum_{i=1}^{r} A_i \otimes s_i,
\end{equation}
where $A_1, \ldots, A_r \in M_n(\mathbb{C})$ are Hermitian 
matrices and $s_1, \ldots, s_r \in \mathcal{A}$ are free 
standard semicircular variables.
\end{defi}

We write 
$\mathbb{E} = \mathrm{id}_{M_n(\mathbb{C})} \otimes \varphi 
\colon M_n(\mathbb{C}) \otimes \mathcal{A} \to M_n(\mathbb{C})$ 
for the conditional expectation onto the matrix subalgebra.

\begin{defi}
The \emph{covariance map} of a matrix semicircular variable 
$S = \sum_{i=1}^r A_i \otimes s_i$ is the linear map 
$\eta_S \colon M_n(\mathbb{C}) \to M_n(\mathbb{C})$ defined by
\begin{equation}
\label{defi_covariance}
\eta_S(X) := \mathbb{E}[S X S] 
  = \sum_{i=1}^{r} A_i X A_i,
\end{equation}
where the second equality follows from 
$\varphi(s_i s_j) = \delta_{ij}$ (by freeness and the 
normalization of the semicircular law). Since each $A_i$ 
is Hermitian, $\eta_S$ is a completely positive map satisfying 
$\eta_S = \eta_S^*$ (self-adjoint with respect to the 
Hilbert--Schmidt inner product).
\end{defi}

The spectral distribution of $S$ is encoded by its Cauchy 
transform. For $b \in M_n(\mathbb{C})$ such that 
$b \otimes \mathrm{id}_{\mathcal{A}} - S$ is invertible, 
define
\begin{equation}
\label{defi_Cauchy}
G_S(b) = \mathbb{E}\big[
  (b \otimes \mathrm{id}_{\mathcal{A}} - S)^{-1}\big].
\end{equation}
In particular, for $z \in \mathbb{C}^+ := \{z : \Im z > 0\}$, 
the scalar evaluation $G_S(zI)$ is well-defined. The 
\emph{spectral density} of $S$ (when it exists) is recovered 
from the Stieltjes inversion formula:
\begin{equation}
\label{equ_Stieltjes}
f(x) = \frac{1}{\pi} 
  \lim_{\varepsilon \downarrow 0} 
  \Im\big[-\tr\, G_S\big((x + i\varepsilon)I\big)\big],
\end{equation}
where $\tr$ denotes the normalized trace on $M_n(\mathbb{C})$.

The Cauchy transform can be computed from the covariance map 
via the following result.

\begin{theo}[Speicher's equation]
\label{theo_Cauchy_matrix_semicircle}
Let $S$ be a matrix semicircular variable with covariance 
map~$\eta_S$. Then for all $b \in M_n(\mathbb{C})$ with 
$\Im b \succ 0$, the Cauchy transform $G = G_S(b)$ satisfies
\begin{equation}
\label{equ_Speicher}
b\, G(b) = I + \eta_S\!\big(G(b)\big)\, G(b).
\end{equation}
\end{theo}

Following \cite{hfs2007}, it is convenient to reformulate 
equation~\eqref{equ_Speicher} on the right half-plane. 
Recall that a matrix $W \in M_n(\mathbb{C})$ is 
\emph{accretive} if $\Re W := (W + W^*)/2 \succeq 0$, and 
\emph{strictly accretive} if $\Re W \succ 0$. The 
substitution $u = -iz$ and $W(u) := iG_S(iu\cdot I)$ 
transforms Speicher's equation at $b = iu \cdot I$ into
\begin{equation}
\label{eq:HMS}
\eta_S\!\big(W(u)\big)\, W(u) + u\, W(u) = I.
\end{equation}


\begin{propo}[Entrywise algebraicity of the matrix Cauchy transform]
\label{prop:entrywise-algebraicity}
Let
\[
S=\sum_{j=1}^r A_j\otimes s_j
\]
be a matrix semicircular element with covariance map
$\eta(X)=\sum_{j=1}^r A_j X A_j^*$, and let
\[
G(z):=G_S(zI_n)\in M_n(\mathbb{C}),\qquad z\in\mathbb{C}^+.
\]
Then for every $1\le p,q\le n$, the scalar function $G_{pq}(z)$ is
algebraic over $\mathbb{C}(z)$: there exists a nonzero polynomial
$P_{pq}(z,w)\in\mathbb{C}[z,w]$ such that
\[
P_{pq}\bigl(z,G_{pq}(z)\bigr)=0\qquad\text{for all } z\in\mathbb{C}^+.
\]
In particular, $H(z)=\tr G(z)=\tfrac{1}{n}\sum_p G_{pp}(z)$ is
algebraic over $\mathbb{C}(z)$.
\end{propo}

For a proof see Appendix \ref{sec:algebraicity}.

\begin{propo}[$H$ has real-coefficient minimal polynomial]
The minimal polynomial of $H(z)=\tr G(z)$ over $\mathbb{C}(z)$ 
has coefficients in $\mathbb{R}(z)$. Equivalently, there is 
a nonzero polynomial $P(z,w)\in\mathbb{R}[z,w]$, irreducible 
in $\mathbb{C}[z,w]$, with $P(z,H(z))=0$ on $\mathbb{C}^+$.
\end{propo}

\begin{proof}
Let $P(w)=w^d+a_{d-1}(z)w^{d-1}+\cdots+a_0(z)\in\mathbb{C}(z)[w]$ be 
the minimal polynomial of $H$ over~$\mathbb{C}(z)$. Since $\mu_S$ is a 
real measure, Schwarz reflection gives $\overline{H(z)}=H(\bar z)$ for 
$z\in\mathbb{C}\setminus\mathbb{R}$. Conjugating the identity 
$P(z,H(z))=0$ yields $\bar P(\bar z,H(\bar z))=0$ on~$\mathbb{C}^+$, 
where $\bar P$ denotes the polynomial obtained by conjugating the 
coefficients of each~$a_i$. By analytic continuation, $\bar P$ also 
annihilates~$H$ on~$\mathbb{C}^+$. Since $\bar P$ is monic of 
degree~$d$, minimality forces $\bar P=P$, so 
$a_i\in\mathbb{R}(z)$ for every~$i$. Clearing denominators gives 
the claimed element of~$\mathbb{R}[z,w]$.
\end{proof}

\begin{coro}[Structure of the spectral measure]
\label{cor:spectral-structure}
The spectral measure decomposes as
\[
\mu_S \;=\; f(x)\,dx \;+\; \sum_{j=1}^{N} m_j\,\delta_{x_j},
\qquad N<\infty,
\]
where the density $f$ is real-analytic on $\mathbb{R}$ except 
at finitely many points, at which it has either a pole 
(corresponding to an atom $x_j$) or an algebraic singularity.
\end{coro}

\begin{proof}
Standard: $H$ is algebraic with real-coefficient minimal 
polynomial, so its singularities in $\mathbb{C}$ form a finite 
set (zeros of the discriminant and leading coefficient). Real 
poles of $H$ correspond to atoms of $\mu_S$ via Stieltjes 
inversion, with masses given by residues; real branch points 
give algebraic singularities of $f$ on each side via Puiseux 
expansion. Off this finite set, $H$ extends analytically across 
$\mathbb{R}$, so the singular continuous part of $\mu_S$ is 
zero.
\end{proof}

\begin{lemma}[Adjoint reflection symmetry and positivity]
\label{lemma_W_symmetry}
Let $W(u)$ be the strictly accretive solution of~\eqref{eq:HMS}
for $u\in D=\{\Re u>0\}$. Then
\[
        W(\bar u)=W(u)^*.
\]
In particular, $W(t)\succ0$ for every $t>0$.
\end{lemma}
For a proof see Lemma 5.1 in \cite{kargin2026_algebraic}.

\begin{exa}[Cusp singularity and congruence invariance for $n=2$]
\label{exa:cusp_singularity}
Let $S = A_1 \otimes s_1 + A_2 \otimes s_2$ with
$A_1 = \bigl(\begin{smallmatrix} 0 & 1 \\ 1 & 0 \end{smallmatrix}\bigr)$ and
$A_2 = \bigl(\begin{smallmatrix} 0 & 0 \\ 0 & 1 \end{smallmatrix}\bigr)$.
This is the Type~I Lancaster--Rodman cell $S_2$ (up to reordering of generators);
it is the smallest cell whose density is singular at the origin, and we work it out
by hand to fix the normalizations used later.

For $u > 0$ the strictly positive solution $W(u)\succ0$ of Speicher's equation
$\eta(W)W + uW = I$ is diagonal, $W(u) = \diag(a(u),\, d(u))$
(Lemma~\ref{lemma_W_symmetry}). Since
$\eta\bigl(\diag(x_1,x_2)\bigr)=\diag(x_2,\,x_1+x_2)$, the entries satisfy
$a(d + u) = 1$ and $d(a + d + u) = 1$. Eliminating $a = 1/(d + u)$ gives
$d(d+u)^2=u$, i.e.\ the cubic
\begin{equation}\label{eq:cusp_cubic}
  d^3 + 2u\,d^2 + u^2\,d - u = 0 .
\end{equation}
For $u>0$ the left-hand side is negative at $d=0$ and strictly increasing in $d>0$
(its derivative $(3d+u)(d+u)$ is positive), so \eqref{eq:cusp_cubic} has a unique
positive root. Its Newton polygon has a single segment of slope $1/3$, so that root
satisfies $d(u)\sim u^{1/3}$ as $u\to0^+$. Writing $s=u^{1/3}$ and solving
\eqref{eq:cusp_cubic} order by order gives $d = s-\tfrac23 s^3+O(s^5)$, whence
\[
  W(u) = \begin{pmatrix}
    u^{-1/3} - \tfrac{1}{3}\,u^{1/3} + O(u) & 0 \\[3pt]
    0 & u^{1/3} - \tfrac{2}{3}\,u + O(u^{5/3})
  \end{pmatrix}.
\]
As $u\to0^+$ the $(1,1)$-entry diverges and the $(2,2)$-entry vanishes; in the
normalized trace ($n=2$),
\[
  \tr W(u)=\tfrac12\bigl(a(u)+d(u)\bigr)=\tfrac12\,u^{-1/3}+O(u^{1/3}).
\]

To pass to the density, recall $G_S(zI)=-\,iW(u)$ with $z=iu$, so that
\[
  f(x)=\tfrac1\pi\lim_{\varepsilon\downarrow0}
        \Im\bigl[-\tr G_S\bigl((x+i\varepsilon)I\bigr)\bigr]
      =\tfrac1\pi\,\Re\,\tr W(u),
\]
evaluated at $u=-ix \ (\Re u>0)$.
The positive branch of $u^{-1/3}$ is the principal one (real for $u>0$); at $u=-ix$
it equals $|x|^{-1/3}e^{\pm i\pi/6}$ for $x\gtrless0$, so
$\Re\,u^{-1/3}=\cos(\pi/6)\,|x|^{-1/3}$ on both sides. Therefore
\begin{equation}\label{eq:cusp_density}
  f(x)=\frac{\cos(\pi/6)}{2\pi}\,|x|^{-1/3}+o\bigl(|x|^{-1/3}\bigr)
      =\frac{\sqrt3}{4\pi}\,|x|^{-1/3}+o\bigl(|x|^{-1/3}\bigr),
  \qquad x\to0 .
\end{equation}
This is a cusp of exponent $\tfrac13=\frac{n-1}{n+1}\big|_{n=2}$, and the constant
$\frac{\sqrt3}{4\pi}=\frac{1}{n\pi}\sin\frac{\pi}{n+1}\big|_{n=2}$ is the $\alpha=0$
specialization of the sharp constant $C_{\alpha,n}$ of
Theorem~\ref{theo:typ_2_cells_sharp}.

The example also previews the behavior under congruence. The flip $c=A_1$ is a
Hermitian involution ($c=c^*=c^{-1}$), and the congruence $A_i\mapsto cA_ic$ sends
the pencil to $A_1'=A_1$,
$A_2'=\bigl(\begin{smallmatrix}1&0\\0&0\end{smallmatrix}\bigr)$. The transformed
Speicher equation is solved by $W'(u)=cW(u)c=\diag\bigl(d(u),a(u)\bigr)$: the
diagonal entries are swapped. Thus the entrywise Puiseux exponents ($-\tfrac13$ and
$+\tfrac13$) are permuted, while the trace exponent --- and hence the density
exponent $\tfrac13$ --- is unchanged. Here $c$ is unitary, so density invariance is
immediate; the substantive statement, that the exponent is preserved under
congruence by an \emph{arbitrary} positive $b$, is
Theorem~\ref{thm:covariance-scaling-singularity-exponent}.
\end{exa}

\section{Invariance under symmetric covariance-map scaling} 
\label{sec:invariance}

Define the \emph{Poisson indicator} of~$S$ at $x = 0$ by
\[
\mathcal P_S(\varepsilon)
\;:=\;
\frac{1}{\pi}\int_{\mathbb R}
  \frac{\varepsilon}{x^2+\varepsilon^2}\,d\mu_S(x)
\;=\;
-\frac{1}{\pi}\,\Im\, H(i\varepsilon),
\qquad \varepsilon > 0,
\]
where $H(z) = \tr\, G_S(zI)$ is the scalar Cauchy transform
and $\mu_S$ is the spectral measure of~$S$.


\begin{propo}[Poisson comparison under congruence]
\label{prop:poisson-comparison}
Let $S\in M_n(\mathbb C)\otimes\mathcal A$ be a bounded self-adjoint random variable,
let $b\in M_n(\mathbb C)$ be invertible, and set
\[
        T=(b\otimes I)\,S\,(b^*\otimes I),
\]
which is again self-adjoint. Then for every $\varepsilon>0$,
\[
\frac{\sigma_{\min}(b)}{\|b\|^3}\,
\mathcal P_S\!\Bigl(\frac{\varepsilon}{\|b\|\sigma_{\min}(b)}\Bigr)
\;\le\;
\mathcal P_T(\varepsilon)
\;\le\;
\frac{\|b\|}{\sigma_{\min}(b)^3}\,
\mathcal P_S\!\Bigl(\frac{\varepsilon}{\|b\|\sigma_{\min}(b)}\Bigr).
\]
\end{propo}

In particular, the blow-up exponent of the Poisson transform at $0$ is
preserved under the congruence $S\mapsto bSb^\ast$.


\begin{proof}
\emph{Reduction to a positive congruence.}
Write the polar decomposition $b=uP$, where $u$ is unitary and
$P:=(b^*b)^{1/2}\succ0$. Since $P^*=P$,
\[
        T=(b\otimes I)\,S\,(b^*\otimes I)
        =(u\otimes I)\,\bigl[(P\otimes I)\,S\,(P\otimes I)\bigr]\,(u^*\otimes I).
\]
The Poisson indicator is invariant under conjugation by the unitary
$u\otimes I$: writing $\tau=\tr\otimes\phi$ and $R:=(P\otimes I)S(P\otimes I)$,
\[
        \mathcal P_T(\varepsilon)
        =\frac{\varepsilon}{\pi}\,\tau\!\bigl((T^2+\varepsilon^2I)^{-1}\bigr)
        =\frac{\varepsilon}{\pi}\,\tau\!\bigl((u\otimes I)(R^2+\varepsilon^2I)^{-1}(u^*\otimes I)\bigr)
        =\mathcal P_R(\varepsilon),
\]
using $T^2=(u\otimes I)R^2(u^*\otimes I)$ and the trace identity
$\tau\bigl((u\otimes I)Y(u^*\otimes I)\bigr)=\tau(Y)$. Moreover
$\|P\|=\|b\|$ and $\sigma_{\min}(P)=\sigma_{\min}(b)$. It therefore suffices to
prove the estimate for the positive congruence $R=(P\otimes I)S(P\otimes I)$,
i.e.\ we may assume $b=P$ is positive. We do so from now on.

Let $M:=\|b\|$, $m:=\sigma_{\min}(b)$, where $b\in M_n(\mathbb C)$ is positive and invertible. Then
\[
mI\preceq b\preceq MI,
\qquad
m^2I\preceq b^2\preceq M^2I,
\qquad
M^{-2}I\preceq b^{-2}\preceq m^{-2}I.
\]

Write $\tau:=\tr\otimes\phi$ and 
$T=bSb$.
We compute
\[
T^2+\alpha^2 I
=
bSb^2Sb+\alpha^2 I
=
b\bigl(Sb^2S+\alpha^2 b^{-2}\bigr)b,
\]
and therefore
\begin{equation}
\label{eq:resolvent_comparison}
(T^2+\alpha^2 I)^{-1}
=
b^{-1}\bigl(Sb^2S+\alpha^2 b^{-2}\bigr)^{-1}b^{-1}.
\end{equation}

Since $S=S^\ast$, the map $X\mapsto SXS$ is CP. Hence from
$m^2I\preceq b^2\preceq M^2I$
we obtain
\[
m^2S^2\preceq Sb^2S\preceq M^2S^2.
\]
Also,
\[
\alpha^2 M^{-2}I\preceq \alpha^2 b^{-2}\preceq \alpha^2 m^{-2}I.
\]
Adding the two inequalities gives
\[
m^2S^2+\alpha^2 M^{-2}I
\;\preceq\;
Sb^2S+\alpha^2 b^{-2}
\;\preceq\;
M^2S^2+\alpha^2 m^{-2}I.
\]
Inverting (which reverses the inequalities) yields
\[
\bigl(M^2S^2+\alpha^2 m^{-2}I\bigr)^{-1}
\;\preceq\;
\bigl(Sb^2S+\alpha^2 b^{-2}\bigr)^{-1}
\;\preceq\;
\bigl(m^2S^2+\alpha^2 M^{-2}I\bigr)^{-1}.
\]

Now apply \(\tau\) to~\eqref{eq:resolvent_comparison}. If
\[
X:=\bigl(Sb^2S+\alpha^2 b^{-2}\bigr)^{-1}\succeq 0,
\]
then
\[
\tau\bigl((T^2+\alpha^2 I)^{-1}\bigr)
=
\tau(b^{-1}Xb^{-1})
=
\tau(b^{-2}X).
\]
Since
\[
M^{-2}I\preceq b^{-2}\preceq m^{-2}I
\]
and $X\succeq 0$, we have
\[
M^{-2}\tau(X)\le \tau(b^{-2}X)\le m^{-2}\tau(X).
\]
Combining this with the operator inequality above, we obtain
\[
\frac{\tau\!\left(\bigl(M^2S^2+\alpha^2 m^{-2}I\bigr)^{-1}\right)}{M^2}
\;\le\;
\tau\bigl((T^2+\alpha^2 I)^{-1}\bigr)
\;\le\;
\frac{\tau\!\left(\bigl(m^2S^2+\alpha^2 M^{-2}I\bigr)^{-1}\right)}{m^2}.
\]
Factoring out $M^2$ and $m^2$, this becomes
\begin{equation}
\label{eq:resolvent_final}
\frac{
\tau\!\left(\Bigl(S^2+\frac{\alpha^2}{m^2M^2}I\Bigr)^{-1}\right)}{M^4}
\;\le\;
\tau\bigl((T^2+\alpha^2 I)^{-1}\bigr)
\;\le\;
\frac{
\tau\!\left(\Bigl(S^2+\frac{\alpha^2}{m^2M^2}I\Bigr)^{-1}\right)}{m^4}.
\end{equation}
Since by definition:
\[
\mathcal P_T(\alpha)
:=
\frac{\alpha}{\pi}\,\tau\bigl((T^2+\alpha^2 I)^{-1}\bigr),
\qquad \alpha>0,
\]
then~\eqref{eq:resolvent_final} can be rewritten as
\begin{equation}
\label{eq:poisson_final}
\frac{m}{M^3}\,
\mathcal P_S\!\Bigl(\frac{\alpha}{mM}\Bigr)
\;\le\;
\mathcal P_T(\alpha)
\;\le\;
\frac{M}{m^3}\,
\mathcal P_S\!\Bigl(\frac{\alpha}{mM}\Bigr).
\end{equation}
\end{proof}

\begin{proposition}[Poisson comparison under covariance-map scaling]
\label{prop:poisson-comparison-covariance}
Let
\[
S_A=\sum_i A_i\otimes s_i,
\qquad
S_B=\sum_j B_j\otimes t_j
\]
be matrix semicircles, and let their covariance maps be
\[
\eta_A(X)=\sum_i A_iXA_i^*,
\qquad
\eta_B(X)=\sum_j B_jXB_j^* .
\]
Assume that there exists an invertible matrix $b\in M_n(\mathbb C)$ such that
\[
        \eta_B(X)=b\,\eta_A(b^*Xb)\,b^*,
        \qquad X\in M_n(\mathbb C).
\]
Then, for every $\varepsilon>0$,
\[
\frac{\sigma_{\min}(b)}{\|b\|^3}\,
\mathcal P_{S_A}\!\left(
\frac{\varepsilon}{\|b\|\sigma_{\min}(b)}
\right)
\le
\mathcal P_{S_B}(\varepsilon)
\le
\frac{\|b\|}{\sigma_{\min}(b)^3}\,
\mathcal P_{S_A}\!\left(
\frac{\varepsilon}{\|b\|\sigma_{\min}(b)}
\right).
\]
\end{proposition}

\begin{proof}
Let
\[
        \widetilde S=(b\otimes I)S_A(b^*\otimes I).
\]
Then $\widetilde S$ is again a matrix semicircle. Its covariance map is
\[
\eta_{\widetilde S}(X)
=
b\,\eta_A(b^*Xb)\,b^*
=
\eta_B(X).
\]
Since the scalar distribution of a centered matrix semicircle is determined by
its covariance map, we have
\[
        \mu_{\widetilde S}=\mu_{S_B},
        \qquad
        \mathcal P_{\widetilde S}(\varepsilon)
        =
        \mathcal P_{S_B}(\varepsilon).
\]
Applying Proposition~\ref{prop:poisson-comparison} to $S_A$ and
$\widetilde S$ gives
\[
\frac{\sigma_{\min}(b)}{\|b\|^3}\,
\mathcal P_{S_A}\!\left(
\frac{\varepsilon}{\|b\|\sigma_{\min}(b)}
\right)
\le
\mathcal P_{\widetilde S}(\varepsilon)
\le
\frac{\|b\|}{\sigma_{\min}(b)^3}\,
\mathcal P_{S_A}\!\left(
\frac{\varepsilon}{\|b\|\sigma_{\min}(b)}
\right).
\]
Replacing $\mathcal P_{\widetilde S}$ by $\mathcal P_{S_B}$ proves the claim.
\end{proof}

\begin{corollary}[Preservation of the Poisson blow-up exponent under covariance-map scaling]
\label{cor:poisson-exponent-covariance-scaling}
Assume the setting of Proposition \ref{prop:poisson-comparison-covariance}. If, for some $\beta\in[0,1)$, there exist constants
$c_1,c_2,\varepsilon_0>0$ such that
\[
        c_1\,\varepsilon^{-\beta}
        \le
        \mathcal P_{S_A}(\varepsilon)
        \le
        c_2\,\varepsilon^{-\beta},
        \qquad 0<\varepsilon<\varepsilon_0,
\]
then there exist constants $c_1',c_2',\varepsilon_0'>0$ such that
\[
        c_1'\,\varepsilon^{-\beta}
        \le
        \mathcal P_{S_B}(\varepsilon)
        \le
        c_2'\,\varepsilon^{-\beta},
        \qquad 0<\varepsilon<\varepsilon_0'.
\]
In particular, the Poisson blow-up exponent at $0$ is invariant under
symmetric scalability of covariance maps.
\end{corollary}

\begin{proof}
Let
\[
        M:=\|b\|,
        \qquad
        m:=\sigma_{\min}(b).
\]
By Proposition~\ref{prop:poisson-comparison-covariance}, for every
$\varepsilon>0$,
\[
\frac{m}{M^3}\,
\mathcal P_{S_A}\!\left(\frac{\varepsilon}{mM}\right)
\le
\mathcal P_{S_B}(\varepsilon)
\le
\frac{M}{m^3}\,
\mathcal P_{S_A}\!\left(\frac{\varepsilon}{mM}\right).
\]
Choose
\[
        \varepsilon_0' := mM\,\varepsilon_0 .
\]
Then for $0<\varepsilon<\varepsilon_0'$ we have
\[
        0<\frac{\varepsilon}{mM}<\varepsilon_0 .
\]
Hence the assumed bounds for $\mathcal P_{S_A}$ give
\[
c_1\left(\frac{\varepsilon}{mM}\right)^{-\beta}
\le
\mathcal P_{S_A}\!\left(\frac{\varepsilon}{mM}\right)
\le
c_2\left(\frac{\varepsilon}{mM}\right)^{-\beta}.
\]
Substituting these inequalities into the comparison estimate yields
\[
\frac{m}{M^3}c_1(mM)^\beta\,\varepsilon^{-\beta}
\le
\mathcal P_{S_B}(\varepsilon)
\le
\frac{M}{m^3}c_2(mM)^\beta\,\varepsilon^{-\beta}.
\]
Thus the conclusion holds with
\[
        c_1'=\frac{m}{M^3}c_1(mM)^\beta,
        \qquad
        c_2'=\frac{M}{m^3}c_2(mM)^\beta .
\]
\end{proof}


\begin{lemma}[Puiseux asymptotics implies Poisson asymptotics]
\label{lem:puiseux-to-poisson}
Let $f\ge 0$ be locally integrable near $0$, and assume that for some
$\alpha\in[0,1)$ and $\delta \in (0, 1 + \alpha)$,
\[
f(x)=c_+\,x^{-\alpha}+O(x^{-\alpha+\delta})
\quad\text{as }x\downarrow 0,
\]
and
\[
f(-x)=c_-\,x^{-\alpha}+O(x^{-\alpha+\delta})
\quad\text{as }x\downarrow 0,
\]
with $c_+,c_-\ge 0$ and $c_++c_->0$.
Then
\[
\frac{1}{\pi}\int_{\mathbb R}\frac{\varepsilon\,f(x)}{x^2+\varepsilon^2}\,dx
=
\frac{c_++c_-}{2}\,
\sec\!\Bigl(\frac{\pi\alpha}{2}\Bigr)\,
\varepsilon^{-\alpha}
+O(\varepsilon^{-\alpha+\delta}).
\]
\end{lemma}

Proof is in Appendix \ref{appx:puiseux_poisson}.

\begin{lemma}[Poisson asymptotics determines the leading Puiseux exponent]
\label{lem:poisson-to-puiseux}
Let $f$ be the density of the absolutely continuous part of a matrix semicircle.
Assume that $f$ has an algebraic singularity at $0$ and no atom at $0$.
Then there exist exponents
\[
-1<\beta_1<\beta_2<\cdots<\beta_N
\]
and coefficients $a_k^\pm\in\mathbb R$ such that
\[
f(x)=\sum_{k=1}^N a_k^+ x^{\beta_k}+o(x^{\beta_N})
\quad\text{as }x\downarrow 0,
\]
\[
f(-x)=\sum_{k=1}^N a_k^- x^{\beta_k}+o(x^{\beta_N})
\quad\text{as }x\downarrow 0.
\]
Let $\beta_*=\min\{\beta_k:\ a_k^+\neq 0\text{ or }a_k^-\neq 0\}$ and set
$\alpha_*=-\beta_*\in[0,1)$.
Then
\[
\mathcal P_X(\varepsilon)\asymp \varepsilon^{-\alpha_*}
\qquad (\varepsilon\downarrow 0).
\]
Conversely, if
\[
\mathcal P_X(\varepsilon)\asymp \varepsilon^{-\alpha}
\qquad (\varepsilon\downarrow 0),
\]
then $\alpha=\alpha_*$.
\end{lemma}

Proof is in Appendix \ref{appx:poisson_puiseux}


\begin{theorem}[Invariance of the leading singularity exponent under covariance-map scaling]
\label{thm:covariance-scaling-singularity-exponent}
Let
\[
        S_A=\sum_{i=1}^r A_i\otimes s_i,
        \qquad
        S_B=\sum_{j=1}^q B_j\otimes t_j
\]
be matrix semicircles, with covariance maps
\[
        \eta_A(X)=\sum_{i=1}^r A_iXA_i^*,
        \qquad
        \eta_B(X)=\sum_{j=1}^q B_jXB_j^* .
\]
Assume that $\eta_A$ and $\eta_B$ are symmetrically scalable.

Assume moreover that $S_A$ is full, so that neither $S_A$ nor $S_B$ has an atom
at $0$. Let $f_A$ and $f_B$ denote the scalar densities of $S_A$ and $S_B$.

Suppose that $f_A$ has an algebraic singularity at $0$ with leading exponent
$-\alpha$, where $\alpha\in[0,1)$. In particular, near $0$,
\[
        f_A(x)=c_{A,+}\,x^{-\alpha}+o(x^{-\alpha})
        \qquad (x\downarrow 0),
\]
and
\[
        f_A(-x)=c_{A,-}\,x^{-\alpha}+o(x^{-\alpha})
        \qquad (x\downarrow 0),
\]
with $c_{A,+},c_{A,-}\ge 0$ and $c_{A,+}+c_{A,-}>0$.

Then $f_B$ has the same leading exponent $-\alpha$ at $0$; that is,
\[
        f_B(x)=c_{B,+}\,x^{-\alpha}+o(x^{-\alpha})
        \qquad (x\downarrow 0),
\]
and
\[
        f_B(-x)=c_{B,-}\,x^{-\alpha}+o(x^{-\alpha})
        \qquad (x\downarrow 0),
\]
for some constants $c_{B,+},c_{B,-}\ge 0$, not both zero.

In particular, the leading singularity exponent at $0$ is invariant under
symmetric scalability of covariance maps.
\end{theorem}

\begin{proof}
By Proposition~\ref{prop:entrywise-algebraicity}, the scalar Cauchy transform of
$S_B$ is algebraic. Hence the density $f_B$ has the local Puiseux-type
expansion required in Lemma~\ref{lem:poisson-to-puiseux}. Since $S_A$ is full
and $\eta_B$ is obtained from $\eta_A$ by an invertible symmetric scaling,
$S_B$ has no atom at $0$.

By Lemma~\ref{lem:puiseux-to-poisson}, the assumed asymptotics of $f_A$ imply
that there exists $\delta>0$ such that
\[
        \mathcal P_{S_A}(\varepsilon)
        =
        \frac{c_{A,+}+c_{A,-}}{2}\,
        \sec\!\left(\frac{\pi\alpha}{2}\right)\,
        \varepsilon^{-\alpha}
        +
        O(\varepsilon^{-\alpha+\delta})
        \qquad (\varepsilon\downarrow 0).
\]
In particular,
\[
        \mathcal P_{S_A}(\varepsilon)\asymp \varepsilon^{-\alpha}
        \qquad (\varepsilon\downarrow 0).
\]
Therefore, by Corollary~\ref{cor:poisson-exponent-covariance-scaling},
\[
        \mathcal P_{S_B}(\varepsilon)\asymp \varepsilon^{-\alpha}
        \qquad (\varepsilon\downarrow 0).
\]

Now Lemma~\ref{lem:poisson-to-puiseux} applies to $S_B$. Let
$\alpha_*\in[0,1)$ denote the leading singularity exponent of $f_B$ in the
sense of that lemma. Since
\[
        \mathcal P_{S_B}(\varepsilon)\asymp \varepsilon^{-\alpha},
\]
the converse part of Lemma~\ref{lem:poisson-to-puiseux} gives
\[
        \alpha_*=\alpha.
\]
Hence the leading term in the Puiseux expansion of $f_B$ has exponent
$-\alpha$. Equivalently, there exist constants $c_{B,+},c_{B,-}\ge 0$, not
both zero, such that
\[
        f_B(x)=c_{B,+}\,x^{-\alpha}+o(x^{-\alpha})
        \qquad (x\downarrow 0),
\]
and
\[
        f_B(-x)=c_{B,-}\,x^{-\alpha}+o(x^{-\alpha})
        \qquad (x\downarrow 0).
\]
This proves the claim.
\end{proof}

Remark: This is Result A of the introduction.



\section{Hermitian binary pencils and reduction to canonical cells} 
\label{sec:direct_sums}

\subsection{Lancaster--Rodman canonical cells}
\label{sec:LaRo_cells}


Let $A = A_1 x_1 + A_2 x_2$ be the associated Hermitian matrix pencil.
According to \cite{lancaster_rodman2005}, every Hermitian binary pencil 
can be brought by a congruence transformation $A \mapsto C A C^*$ to a direct sum of 
indecomposable cells of the following three types.

\medskip
\noindent\textbf{Type~I} (eigenvalue at $\infty$)\textbf{:}\quad
$\delta\,(F_n\, x_1 + G_n\, x_2)$, \quad $\delta = \pm 1$, \  $n \ge 1$.

\smallskip
\noindent\textbf{Type~II} (real finite eigenvalue $\alpha$)\textbf{:}\quad
$\eta\,\bigl(F_n\,(x_2 + \alpha\, x_1) + G_n\, x_1\bigr)$, \quad $\eta = \pm 1$, 
\ $\alpha \in \R$, \  $n \ge 1$.

\smallskip
\noindent\textbf{Type~III} (conjugate pair of non-real eigenvalues $\beta, \ovln\beta$)\textbf{:}\quad
\[
\bmatr{0 & F_m(x_2 + \beta\, x_1) + G_m\, x_1 \\ 
F_m(x_2 + \ovln\beta\, x_1) + G_m\, x_1 & 0}, 
\qquad \beta \in \bC \setminus \R, \quad m \ge 1.
\]

\noindent
Here $F_m$ is the $m \times m$ \emph{sip} (standard involutory permutation) matrix 
with ones on the anti-diagonal,
\[
F_m = 
\bmatr{
0      & \cdots & 0 & 1 \\
\vdots &        & \iddots & 0 \\
0      & \iddots &        & \vdots \\
1      & 0      & \cdots & 0
} = F_m^{-1},
\]
and\ $G_m = \bigl[\begin{smallmatrix} F_{m-1} & 0 \\ 0 & 0 \end{smallmatrix}\bigr]$.
A Type~I cell has size~$n$, a Type~II cell has size~$n$, and a Type~III cell has size~$2m$.
The signs $\delta, \eta = \pm 1$ constitute the \emph{sign characteristic} of the pencil.
The canonical form is unique up to permutation of blocks and the replacement 
$\beta \leftrightarrow \ovln\beta$ within Type~III cells.

\begin{remark}\label{rem:LR_splittability}
A $1 \times 1$ cell of Type~I is simply $\delta\, x_1$ and a $1 \times 1$ cell of 
Type~II is $\eta\,(x_2 + \alpha\, x_1)$; both are scalar pencils giving rise to 
(rescaled) scalar semicircular distributions with no singularity.
For $n \ge 2$, Type~I and~II cells are splittable but not LR-semisimple, and the 
associated matrix semicircular element develops a singularity at the origin.
A Type~III cell of size $2$ (i.e., $m = 1$) is splittable and LR-semisimple; for $m \ge 2$ it is again splittable but not LR-semisimple.
\end{remark}

\begin{defi}[Regular and full pencils]
\label{defi:regular-full}
Let $A=A_1x_1+\dots+A_rx_r$ be a self-adjoint matrix pencil, $A_i=A_i^*\in M_N(\bC)$.
\begin{enumerate}[label=\textup{(\roman*)}]
\item $A$ is \emph{regular} if $\det(A_1x_1+\dots+A_rx_r)\not\equiv0$ in $\bC[x_1,\dots,x_r]$;
equivalently, some real combination $\sum_i t_iA_i$ ($t\in\bR^r$) is invertible.
\item $A$ is \emph{full} if its noncommutative inner rank is maximal, $\rank A=N$, where
$\rank A$ is the least $\rho$ admitting a factorization $A=PQ$ with
$P\in M_{N\times \rho}(\bC\langle x_1,\dots,x_r\rangle)$ and
$Q\in M_{\rho\times N}(\bC\langle x_1,\dots,x_r\rangle)$.
\end{enumerate}
\end{defi}

By \cite[Thm.~4.3]{hoffmann_mai_speicher2024}, for free standard semicircular generators the
inner rank is read off the atom of the scalar law,
$\rank A=N\bigl(1-\mu_{S_A}(\{0\})\bigr)$; thus $A$ is full iff $\mu_{S_A}(\{0\})=0$. This is the
sense of ``full'' used throughout (no atom at the origin).

\begin{propo}[For binary pencils, regularity $=$ fullness]
\label{prop:regular-full}
Let $A=A_1x_1+A_2x_2$ be a self-adjoint binary pencil with $A_i\in M_N(\bC)$. The following
are equivalent:
\textup{(i)} $A$ is regular;
\textup{(ii)} $A$ is full;
\textup{(iii)} $\mu_{S_A}(\{0\})=0$.
\end{propo}

\begin{proof}
The equivalence \textup{(ii)}$\iff$\textup{(iii)} is \cite[Thm.~4.3]{hoffmann_mai_speicher2024}.

\textup{(i)}$\Rightarrow$\textup{(ii)}, by contraposition. If $A$ is not full then
$\rank A=:r<N$, so $A=PQ$ with $P,Q$ matrices over $\bC\langle x_1,x_2\rangle$ of inner size
$r$. Specializing the formal variables to any scalars $(x_1,x_2)=(t_1,t_2)\in\bC^2$ gives
$A_1t_1+A_2t_2=P(t)Q(t)$, a product factoring through an $r$-dimensional space, so
$\rank(A_1t_1+A_2t_2)\le r<N$ and $\det(A_1t_1+A_2t_2)=0$. As this holds for all
$(t_1,t_2)$, the form $\det(A_1x_1+A_2x_2)$ vanishes identically: $A$ is singular.

\textup{(ii)}$\Rightarrow$\textup{(i)}, again by contraposition. Assume $A$ is singular. By
the Kronecker canonical form of a pencil under strict equivalence
\cite[Ch.~XII]{gantmacher59}, there are invertible constant matrices $P,Q\in\GL_N(\bC)$ such
that $PAQ$ is a direct sum of canonical blocks, among which --- $A$ being square and
singular --- there is at least one \emph{column} minimal-index block
\[
        L_\epsilon \;=\; x_1\,[\,I_\epsilon\ \ 0\,]\;+\;x_2\,[\,0\ \ I_\epsilon\,]
        \qquad(\epsilon\times(\epsilon+1),\ \epsilon\ge0).
\]
Let $W\subseteq\bC^N$ be the $(\epsilon+1)$-dimensional coordinate subspace spanned by the
columns of this block, and write $\wh A_i:=PA_iQ$ for the coefficients of $PAQ$. By
block-diagonality $\wh A_1W+\wh A_2W$ lies in the $\epsilon$-dimensional row space of the
block, so $\dim(\wh A_1W+\wh A_2W)\le\epsilon<\epsilon+1=\dim W$. Put $W':=QW$; since $P$ is
invertible,
\[
        \dim(A_1W'+A_2W')=\dim(\wh A_1W+\wh A_2W)<\dim W=\dim W',
\]
so $W'$ is a \emph{shrunk subspace} for the Kraus family $\{A_1,A_2\}$ --- equivalently the
completely positive map $\eta(X)=A_1XA_1+A_2XA_2$ strictly decreases the rank of the
projection onto $W'$. By the operator-scaling characterization of fullness
\cite[\S1 and App.]{ggow2020}, $A$ is not full.
\end{proof}

\begin{remark}[The equivalence is special to $r=2$]
\label{rem:regular-binary-only}
For $r\ge3$ fullness is strictly weaker than regularity --- this gap is exactly what makes
the noncommutative Edmonds problem nontrivial \cite{hoffmann_mai_speicher2024,ggow2020}. The
standard witness is the Hermitian dressing of the generic $3\times3$ skew form,
$A_1x_1+A_2x_2+A_3x_3=i\!\left(\begin{smallmatrix}0&x_1&x_2\\-x_1&0&x_3\\-x_2&-x_3&0
\end{smallmatrix}\right)$ with Hermitian $A_k$ (the $\pm i$ off the diagonal): it has
$\det\equiv0$, hence is singular, yet $\rank=3$, hence full. For $r=2$ the gap closes because
Kronecker's form confines the singularity to a column block. Accordingly
Proposition~\ref{prop:regular-full} is stated for binary pencils, which is the setting of
\S\S\ref{sec:type_I}--\ref{sec:classification-synthesis}.
\end{remark}

\begin{lemma}[Types~II and~III preserve no maximal abelian subalgebra]
\label{lem:no-MASA}
Let $S=\sum_i A_i\otimes s_i$ be a matrix semicircular element with Hermitian
coefficients $A_i\in M_N(\bC)$ and covariance map $\eta(X)=\sum_i A_iXA_i$, and
write $\mu_k:=(\id\otimes\phi)(S^k)$ for its matrix-valued moments.
\begin{enumerate}
\item The odd moments vanish, $\mu_1=\mu_3=0$, while $\mu_2=\eta(I)$ and
$\mu_4=\eta^2(I)+\eta(I)^2$; consequently
\begin{equation}\label{eq:moment-commutator}
[\mu_2,\mu_4]=[\eta(I),\eta^2(I)].
\end{equation}
\item For a Type~II cell $\bigl(A_1=\alpha F_n+G_n,\ A_2=F_n,\ \alpha\in\bR\setminus\{0\},\ n\ge2\bigr)$
and for a Type~III cell $\bigl(N=2m,\ m\ge2,\ \beta\in\bC\setminus\bR\bigr)$ one has
$[\eta(I),\eta^2(I)]\neq0$. Hence $\eta$ leaves no maximal abelian subalgebra of
$M_N(\bC)$ invariant.
\end{enumerate}
\end{lemma}

Proof is in Appendix \ref{sec:no-MASA}. 

\begin{remark}[Variance-profile reading of Lemma~\ref{lem:no-MASA}]
\label{rem:no-MASA-meaning}
The scalar law of $S$ is the law of a variance profile in some orthonormal basis
exactly when $\eta$ leaves a maximal abelian subalgebra $\mathcal A$ invariant: in
a basis diagonalizing $\mathcal A$ the completely positive map $\eta$ restricts to
a nonnegative profile $S=(s_{ij})$ on the diagonal, whose entries solve the
associated vector Dyson equation (cf.\ Remark~\ref{rem:typeI-KR} and
\cite{MR4901646}). Lemma~\ref{lem:no-MASA} shows this fails for Types~II and~III:
their scalar laws are variance-profile laws in \emph{no} basis. In particular the
diagonal reduction that places Type~I inside the Kr\"uger--Renfrew framework
(Remark~\ref{rem:typeI-KR}) has no Type~II/III analogue, which is why it is the
leading-order analysis, rather than a diagonal compression, that transfers to
these cells.
\end{remark}

\subsection{Direct sums and dominant singularities}
\begin{propo}
\label{prop:direct_sums}
Let $A \in M_{n_1}(\bC)\otimes \AA$ and $B \in M_{n_2}(\bC)\otimes \AA$ be self-adjoint, and let
\[
C=A\oplus B=
\begin{pmatrix}
A&0\\
0&B
\end{pmatrix}
\in M_{n_1+n_2}(\bC)\otimes \AA.
\]
Let $\mu_A,\mu_B,\mu_C$ be their scalar spectral measures (with respect to the normalized matrix traces). Then
\[
\mu_C
=
\frac{n_1}{n_1+n_2}\mu_A
+
\frac{n_2}{n_1+n_2}\mu_B.
\]
In particular, if $\mu_A$ and $\mu_B$ have densities $f_A$ and $f_B$, then $\mu_C$ has density
\[
f_C(x)=\frac{n_1}{n_1+n_2}f_A(x)+\frac{n_2}{n_1+n_2}f_B(x)
\]
for almost every $x$.
\end{propo}

\begin{proof}
For $z\in\C^+$,
\[
(zI-C)^{-1}
=
\begin{pmatrix}
(zI-A)^{-1}&0\\
0&(zI-B)^{-1}
\end{pmatrix}.
\]
Applying the conditional expectation $id\otimes \phi$ gives
\[
G_C(z)=G_A(z)\oplus G_B(z).
\]
Now take the normalized trace on $M_{n_1+n_2}(\bC)$:
\[
H_C(z)
=
\tr_{n_1+n_2}(G_C(z))
=
\frac{1}{n_1+n_2}\Big(\Tr(G_A(z))+\Tr(G_B(z))\Big).
\]
Since
\[
H_A(z)=\tr_{n_1}(G_A(z))=\frac{1}{n_1}\Tr(G_A(z)),
\qquad
H_B(z)=\tr_{n_2}(G_B(z))=\frac{1}{n_2}\Tr(G_B(z)),
\]
we obtain
\[
H_C(z)
=
\frac{n_1}{n_1+n_2}H_A(z)
+
\frac{n_2}{n_1+n_2}H_B(z).
\]
Therefore the scalar Cauchy transform of $\mu_C$ is the same convex combination of the scalar Cauchy transforms of $\mu_A$ and $\mu_B$, which implies
\[
\mu_C
=
\frac{n_1}{n_1+n_2}\mu_A
+
\frac{n_2}{n_1+n_2}\mu_B.
\]
If $\mu_A$ and $\mu_B$ are absolutely continuous with densities $f_A$ and $f_B$, then the same is true for $\mu_C$, with density
\[
f_C(x)=\frac{n_1}{n_1+n_2}f_A(x)+\frac{n_2}{n_1+n_2}f_B(x).
\]
\end{proof}

\subsection{Reduction of the classification problem}
\label{sec:classification_reduction}

\begin{propo}\label{prop:reduction-to-cells}
Let $A=A_1x_1+A_2x_2$ be a Hermitian binary pencil whose matrix semicircular
element $S_A=A_1\otimes s_1+A_2\otimes s_2\in M_n(\bC)\otimes\AA$ is full, and let
$C\in M_n(\bC)$ be invertible with
\[
  C A C^{*}=B=\bigoplus_{i=1}^{k}B^{(i)}
\]
the Lancaster--Rodman canonical form, $B^{(i)}$ an indecomposable cell of size $n_i$,
$\sum_i n_i=n$. For each cell let $\alpha_i\in[0,1)$ be the order of the singularity of
$\mu_{S_{B^{(i)}}}$ at $0$, i.e.\ its leading Puiseux exponent is $-\alpha_i$. Then
$\mu_{S_A}$ has no atom at $0$, and its leading Puiseux exponent there equals
\[
  -\alpha_{*},\qquad \alpha_{*}:=\max_{1\le i\le k}\alpha_i .
\]
Thus the singularity of $\mu_{S_A}$ at $0$ is governed by the most singular cell --- the
one of largest $\alpha_i$, equivalently smallest (most negative) leading exponent
$-\alpha_i$ --- and $\alpha_i$ depends only on the type and size of $B^{(i)}$, not on its
sign characteristic. Consequently, classifying the singularity at $0$ of an arbitrary
Hermitian binary pencil reduces to computing $\alpha_i$ for each indecomposable canonical cell.
\end{propo}

\begin{proof}
Write $S_B=\sum_j B_j\otimes s_j$ for the matrix semicircular element of $B$; since
$B=CAC^{*}$ we have $S_B=(C\otimes I)\,S_A\,(C^{*}\otimes I)$.

\emph{Step 1 (congruence invariance of the exponent at $0$).}
Polar-decompose $C=Up$, with $U$ unitary and $p=(C^{*}C)^{1/2}\succ0$, and set
$S_p=(p\otimes I)\,S_A\,(p\otimes I)$. Then $S_B=(U\otimes I)\,S_p\,(U^{*}\otimes I)$, so
$S_B$ and $S_p$ are unitarily conjugate; hence $\mu_{S_B}=\mu_{S_p}$, with identical local
behaviour at every point. The covariance map of $S_p$ is
\[
  \eta_{S_p}(X)=\sum_i (pA_ip)\,X\,(pA_ip)=p\,\eta_A(pXp)\,p,
\]
the symmetric scaling of $\eta_A$ by the positive invertible $p$. As $S_A$ is full and is a
matrix semicircle (so $f_A$ has algebraic singularity with a leading exponent in $[0,1)$ by
Lemma~\ref{lem:poisson-to-puiseux}), Theorem~\ref{thm:covariance-scaling-singularity-exponent}
applies and gives that $\mu_{S_p}$ has the same leading exponent at $0$ as $\mu_{S_A}$, and no
atom there. Therefore $\mu_{S_B}$ and $\mu_{S_A}$ share their leading Puiseux exponent at $0$,
and $\mu_{S_B}$ has no atom at $0$.

\emph{Step 2 (the direct sum is governed by its most singular cell).}
Iterating Proposition~\ref{prop:direct_sums} over the $k$ blocks,
\[
  \mu_{S_B}=\sum_{i=1}^{k}\frac{n_i}{n}\,\mu_{S_{B^{(i)}}},\qquad \frac{n_i}{n}>0 .
\]
Evaluating on $\{0\}$ gives $0=\mu_{S_B}(\{0\})=\sum_i\frac{n_i}{n}\,\mu_{S_{B^{(i)}}}(\{0\})$, a
sum of nonnegative terms; hence each cell is itself atom-free at $0$, so each $\alpha_i\in[0,1)$
is well defined by Lemma~\ref{lem:poisson-to-puiseux}. Writing $f_i$ for the density of
$\mu_{S_{B^{(i)}}}$,
\[
  f_i(x)=c_{i,+}\,x^{-\alpha_i}+o(x^{-\alpha_i}),\quad
  f_i(-x)=c_{i,-}\,x^{-\alpha_i}+o(x^{-\alpha_i})\qquad(x\downarrow0),
\]
with $c_{i,\pm}\ge0$ and $c_{i,+}+c_{i,-}>0$, the densities add: $f_B=\sum_i\frac{n_i}{n}f_i$.
Set $\alpha_*=\max_i\alpha_i$ and $I_*=\{i:\alpha_i=\alpha_*\}$. The cells with $\alpha_i<\alpha_*$
contribute $o(x^{-\alpha_*})$, whence
\[
  f_B(\pm x)=\Bigl(\sum_{i\in I_*}\frac{n_i}{n}\,c_{i,\pm}\Bigr)x^{-\alpha_*}+o(x^{-\alpha_*})
  \qquad(x\downarrow0).
\]
Each leading coefficient is a sum of nonnegative terms with strictly positive weights, so no
cancellation can occur; moreover $\sum_{i\in I_*}\frac{n_i}{n}(c_{i,+}+c_{i,-})>0$, so the two
coefficients are not both zero. Hence the leading Puiseux exponent of $\mu_{S_B}$ at $0$ is
exactly $-\alpha_*$. Combined with Step~1, $\mu_{S_A}$ has leading exponent $-\alpha_*$, as claimed.

Finally, the sign $\delta$ (Type~I) or $\eta$ (Type~II) multiplies the whole cell, replacing
$S_{B^{(i)}}$ by $-S_{B^{(i)}}$; since $(s_j)$ is a symmetric family, $-S_{B^{(i)}}$ has the same
distribution, so $\alpha_i$ is independent of the sign characteristic.
\end{proof}
 

\Part{Singularities of canonical binary cells}

\section{Type I cells: the chain model} 
\label{sec:type_I}

We start to analyze the binary matrix semicircles and we start with the cells that have Type I in Lancaster-Rodman classification: $A_n = F_n x_1 + G_n x_2$. 

\begin{theo}[Type~I cells — explicit density at the origin]
\label{theo:typ_1_cells_sharp}
For the matrix semicircle $S_n = F_n\otimes s_1 + G_n\otimes s_2$,
the scalar density $f$ satisfies
\[
f(x)\;=\;\frac{\sin\!\left(\pi/(n+1)\right)}{\pi n}\,
         |x|^{-(n-1)/(n+1)}
   + o\!\left(|x|^{-(n-1)/(n+1)}\right)
   \quad (|x|\downarrow 0).
\]
\end{theo}

We will give a proof of Theorem \ref{theo:typ_1_cells_sharp} in Section \ref{sec:proof_thm_case_I}. We start by analyzing the matrix Cauchy transform of $S_n$ in more detail.  

\subsection{Diagonal reduction and the chain equations} 
\label{sec:diagonal_reduction}

Let $\eta$ denote the covariance map of $S_n$. Note that for a diagonal matrix $D = \diag(d_1, \ldots, d_n)$, we have
\begin{align*}
F_n D F_n &= \diag(d_n, \ldots, d_1), \\
G_n D G_n &=  \diag(d_{n-1}, d_{n - 2}, \ldots, d_1, 0).
\end{align*}
Hence $\eta$ preserves the sub-algebra of diagonal matrices, 
\[
\eta(D) = \diag(d_n + d_{n - 1}, d_{n-1} + d_{n - 2}, \ldots, d_2 + d_1, d_1).
\]
Since we can start the iterative solution from the diagonal initial $W_0(u)$, it follows that the accretive solution of Speicher's equation \eqref{eq:HMS} is diagonal. 

Let the solution $W(u) = \diag\big(w_1(u), w_2(u), \ldots, w_n(u)\big)$. Then we have the system of equations: 
\begin{align*}
\Big(w_n(u) + w_{n - 1}(u) + u\Big) w_1(u) = 1, \\
\Big(w_{n-1}(u) + w_{n - 2}(u) + u\Big) w_2(u) = 1, \\
\ldots  \\
\Big(w_{2}(u) + w_{1}(u) + u\Big) w_{n - 1}(u) = 1, \\
\Big(w_1(u) + u\Big) w_n(u) = 1.
\end{align*} 

Let $y_1 = w_n$, $y_2 = w_1$, $y_3 = w_{n - 1}$, $y_4 = w_2$, \ldots. Formally, 
set
\[
y_{2m-1}(u)=w_{n+1-m}(u),\qquad
y_{2m}(u)=w_m(u),
\]
whenever these indices are defined. Then
\[
(y_1,y_2,y_3,y_4,\dots)=(w_n,w_1,w_{n-1},w_2,\dots),
\]
and the system for the functions $y_k (u)$'s is
\begin{align}
\label{eq:chain}
y_1\big(u+y_2\big)&=1,\\
\notag
&\ldots \\
\notag
y_k\big(u+y_{k-1}+y_{k+1}\big)&=1,\qquad 2\le k\le n-1,\\
\notag
&\ldots \\
\notag
y_n\big(u+y_{n-1}+y_n\big)&=1.
\end{align}
Observe that the right boundary equation involves $y_n$ quadratically, reflecting the self-coupling at the midpoint of the chain.

We are interested in the behavior of
each $y_k(u)$ as $u \downarrow 0$. 


\subsubsection*{Reduction to a two-step recurrence}

Adopt the conventions $y_0 := 0$ and $y_{n+1} := y_n$. Dividing each equation
in~\eqref{eq:chain} by the corresponding $y_k$ and solving for the successor
gives the uniform rewrite
\begin{equation}
\label{eq:rec}
y_{k+1} \;=\; \frac{1}{y_k} - u - y_{k-1}, \qquad k=1,\dots,n,
\end{equation}
subject to the boundary conditions
\begin{equation}
\label{eq:bc}
y_0 = 0, \qquad y_{n+1} = y_n.
\end{equation}
Indeed, for $k=1$ \eqref{eq:rec} reduces to $y_2 = 1/y_1 - u$, which is the
first equation of~\eqref{eq:chain}; for $2 \le k \le n-1$ it is the generic
equation; and for $k=n$, the condition $y_{n+1} = y_n$ is exactly
$y_n = 1/y_n - u - y_{n-1}$, i.e.\ the last equation of~\eqref{eq:chain}.

Equation~\eqref{eq:rec} is the autonomous discrete Painlev\'e\,I map (also
known as the McMillan map),
\[
T:(x,y)\longmapsto\bigl(y,\,\tfrac{1}{y}-u-x\bigr),
\]
with $u$ playing the role of a parameter. The system~\eqref{eq:chain} is thus
the orbit of $T$ starting from $(y_0,y_1) = (0,y_1)$ and terminating on the
diagonal $\{y_{n+1}=y_n\}$.

\begin{remark}[Type~I as a variance profile]
\label{rem:typeI-KR}
The diagonal reduction places Type~I inside the variance-profile framework of~\cite{MR4901646}. The compression of $\eta$ to the diagonal
subalgebra is the symmetric nonnegative profile $S=(s_{ij})$ with
$s_{ij}=|(F_n)_{ij}|^2+|(G_n)_{ij}|^2=\mathbf 1[i+j=n+1]+\mathbf 1[i+j=n]$,
supported on two adjacent antidiagonals, and the diagonal entries of $G_S$ solve
the associated vector Dyson equation. This $S$ has support but not total support, so by~\cite[Prop.~2.1]{MR4901646}
the scalar density blows up at the origin, and by~\cite[Thm.~2.8]{MR4901646}
the \emph{singularity degree} equals $\sigma=\ell/(\ell+2)$, with $\ell$ the
longest increasing chain in the zero pattern of $S$; here $\ell=n-1$, giving
$\sigma=(n-1)/(n+1)$ as in Theorem~\ref{theo:typ_1_cells_sharp}. Thus the exponent
is not new, and the convergent fractional-power expansion of the $y_k$ may be
taken from~\cite[Prop.~4.1]{MR4901646}. 
The same exponent was obtained independently and contemporaneously by~\cite{kolupaiev2021}, whose Assumption~1.1 ($s_{ij}>0$ for
$i+j\in\{n,n+1\}$ and $s_{ij}=0$ for $i+j\ge n+2$) is exactly this critical
antidiagonal staircase: by a direct asymptotic argument he derives
$\rho(E)\sim|E|^{-(n-1)/(n+1)}$ together with the componentwise asymptotics
$m_k(z)\,z^{-(1-\frac{2k}{n+1})}\to c_k\,e^{\,i\pi k/(n+1)}$. As in~\cite{MR4901646},
the leading amplitudes $c_k>0$ there are pinned only implicitly, as the unique
solution of a log-linear system, so neither reference supplies the sharp constant
of Theorem~\ref{theo:typ_1_cells_sharp}.

What the chain~\eqref{eq:chain} adds is
its integrable structure---it is an orbit of the autonomous discrete
Painlev\'e~I (McMillan) map (\S\ref{sec:motion_integral})---which we use to
obtain the \emph{sharp leading constant} in closed form: the proof of
\cite[Thm.~2.8]{MR4901646} reduces this constant to a positive amplitude that it
leaves undetermined, and the McMillan structure evaluates that amplitude, giving
the value in Theorem~\ref{theo:typ_1_cells_sharp}.
The reduction itself is special to Type~I: Types~II and~III preserve no maximal
abelian subalgebra, and it is the leading-order method, not this remark, that
transfers to them.
\end{remark}


\subsection{Integral of motion and endpoint identity} 
\label{sec:motion_integral}

\begin{propo}[QRT/McMillan invariant]\label{prop:invariant}
The function
\[
H(x,y) \;=\; (xy-1)(x+u)(y+u) + u^2
\]
is invariant under $T$: $H(T(x,y)) = H(x,y)$.
\end{propo}

\begin{proof}
Write $T(x,y) = (y,z)$, $z = 1/y - u - x$. The defining relation
$y(u+x+z)=1$ gives $yz - 1 = -y(x+u)$ and $y(z+u) = 1 - xy = -(xy-1)$. Hence
\begin{align*}
(yz-1)(y+u)(z+u)
&= \bigl[-y(x+u)\bigr](y+u)(z+u) 
\\
&= -(x+u)(y+u)\cdot y(z+u)
= (x+u)(y+u)(xy-1),
\end{align*}
which is symmetric in the roles of the two pairs, so $H(y,z) = H(x,y)$.
\end{proof}

The value of $H$ on our orbit is pinned down by the Dirichlet boundary
condition:
\[
H(y_0,y_1) \;=\; H(0,y_1) \;=\; (-1)\cdot u \cdot (y_1+u) + u^2
\;=\; -u\,y_1.
\]
Consequently,
\begin{equation}
\label{eq:orbit-invariant}
(y_{k-1}y_k - 1)(y_{k-1}+u)(y_k+u) \;=\; -u(y_1+u),
\qquad k=1,\dots,n.
\end{equation}


\subsubsection*{The endpoint identity} 

\begin{propo}[Endpoint identity]\label{prop:endpoint}
Any solution $(y_1,\dots,y_n)$ of~\eqref{eq:chain} satisfies
\begin{equation}
\label{eq:endpoint}
(y_n^2-1)(y_n+u)^2 \;=\; -u(y_1+u).
\end{equation}
\end{propo}

\begin{proof}
Apply Proposition~\ref{prop:invariant} to the pair $(y_n,y_{n+1}) = (y_n,y_n)$.
Since $H$ is constant along the orbit,
\[
H(y_n,y_n) \;=\; H(y_0,y_1) \;=\; -u\,y_1,
\]
which expands to $(y_n^2-1)(y_n+u)^2 + u^2 = -u y_1$, i.e.\
\eqref{eq:endpoint}.
\end{proof}

Equation \eqref{eq:endpoint} relates the two extreme entries of the chain by a single
polynomial identity, independent of~$n$. Together with
\eqref{eq:orbit-invariant} it provides a very efficient replacement for
iterating~\eqref{eq:rec}: every consecutive pair satisfies the same
algebraic equation, whose right-hand side tends to~$0$ as~$u\downarrow 0$.


\subsection{Puiseux exponents}
\label{sec:puiseux_exponents}

\begin{theo}[Puiseux exponents]
\label{thm:puiseux}
Assume each $y_k(u)$ admits a leading Puiseux behavior
\begin{equation}
\label{eq:ansatz}
y_k(u) \;=\; c_k\,u^{\beta_k}\bigl(1+o(1)\bigr),\qquad u\downarrow 0,
\end{equation}
with nonzero $c_k\in\bC$ and $\beta_k\in\bQ$ with $|\beta_k| < 1$.
Then the exponents are uniquely determined, and
\begin{equation}
\label{eq:exponents}
\boxed{\;\beta_k \;=\; (-1)^k\,\frac{n+1 - 2\lceil k/2\rceil}{n+1},
\qquad k=1,\dots,n.\;}
\end{equation}
Equivalently, writing $k=2m-1$ and $k=2m$,
\[
\beta_{2m-1} \;=\; -\frac{\,n-2m+1\,}{n+1},
\qquad
\beta_{2m} \;=\; \phantom{-}\frac{\,n-2m+1\,}{n+1}.
\]
Moreover, the leading coefficients satisfy
\begin{equation}
\label{eq:coeffs}
c_{2m-1} = c_1^{\,m},\qquad c_{2m} = c_1^{-m},\qquad c_1^{\,n+1}=1.
\end{equation}
\end{theo}

\begin{coro}
On the positive accretive branch, $c_1 = 1$ and \eqref{eq:coeffs} gives
$c_k = 1$ for all~$k$. Setting $m_k(t) := t^{-p_k}y_k(t^{n+1})$ with
$p_k = (n+1)\beta_k$, the rescaled chain entries therefore satisfy
$m_k(t) \to 1$ as $t \downarrow 0$.
\end{coro}

\begin{proof}[Proof of Theorem \ref{thm:puiseux}]
All asymptotic relations are as $u\downarrow 0$.
The proof uses two tools in alternation: the orbit
invariant~\eqref{eq:orbit-invariant}, which constrains
consecutive pairs, and the recurrence~\eqref{eq:rec}, which
propagates exponents forward. Together they produce a
two-step induction whose key mechanism is an exact
cancellation that forces the exponents to alternate in sign
and grow in a rigid arithmetic pattern.


\smallskip
\noindent\textbf{Step~1: Trichotomy from the invariant.}
Since $|\beta_k|<1$, we have $y_k + u \sim c_k\,u^{\beta_k}$ for
each~$k$.
Let $E(u) := -u(y_1+u) \sim -c_1\,u^{1+\beta_1}$. Applying
\eqref{eq:orbit-invariant} to the pair $(y_{k-1},y_k)$ gives
\begin{equation}
\label{eq:pair-leading}
(y_{k-1}y_k - 1)\cdot c_{k-1}c_k\,u^{\sigma_k}
\;\sim\; E(u),
\qquad \sigma_k := \beta_{k-1}+\beta_k.
\end{equation}
Since $E(u)\to 0$, exactly one of three cases holds for
each~$k$:
\begin{enumerate}[label=\textup{(\Alph*)}]
\item $\sigma_k = 0$: then $y_{k-1}y_k \to c_{k-1}c_k$,
and the left side of~\eqref{eq:pair-leading} tends to a
nonzero constant unless $c_{k-1}c_k = 1$. Matching with
$E(u)\to 0$ forces
\begin{equation}
\label{eq:caseA}
c_{k-1}c_k = 1.
\end{equation}

\item $\sigma_k < 0$: then $y_{k-1}y_k\to\infty$ and the
left side has order $u^{2\sigma_k}$. Matching exponents gives
$2\sigma_k = 1+\beta_1 > 0$, a contradiction. \emph{This
case never occurs.}

\item $\sigma_k > 0$: then $y_{k-1}y_k\to 0$ and
$y_{k-1}y_k - 1 \to -1$. Matching both sides,
\begin{equation}
\label{eq:caseC}
\beta_{k-1}+\beta_k = 1+\beta_1,
\qquad c_{k-1}c_k = c_1.
\end{equation}
\end{enumerate}


\smallskip
\noindent\textbf{Step~2: Base case and parity pattern.}
From $y_2 = 1/y_1 - u$ and $|\beta_1|<1$ we read off
$\beta_2 = -\beta_1$ and $c_2 = c_1^{-1}$, so $\sigma_2 = 0$
(Case~A). We claim the following pattern persists:
\begin{equation}
\label{eq:parity_0}
\sigma_{2m} = 0 \;\;(\text{Case A}),
\qquad
\sigma_{2m+1} = 1+\beta_1 \;\;(\text{Case C}),
\end{equation}
for all admissible~$m$. The inductive step has two halves: 
given Case~A at pair $(y_{2m-1},y_{2m})$, we must show Case~C 
at pair $(y_{2m},y_{2m+1})$; then given Case~C there, we must 
show Case~A at pair $(y_{2m+1},y_{2m+2})$.


\smallskip
\noindent\textbf{Step~3: Case~A $\Rightarrow$ Case~C
(cancellation).}
This is the heart of the argument. Assume Case~A holds at
$(y_{2m-1}, y_{2m})$, so $\beta_{2m} = -\beta_{2m-1}$ and
$c_{2m} = c_{2m-1}^{-1}$. Suppose for contradiction that
Case~A also holds at the next pair, giving
\begin{equation}
\label{eq:caseA-hypothetical_0}
\beta_{2m+1} = -\beta_{2m} = \beta_{2m-1},
\qquad
c_{2m+1} = c_{2m}^{-1} = c_{2m-1}.
\end{equation}
Consider the recurrence at $k = 2m$:
\begin{equation}
\label{eq:rec-2m}
y_{2m+1} \;=\; \underbrace{\frac{1}{y_{2m}}}%
_{\sim\, c_{2m-1}\,u^{\beta_{2m-1}}}
\;-\; u
\;-\; \underbrace{y_{2m-1}}_{\sim\, c_{2m-1}\,u^{\beta_{2m-1}}}.
\end{equation}
The two dominant terms share the same exponent
$\beta_{2m-1}$ and the same coefficient $c_{2m-1}$, so they
cancel:
\begin{equation}
\label{eq:cancellation}
\frac{1}{y_{2m}} - y_{2m-1}
\;=\; c_{2m-1}\,u^{\beta_{2m-1}}
  \bigl[(1+\varrho_2)^{-1} - (1+\varrho_1)\bigr]
\;=\; o\!\left(u^{\beta_{2m-1}}\right),
\end{equation}
where $\varrho_i(u) = o(1)$ are the sub-leading corrections. 
Since $\beta_{2m-1} < 1$, the surviving $-u$ term is even 
smaller, so $y_{2m+1}(u) = o(u^{\beta_{2m-1}})$. But 
\eqref{eq:caseA-hypothetical_0} predicts 
$y_{2m+1} \sim c_{2m-1}\,u^{\beta_{2m-1}}$ with 
$c_{2m-1}\neq 0$ --- contradiction. Hence Case~A 
fails and Case~C must hold at $(y_{2m}, y_{2m+1})$:
\begin{equation}
\label{eq:beta-2m+1-caseC}
\beta_{2m+1} = \beta_{2m-1} + (1+\beta_1),
\qquad c_{2m+1} = c_1\,c_{2m-1}.
\end{equation}


\smallskip
\noindent\textbf{Step~4: Case~C $\Rightarrow$ Case~A
(dominance).}
Now assume Case~C at $(y_{2m},y_{2m+1})$. The recurrence at
$k = 2m+1$ reads
\begin{equation}
\label{eq:rec-2m+1}
y_{2m+2} \;=\; \frac{1}{y_{2m+1}} - u - y_{2m}.
\end{equation}
The three terms on the right have leading exponents
$-\beta_{2m+1}$, $1$, and $\beta_{2m}$ respectively. Using
$\beta_{2m} + \beta_{2m+1} = 1+\beta_1$ and $|\beta_1|<1$,
both differences
\begin{align}
\label{eq:cmp-beta2m}
\beta_{2m} - (-\beta_{2m+1})
  &= \beta_{2m}+\beta_{2m+1} = 1+\beta_1 > 0,\\
\label{eq:cmp-one}
1 - (-\beta_{2m+1})
  &= 1+\beta_{2m+1}
  = 2 + \beta_1 + \beta_{2m-1} > 0,
\end{align}
are strictly positive (using $|\beta_1|, |\beta_{2m-1}| < 1$
in the second). So $1/y_{2m+1}$ is the unique dominant term,
giving
\begin{equation}
\label{eq:beta-2m+2}
\boxed{\;\beta_{2m+2} = -\beta_{2m+1},\qquad
c_{2m+2} = c_{2m+1}^{-1},\;}
\end{equation}
which is $\sigma_{2m+2} = 0$, i.e.\ Case~A. This closes the
induction and establishes~\eqref{eq:parity_0}.


\smallskip
\noindent\textbf{Step~5: Determining $\beta_1$ from the
endpoint.}
The parity pattern gives the recursion
$\beta_{2m} = -\beta_{2m-1}$ and
$\beta_{2m+1} = \beta_{2m-1} + (1+\beta_1)$, with solution
\begin{equation}
\label{eq:beta-linear_0}
\beta_{2m-1} = \beta_1 + (m-1)(1+\beta_1),
\qquad
\beta_{2m} = -\beta_{2m-1}.
\end{equation}

\emph{Case $n = 2M$ even.} Then $\beta_n = \beta_{2M} > 0$ 
and the endpoint identity~\eqref{eq:endpoint} gives
$(c_n^2 u^{2\beta_n} - 1)(c_n u^{\beta_n}+u)^2 \sim
-c_1\,u^{1+\beta_1}$. Since $\beta_n < 1$, the leading 
balance is $-c_n^2\,u^{2\beta_n} \sim -c_1\,u^{1+\beta_1}$, 
i.e.\ $2\beta_n = 1+\beta_1$ and $c_n^2 = c_1$. 
Substituting~\eqref{eq:beta-linear_0}:
\[
-2\beta_1 - 2(M-1)(1+\beta_1) = 1+\beta_1
\quad\Longleftrightarrow\quad
\beta_1 = -\frac{2M-1}{2M+1} = -\frac{n-1}{n+1}.
\]

\emph{Case $n = 2M-1$ odd.} The pair $(y_{n-1},y_n)$ is
Case~C, so $\beta_{n-1}+\beta_n = 1+\beta_1 > 0$, and the
last chain equation gives
\begin{equation}
\label{eq:last-odd}
y_n^2 = 1 - u\,y_n - y_{n-1}y_n = 1 + o(1),
\end{equation}
since $u\,y_n = O(u^{1+\beta_n}) = o(1)$ and
$y_{n-1}y_n = O(u^{1+\beta_1}) = o(1)$. Hence $\beta_n = 0$
and $c_n^2 = 1$. Substituting
$\beta_n = \beta_{2M-1} = \beta_1 + (M-1)(1+\beta_1)$:
\[
\beta_1 + (M-1)(1+\beta_1) = 0
\quad\Longleftrightarrow\quad
\boxed{\;\beta_1 = -\frac{M-1}{M} = -\frac{n-1}{n+1}.\;}
\]

In both cases $\beta_1 = -(n-1)/(n+1)$,
and~\eqref{eq:beta-linear_0} yields~\eqref{eq:exponents}.


\smallskip
\noindent\textbf{Step~6: Coefficients.}
The relations $c_{2m-1}c_{2m}=1$ (Case~A) and
$c_{2m}c_{2m+1}=c_1$ (Case~C) give inductively
$c_{2m-1} = c_1^m$ and $c_{2m} = c_1^{-m}$. The endpoint
constraint ($c_n^2 = c_1$ for $n$ even; $c_n^2 = 1$ for $n$
odd) reduces in both cases to $c_1^{n+1} = 1$.
\end{proof}

\subsection{Density asymptotics}
\label{sec:proof_thm_case_I}


\medskip

We now assemble the proof of the sharp density theorem
(Theorem~\ref{theo:typ_1_cells_sharp}).  Two ingredients are
needed beyond Theorem~\ref{thm:puiseux} itself:

\begin{enumerate}[label=(\roman*)]
\item A \emph{Tauberian transfer} 
(Corollary~\ref{cor:chain-density_0}) that converts the Puiseux
asymptotics $y_k(u)\sim c_k u^{\beta_k}$ of the Poisson
transform into matching power-law behavior
$\rho_k(x)\sim \tilde c_k|x|^{\beta_k}$ of the spectral
density.

\item A \emph{Stieltjes inversion} that reads off the sharp
leading constant of the scalar density $f(x)$ from the trace of
the matrix Cauchy transform.
\end{enumerate}

Recall that if $y_k(u) = -\Im G_{kk}(iu)$ is the imaginary part on the
imaginary axis of a diagonal entry of a matrix Cauchy transform, the
Poisson representation
\[
y_k(u) \;=\; \int_{\bR}\frac{u}{x^2+u^2}\,d\mu_k(x)
\]
implies, via a standard Tauberian correspondence
(cf.\ Lemma~\ref{lem:puiseux-to-poisson}), that $y_k(u)\sim c_k u^{\beta_k}$
as $u\downarrow 0$ translates into $\rho_k(x)\sim \tilde c_k\,|x|^{\beta_k}$
for the spectral density~$\rho_k$ near $x=0$. Combining with
Theorem~\ref{thm:puiseux}:

\begin{corollary}[Spectral regularity along the chain]
\label{cor:chain-density_0}
Under the Puiseux hypothesis of Theorem~\ref{thm:puiseux}, the diagonal
spectral densities at the origin behave as
\[
\rho_k(x) \;\sim\; \tilde c_k\,|x|^{\beta_k},
\qquad \beta_k = (-1)^k\,\frac{n+1-2\lceil k/2\rceil}{n+1}.
\]
In particular:
\begin{itemize}
\item the odd-indexed entries $\rho_{2m-1}$ have integrable power-law
blow-ups of order $(n+1-2m)/(n+1)$;
\item the even-indexed entries $\rho_{2m}$ have power-law cusps of the
same order $(n+1-2m)/(n+1)$;
\item for $n$ odd, the terminal entry $\rho_n$ is regular
at $0$ ($\beta_n=0$); for $n$ even, all diagonal entries are non-regular.
\end{itemize}
\end{corollary}

\begin{proof}[Proof of Theorem \ref{theo:typ_1_cells_sharp}]

\smallskip
\noindent\textbf{Step 1: Existence of the chain expansion.}
Let $W(u)=\diag(w_1(u),\dots,w_n(u))$ be the accretive diagonal solution of
Speicher's equation, whose diagonality was established in
\S\ref{sec:diagonal_reduction}, and let
$(y_k(u))_{k=1}^{n}=(w_n,w_1,w_{n-1},w_2,\dots)$ be the reshuffled chain
solving~\eqref{eq:chain}; each $y_k$ is real and positive for $u>0$. On the
imaginary axis $(y_k(u))$ is the solution of the vector Dyson equation for the
variance profile $S$ of Remark~\ref{rem:typeI-KR}, and $S$ has support, since
its antidiagonal entries $s_{i,\,n+1-i}=1$ form a positive diagonal. By
\cite[Prop.~4.1 and Lem.~3.1]{MR4901646} this solution admits a convergent
expansion in fractional powers of $u$ at $u=0^{+}$; in particular each $y_k$ has
a leading Puiseux term
\[
y_k(u)\;=\;c_k\,u^{\beta_k}\bigl(1+o(1)\bigr),
\qquad c_k>0,\quad \beta_k\in(-1,1),\qquad u\downarrow 0 .
\]
The hypotheses of Theorem~\ref{thm:puiseux} are therefore met. 
(Theorem~\ref{thm:puiseux} recovers exponents $\beta_k$---and, beyond
\cite{MR4901646}, the coefficients---from the integrable structure of the chain.)

\smallskip
\noindent\textbf{Step 2: Exponents and coefficients.}
By Theorem~\ref{thm:puiseux},
\[
\beta_{2m-1} \;=\; -\,\frac{n-2m+1}{n+1},\qquad
\beta_{2m}   \;=\; \phantom{-}\frac{n-2m+1}{n+1},
\]
\[
c_{2m-1}=c_1^m,\quad c_{2m}=c_1^{-m},\quad c_1^{n+1}=1.
\]
Combined with $c_1>0$ (Step 1), the last relation forces 
$c_1=1$, and hence $c_k=1$ for every $k$. In particular,
$\beta_1 = -\alpha_n$ is strictly the smallest exponent,
and
\[
\sum_{k=1}^n y_k(u) \;=\; u^{-\alpha_n}\bigl(1+o(1)\bigr),
\qquad u\downarrow 0.
\]

\smallskip
\noindent\textbf{Step 3: From $W$ to $f$.}
Since $W(u)=iG_S(iu I)$,
\[
H(iu) \;=\; \tr G_S(iu I) \;=\; -\frac{i}{n}\sum_{k=1}^n y_k(u)
\;=\; -\frac{i}{n}\,u^{-\alpha_n}\bigl(1+o(1)\bigr).
\]
By Proposition~\ref{prop:entrywise-algebraicity}, $H$ is algebraic 
and, by Step 2, does not have a pole at $0$; the classical 
Puiseux theorem then gives a convergent expansion of $H$ 
in $\mathbb{C}^+$ in powers of some $z^{1/m}$. Matching 
leading terms along $z=iu = u\,e^{i\pi/2}$ identifies the 
leading coefficient $C$:
\[
H(z) \;=\; C\,z^{-\alpha_n}\bigl(1+o(1)\bigr),
\qquad
C \;=\; -\frac{i}{n}\,e^{i\pi\alpha_n/2}
   \;=\; \frac{1}{n}\,e^{-i\pi(1-\alpha_n)/2},
\]
with the branch $z^{-\alpha_n}>0$ for $z>0$.
Stieltjes inversion $f(x)=-\tfrac{1}{\pi}\Im H(x+i0^+)$ 
gives, for $x\downarrow 0$,
\[
f(x) \;=\; \frac{\cos(\pi\alpha_n/2)}{\pi n}\,x^{-\alpha_n}
   + o\!\left(x^{-\alpha_n}\right),
\]
and the symmetry $S\stackrel{d}{=}-S$ gives the same 
expansion for $x\uparrow 0$ with $x$ replaced by $|x|$.
\end{proof}


\section{Type II cells: statement and regularization strategy} 
\label{sec:type_II}

The second case that we consider is $A_n =  F_n x_1 + (\alpha F_n + G_n) x_2$, where $\alpha \ne 0$ is a real parameter.

\begin{theo}[Type~II cells — explicit density at the origin]
\label{theo:typ_2_cells_sharp}
For the matrix semicircle  $S_n =  F_n \otimes s_1 + (\alpha F_n + G_n) \otimes s_2$, with a real $\alpha \ne 0$,
the scalar density $f_\alpha$ satisfies
\[
        f_\alpha(x)
        =
        C_{\alpha,n}\,|x|^{-\frac{n-1}{n+1}}
        +
        o\left(|x|^{-\frac{n-1}{n+1}}\right),
        \qquad x\to0,
\]
where
\[
        \boxed{
        C_{\alpha,n}
        =
        \frac{1}{n\pi}
        \sin\left(\frac{\pi}{n+1}\right)
        (1+\alpha^2)^{-\frac{2n-1}{n+1}}.
        }
\]
In particular, the singularity exponent is independent of $\alpha$:
\[
        \boxed{
        f_\alpha(x)
        \asymp
        |x|^{-\frac{n-1}{n+1}}
        \quad\text{near }x=0.
        }
\]
\end{theo}

\subsection*{Strategy of the proof}

Theorem~\ref{theo:typ_2_cells_sharp} is proved by extracting the leading behaviour of
$\operatorname{tr}_n W_\alpha(u)$ as $u\downarrow0$, where $W_\alpha(u)=iG_\alpha(iu)$ solves
Speicher's equation; a Stieltjes inversion (\S\ref{sec:proof_thm_case_II}) then converts the
leading $u^{-p}$ term of the trace, with $p=\tfrac{n-1}{n+1}$, into the leading $|x|^{-p}$ term
of $f_\alpha$ and its sharp constant. Unlike the Type~I cell, here $\eta_\alpha$ does not
preserve the diagonal, so there is no scalar chain and Speicher's equation must be analysed
directly. The obstacle is that $W_\alpha(u)$ degenerates as $u\to0$, its entries scaling at the
distinct rates $\beta_i=\tfrac{n-2i+1}{n+1}$. The argument has two stages.

\emph{Stage 1 (this section): regularization and the obstruction.}
After recording the joint analyticity of $W_\alpha$ in $(\alpha,u)$
(\S\ref{sec:analyticity-in-alpha}) and a $\mathbb Z/2$-symmetry that makes every spectral
invariant even in $\alpha$, so that we may assume $\alpha\ge0$ (\S\ref{sec:Z2-symmetry}), we
remove the degeneration by the rescaling 
\[
M_\alpha(t)=D(t)^{-1}W_\alpha(u)D(t)^{-1}
\]
 with
$D(t)=\diag(t^{\gamma_i})$, $t=u^{1/(n+1)}$ (\S\ref{sec:rescaling}). The rescaled equation is
regular at $t=0$, with limiting equation $N^{-1}=\Theta(N)$, $\Theta=\operatorname{Ad}F_n$.
Crucially, its positive-solution set is not a point but a manifold $\mathcal M$ of dimension
$\lfloor n^2/4\rfloor$: the linearization at the base point has kernel the $\Theta$-odd subspace
$E_-$ (\S\ref{sec:linearization}). This kernel is the obstruction to a direct
implicit-function argument and shapes the rest of the proof.

\emph{Stage 2 (\S\ref{sec:LS}): Lyapunov--Schmidt and the shifted base point.}
Splitting $N=e^{Y}+Z$ along $E_-\oplus E_+$, the transverse component $Z$ is eliminated by the
implicit function theorem (\S\ref{sec:LS-reduction}), and the reduced equation on the kernel
variable $Y$ is shown to be divisible by $t^2$ (\S\ref{sec:lyapunov-divisibility}), yielding a
regular doubly-reduced equation $\mathfrak R(Y,\alpha,t)$. The subtle point is that the true
solution does \emph{not} limit to $Y=0$: the constant term $\mathfrak R(0,\alpha,0)$ is nonzero
(\S\ref{sec:constant-term}). We therefore compute the reduced linearization
(\S\ref{sec:linearization-B}), locate the correct base point $Y_*(\alpha)\in E_-$ in closed form
(\S\ref{sec:shifted-base-point}), and recenter there (\S\ref{sec:extension}); at $Y_*$ the
reduced linearization is invertible, so the implicit function theorem produces a holomorphic
solution $M_\alpha(t)=M_\alpha^{(0)}+O(t)$ with explicit diagonal limit $M_\alpha^{(0)}$
(\S\ref{sec:singularities_type_2}). Its most singular entry gives the leading term of
$\operatorname{tr}_n W_\alpha(u)$ and completes the proof (\S\ref{sec:proof_thm_case_II}).

%

\subsection{Joint analyticity of $W_\alpha(u)$ in $(\alpha, u)$}
\label{sec:analyticity-in-alpha}

The solution $W_\alpha(u)$ exists as a unique positive definite real symmetric matrix for every $\alpha\in\mathbb R$ and $u>0$, by the existence and uniqueness theorem of~\cite{hfs2007}. The following theorem records its smooth dependence on $(\alpha, u)$.

\begin{theorem}[Joint real-analyticity in $(\alpha, u)$]
\label{thm:analyticity-in-alpha}
The map
\[
(\alpha, u) \;\longmapsto\; W_\alpha(u) \;\in\; \mathrm{Sym}_n^{++}
\]
is real-analytic on $\mathbb R \times (0, \infty)$.
\end{theorem}

Proof is in Appendix \ref{sec:analyticity}.

\subsection{A $\mathbb{Z}/2$-symmetry of the family $\eta_\alpha$}
\label{sec:Z2-symmetry}

In this section we identify a discrete symmetry of the family of covariance maps $\{\eta_\alpha\}_{\alpha\in\mathbb R}$ that links the values at $\alpha$ and $-\alpha$ via a fixed congruence. The symmetry transfers to the solutions of Speicher's equation and yields several immediate consequences. In particular, all spectral invariants of $W_\alpha(u)$ are even functions of $\alpha$.

\subsubsection*{Setup and statement}

 Define
\begin{equation}
\label{eq:Jn-def}
J_n \;:=\; \operatorname{diag}\bigl((-1)^1, (-1)^2, \ldots, (-1)^n\bigr)
\;=\;\operatorname{diag}(-1, 1, -1, \ldots).
\end{equation}
Note that $J_n^* = J_n$, $J_n^2 = I$, so $J_n$ is a real orthogonal involution.

\begin{proposition}[$\mathbb{Z}/2$-symmetry]
\label{prop:Z2-symmetry}
For every $n\ge 2$ and every $\alpha\in\mathbb R$,
\begin{equation}
\label{eq:eta-symmetry}
J_n\,\eta_\alpha(B)\,J_n \;=\; \eta_{-\alpha}(J_n B J_n)
\qquad\text{for all }B\in M_n(\mathbb C).
\end{equation}
Consequently, the unique accretive solution $W_\alpha(u)\in\mathrm{Sym}_n^{++}$ of Speicher's equation satisfies
\begin{equation}
\label{eq:W-symmetry}
W_{-\alpha}(u) \;=\; J_n\, W_\alpha(u)\, J_n
\qquad\text{for all }\alpha\in\mathbb R,\ u>0.
\end{equation}
\end{proposition}

\begin{proof}
\emph{Step 1: action of conjugation by $J_n$ on $F_n$ and $G_n$.}
For any matrix $M\in M_n(\mathbb C)$, conjugation by $J_n$ acts entry-wise as $(J_n M J_n)_{ij} = (-1)^{i+j} M_{ij}$. Since $F_n$ is supported on the anti-diagonal $\{(i,j) : i+j = n+1\}$ and $G_n$ on $\{(i,j) : i+j = n\}$, we obtain
\begin{equation}
\label{eq:JFnGnJ-rules}
J_n F_n J_n = (-1)^{n+1} F_n,\qquad
J_n G_n J_n = (-1)^n\, G_n.
\end{equation}

\emph{Step 2: the symmetry of $\eta_\alpha$.}
Using $J_n^2 = I$ and the identity $J_n (X Y X) J_n = (J_n X J_n)(J_n Y J_n)(J_n X J_n)$,
\begin{align*}
J_n\,\eta_\alpha(B)\,J_n
&= (J_n F_n J_n)(J_n B J_n)(J_n F_n J_n)
\\&\quad + (J_n(\alpha F_n + G_n)J_n)(J_n B J_n)(J_n(\alpha F_n + G_n)J_n)
\\
&= F_n (J_n B J_n) F_n
\\
& + \bigl((-1)^{n+1}\alpha F_n + (-1)^n G_n\bigr)(J_n B J_n)\bigl((-1)^{n+1}\alpha F_n + (-1)^n G_n\bigr),
\end{align*}
using \eqref{eq:JFnGnJ-rules} and $\bigl((-1)^{n+1}\bigr)^2 = 1$ in the first summand.

We now check that the second summand equals the corresponding term of $\eta_{-\alpha}(J_n B J_n)$, namely $(-\alpha F_n + G_n)(J_n B J_n)(-\alpha F_n + G_n)$. The conjugation $X\mapsto MXM$ is invariant under $M \mapsto -M$, so it suffices to verify
\[
(-1)^{n+1}\alpha F_n + (-1)^n G_n \;=\; \pm(-\alpha F_n + G_n).
\]
Both signs of $(-1)^{n+1}$ and $(-1)^n$ are opposite, so the two cases reduce to:
\begin{itemize}
\item \emph{$n$ even:} $-\alpha F_n + G_n = -\alpha F_n + G_n$ (sign $+$).
\item \emph{$n$ odd:} $\alpha F_n - G_n = -(-\alpha F_n + G_n)$ (sign $-$).
\end{itemize}
In either case, the conjugation of $J_n B J_n$ by this matrix coincides with conjugation by $-\alpha F_n + G_n$. Hence
\begin{align*}
J_n\,\eta_\alpha(B)\,J_n &=\; F_n (J_n B J_n) F_n + (-\alpha F_n + G_n)(J_n B J_n)(-\alpha F_n + G_n)
\\
&=\; \eta_{-\alpha}(J_n B J_n),
\end{align*}
proving \eqref{eq:eta-symmetry}.

\emph{Step 3: symmetry of the solution.}
Conjugating Speicher's equation $W_\alpha(u)^{-1} = uI + \eta_\alpha(W_\alpha(u))$ by $J_n$ and using $J_n^2 = I$:
\begin{align*}
(J_n W_\alpha(u) J_n)^{-1}
&=\; J_n W_\alpha(u)^{-1} J_n
\;=\; uI + J_n\,\eta_\alpha(W_\alpha(u))\,J_n
\\
&\stackrel{\eqref{eq:eta-symmetry}}{=}\; uI + \eta_{-\alpha}(J_n W_\alpha(u) J_n).
\end{align*}
Hence $J_n W_\alpha(u) J_n$ is an accretive solution of Speicher's equation associated with $\eta_{-\alpha}$. By uniqueness of the accretive solution (\cite{hfs2007}), $J_n W_\alpha(u) J_n = W_{-\alpha}(u)$, which is \eqref{eq:W-symmetry}.
\end{proof}

\begin{remark}[Even dependence on the coupling parameter]
\label{rem:even-alpha}
Since $W_{-\alpha}(u) = J_n W_\alpha(u) J_n$ with $J_n$ 
orthogonal (Proposition~\ref{prop:Z2-symmetry}), every spectral 
invariant of $W_\alpha(u)$ --- eigenvalues, trace, determinant, 
operator norm --- is an even function of~$\alpha$. Combined 
with the real-analyticity of $\alpha \mapsto W_\alpha(u)$ 
(Theorem~\ref{thm:analyticity-in-alpha}), these invariants are 
in fact real-analytic functions of~$\alpha^2$. In particular, we 
may assume $\alpha \ge 0$ without loss of generality throughout 
the remainder of this section.
\end{remark}

\subsection{Rescaling and the limiting equation}\label{sec:rescaling}

We rescale Speicher's equation
\begin{align*}
        W_\alpha(u)^{-1} &=\; uI + \eta_\alpha(W_\alpha(u)),
        \\
        \eta_\alpha(B) &= (1+\alpha^2)F_nBF_n + \alpha(F_nBG_n+G_nBF_n) + G_nBG_n,
\end{align*}
so that all positive powers of $u$ enter through a single small parameter
\[
        t \;:=\; u^{1/(n+1)},
        \qquad u = t^{n+1}.
\]

\textbf{The rescaled unknown.}
The Puiseux exponents of the unperturbed branch are
\[
        \beta_i \;=\; \frac{n-2i+1}{n+1},
        \qquad i=1,\ldots,n,
\]
with $\beta_1 > \cdots > \beta_n$, $\sum_i\beta_i = 0$, and
$\beta_1 = -\beta_n = (n-1)/(n+1)$.  Set
\[
        \gamma_i \;:=\; \tfrac{n+1}{2}\beta_i = \tfrac{n-2i+1}{2},
        \qquad
        D(t) \;:=\; \diag\bigl(t^{\gamma_1},\ldots,t^{\gamma_n}\bigr),
\]
and introduce the rescaled unknown
\begin{equation}\label{eq:M-rescaled-n}
        M_\alpha(t) \;:=\; D(t)^{-1}\, W_\alpha(u)\, D(t)^{-1},
        \qquad u = t^{n+1}.
\end{equation}
With this normalization, conjugation by $D$ acts entrywise as
\begin{equation}\label{eq:DXD-formula}
        [DXD]_{ij} \;=\; t^{\gamma_i+\gamma_j}\,X_{ij}
        \;=\; t^{\,n+1-(i+j)}\,X_{ij},
\end{equation}
so the scaling of the $(i,j)$-entry depends only on the antidiagonal
index~$i+j$.

Substituting $W_\alpha = DMD$ into Speicher's equation and conjugating by~$D$ on
both sides yields the \emph{exact rescaled equation}
\begin{equation}\label{eq:HMS-rescaled-n}
        \boxed{\;
        M^{-1} \;=\; uD^2 + D\,\eta_\alpha(DMD)\,D
        \;=:\; \mathcal R_\alpha(M;t)\;}.
\end{equation}
Both $W_\alpha(u)\succ0$ and $D(t)\succ0$, so $M_\alpha(t)\succ0$ as well.

\paragraph{Explicit form of $\mathcal R_\alpha$.}
The exponent identities
\[
        \gamma_i + \gamma_{n+1-i} = 0,\qquad
        \gamma_i + \gamma_{n-i} = 1\ \ (1\le i\le n-1)
\]
are exactly what is needed to evaluate the four conjugated products in
$D\,\eta_\alpha(DMD)\,D$.  Direct calculation using \eqref{eq:DXD-formula}
gives
\[
\begin{aligned}
        D F_n(DMD) F_n D &= F_nMF_n,
        & D F_n(DMD) G_n D &= t\,F_nMG_n, \\
        D G_n(DMD) F_n D &= t\,G_nMF_n,
        & D G_n(DMD) G_n D &= t^2\,G_nMG_n,
\end{aligned}
\]
together with $uD^2 = \diag\bigl(t^{2(n+1-i)}\bigr)_{i=1}^n$.  Combining
these via $\eta_\alpha(B) = (1+\alpha^2)F_nBF_n + \alpha(F_nBG_n+G_nBF_n) + G_nBG_n$
yields the explicit decomposition
\begin{equation}\label{eq:RHS-explicit-n}
        \mathcal R_\alpha(M;t)
        \;=\; c\,F_nMF_n
        \;+\; \alpha t\,\mathcal P(M)
        \;+\; t^2\,\mathcal Q(M)
        \;+\; \mathcal U(t),
\end{equation}
where $c := 1+\alpha^2$ and
\begin{align}
        \label{eq:P-def}
        \mathcal P(M) &\;:=\; F_nMG_n + G_nMF_n, \\
        \label{eq:Q-def}
        \mathcal Q(M) &\;:=\; G_nMG_n, \\
        \label{eq:U-def}
        \mathcal U(t)_{ij} &\;:=\; t^{2(n+1-i)}\,\delta_{ij}.
\end{align}
The leading piece $cF_nMF_n$ is independent of $\alpha$ in form: the
$\alpha$-dependence enters only through the scalar $c=1+\alpha^2$ and the
cross-term $\alpha t\,\mathcal P(M)$.

\begin{exa}
For $n=2$, 
\begin{align*}
\PP(M) &= F_2 M G_2 + G_2 M F_2 = \begin{pmatrix} 2 M_{12} &  M_{11} \\  M_{11} & 0 \end{pmatrix},
\\
\QQ(M) &= G_2 M G_2 = \begin{pmatrix} M_{11} & 0 \\ 0 & 0 \end{pmatrix}, \text{ and }
\UU(t) = \begin{pmatrix} t^4 & 0 \\ 0 & t^2 \end{pmatrix}.
\end{align*}
\end{exa}

\textbf{The limiting equation.}
Setting $t=0$ in \eqref{eq:RHS-explicit-n} yields the \emph{limiting equation}
\begin{equation}\label{eq:limit-eq-n}
        M_*^{-1} \;=\; c\,F_n M_* F_n,
        \qquad c = 1+\alpha^2.
\end{equation}
Remarkably, this equation has the same form for all $n$ and depends on
$\alpha$ only through the scalar $c$.  In particular it admits the scalar
positive solution $M_* = c^{-1/2}I$.  The Puiseux exponents are determined by
the rescaling \eqref{eq:M-rescaled-n}, which is independent of $\alpha$;
hence any $\alpha$-dependence in the leading asymptotics of $W_\alpha(u)$
must enter through the prefactor $M_*(\alpha)$.  The full set of positive
solutions of \eqref{eq:limit-eq-n} forms a manifold $\mathcal M_c \subset
\mathrm{Sym}_n^{++}$ whose dimension is computed in~\S\ref{sec:linearization}.

\textbf{Normalization at the identity.}
For the implicit-function analysis it is convenient to centre the limiting
equation at $M_* = c^{-1/2}I$.  We therefore set
\[
        N \;:=\; c^{1/2}\,M.
\]
Substituting $M = c^{-1/2}N$ into \eqref{eq:RHS-explicit-n} and dividing by
$c^{1/2}$ gives the \emph{normalized rescaled equation}
\begin{equation}\label{eq:normalized}
        \boxed{\;
        N^{-1} \;=\; F_n N F_n + \frac{\alpha}{c}\,t\,\mathcal P(N)
        + \frac{1}{c}\,t^2\,\mathcal Q(N)
        + c^{-1/2}\,\mathcal U(t)\;}.
\end{equation}
Equivalently $\widetilde{\mathcal G}(N,\alpha,t)=0$, where
\[
        \widetilde{\mathcal G}(N,\alpha,t) \;:=\;
        N^{-1} - F_n N F_n - \tfrac{\alpha}{c}\,t\,\mathcal P(N)
        - \tfrac{1}{c}\,t^2\,\mathcal Q(N) - c^{-1/2}\,\mathcal U(t).
\]
At $t=0$, \eqref{eq:normalized} reduces to $N_*^{-1} = F_n N_* F_n$, with
scalar positive solution $N_* = I$.  This is the form used throughout~\S\ref{sec:LS}.

\subsection{Linearization and its kernel}
\label{sec:linearization}

We linearize the rescaled equation at the scalar base point and show that
the kernel of the linearization is precisely the antiinvariant subspace of
an involution~$\Theta$ on $\mathrm{Sym}_n$. The chain-order presentation
makes this kernel transparent and is the description used in~\S\ref{sec:LS}.

\paragraph{Setup.}
By \S\ref{sec:rescaling} the normalized rescaled equation
\eqref{eq:normalized} reads $\widetilde{\mathcal G}(N,\alpha,t)=0$, and at
$t=0$ reduces to $N_*^{-1}=F_nN_*F_n$, with scalar positive solution
$N_*=I$. We linearize $\widetilde{\mathcal G}(\,\cdot\,,\alpha,t)$ in $N$
at $(N,t)=(I,0)$:
\[
        L_0(H)
        \;:=\; D_N\widetilde{\mathcal G}(I,\alpha,0)[H].
\]

\textbf{Chain order and the involution $\Theta$.}
Let $\pi$ be the chain-order permutation
\[
        (y_1,y_2,y_3,y_4,\ldots) = (w_n,w_1,w_{n-1},w_2,\ldots),
\]
with permutation matrix~$P$. In the chain basis, the natural-order
reversal $F_n$ becomes the adjacent-pair swap
\begin{equation}\label{eq:Theta-chain}
        J \;:=\; P F_n P^\mathsf{T},
        \qquad
        Je_{2r-1}=e_{2r},\ \ Je_{2r}=e_{2r-1},
        \quad\text{(and $Je_n = e_n$ if $n$ is odd).}
\end{equation}
Define the involution
\[
        \Theta(H) \;:=\; J H J \;=\; \mathrm{Ad}\,J(H).
\]
Equivalently, in natural-order coordinates $\Theta(H)=F_nHF_n$, since $J$
and $F_n$ are conjugate via~$P$. Because $J^2=I$ and $J=J^\mathsf{T}$,
$\Theta^2=\mathrm{id}$ and $\Theta$ is self-adjoint with respect to the
Hilbert--Schmidt inner product $\langle X,Y\rangle = \operatorname{tr}(XY)$.
Hence $\mathrm{Sym}_n$ splits orthogonally as
\[
        \mathrm{Sym}_n
        \;=\; \mathrm{Sym}_n^+ \oplus \mathrm{Sym}_n^-,
        \qquad
        \mathrm{Sym}_n^\pm := \{H : \Theta(H)=\pm H\},
\]
with projections $H_\pm = \tfrac12(H\pm\Theta(H))$.

\textbf{The chain-order rescaling identity.}
By~\eqref{eq:RHS-explicit-n} at $\alpha=0$, the natural-order rescaled
equation reads
\[
        \mathcal R_0(M;t)
        \;=\; F_n M F_n + t^2 G_n M G_n + \mathcal U(t).
\]
Conjugating by~$P$ converts this into chain-order form
\begin{equation}\label{eq:chain-rescaling}
        \mathcal R_0^{\mathrm{chain}}(M;t)
        \;=\; \Theta(M) + t^2\, K M K + \mathcal U^{\mathrm{chain}}(t),
        \qquad K := P G_n P^\mathsf{T}.
\end{equation}
The key feature is that the leading piece is exactly $\Theta(M)$, with no
$t$-dependent dressing; this is the reason the kernel of the linearization
collapses onto a single involution.

\textbf{The linearization at the scalar base point.}
Differentiating $\widetilde{\mathcal G}$ in $N$ at $(N,t)=(I,0)$ gives
\begin{equation}\label{eq:L0-formula}
        L_0(H) \;=\; -H - \Theta(H)
        \;=\; -(I+\Theta)\,H,
\end{equation}
independently of~$\alpha$, since the $\alpha$-dependent terms in
\eqref{eq:normalized} carry positive powers of~$t$.

\begin{lemma}[Linearization at the scalar base point]
\label{lem:L0}
The operator $L_0\colon\mathrm{Sym}_n\to\mathrm{Sym}_n$ defined by
\eqref{eq:L0-formula} is self-adjoint with respect to the Hilbert--Schmidt
inner product, and
\[
        L_0(H_+) = -2H_+,\qquad L_0(H_-) = 0.
\]
In particular,
\[
        \ker L_0 = \mathrm{Sym}_n^-,
        \qquad
        \operatorname{im} L_0 = \mathrm{Sym}_n^+,
\]
with the orthogonal decomposition $\mathrm{Sym}_n = \ker L_0\oplus\operatorname{im}L_0$.
\end{lemma}

\begin{proof}
Self-adjointness: $L_0 = -(I+\Theta)$ is a real linear combination of two
self-adjoint operators. The eigenvalue identities follow from
$\Theta(H_\pm)=\pm H_\pm$:
\[
        L_0(H_+) = -(H_++H_+)=-2H_+,
        \qquad
        L_0(H_-) = -(H_- - H_-) = 0.
        \qedhere
\]
\end{proof}

\textbf{Dimension of the kernel.}

\begin{proposition}\label{prop:kernel-dim}
$\dim_{\mathbb R}\mathrm{Sym}_n^- = \lfloor n^2/4\rfloor$.
\end{proposition}

\begin{proof}
Since $\Theta^2 = \mathrm{id}$ on $\mathrm{Sym}_n$,
\[
        \dim\mathrm{Sym}_n^+ - \dim\mathrm{Sym}_n^-
        = \operatorname{tr}_{\mathrm{Sym}_n}(\Theta).
\]
Combined with $\dim\mathrm{Sym}_n^+ + \dim\mathrm{Sym}_n^- = n(n+1)/2$, this
gives
\[
        \dim\mathrm{Sym}_n^-
        = \tfrac12\bigl(n(n+1)/2 - \operatorname{tr}_{\mathrm{Sym}_n}\Theta\bigr).
\]

Compute $\operatorname{tr}_{\mathrm{Sym}_n}\Theta$. In the orthonormal
basis $\{S_{ij}:i\le j\}$ of $\mathrm{Sym}_n$ with
$S_{ij} := (E_{ij}+E_{ji})/\sqrt{1+\delta_{ij}}$, $\Theta$ permutes basis
elements according to the index-pair involution
\[
        \sigma\colon\{i,j\}\mapsto\{\tau(i),\tau(j)\},
\]
where $\tau$ is the chain-pair swap $\tau(2r-1)=2r$, $\tau(2r)=2r-1$ (and
$\tau(n)=n$ if $n$ is odd). Hence
$\operatorname{tr}_{\mathrm{Sym}_n}\Theta = \#\{\sigma\text{-fixed unordered pairs}\}$,
and a $\sigma$-fixed pair $\{i,j\}$ is one of the following:
\begin{itemize}
\item $i = j$ with $\tau(i) = i$: occurs only if $n$ is odd, with $i = j = n$;
\item $i \ne j$ with $\{\tau(i),\tau(j)\} = \{i,j\}$, $\tau(i)=j$:
i.e.\ $\{i,j\}=\{2r-1,2r\}$ for some $r$. There are $\lfloor n/2\rfloor$ such pairs.
\end{itemize}
Hence $\operatorname{tr}_{\mathrm{Sym}_n}\Theta = \lceil n/2\rceil$, and
\[
        \dim\mathrm{Sym}_n^-
        = \tfrac12\bigl(n(n+1)/2-\lceil n/2\rceil\bigr)
        = \lfloor n^2/4\rfloor,
\]
verified by parity:
\[
\begin{array}{l@{\quad}l}
        n=2m:   & \tfrac12(m(2m+1)-m) = m^2 = n^2/4,\\[2pt]
        n=2m+1: & \tfrac12((2m+1)(m+1)-(m+1)) = m(m+1) = \lfloor n^2/4\rfloor.
\end{array}
\]
\end{proof}

\textbf{Explicit basis at a diagonal solution.}
For Lyapunov--Schmidt computations in natural-order coordinates the
following description of $\ker L_0$ is convenient. We work at any diagonal
positive solution $M_*$ of the original (un-normalized) limiting
equation~\eqref{eq:limit-eq-n}.

\begin{corollary}[Kernel basis at a diagonal solution, natural order]
\label{cor:kernel-basis}
Let $M_*=\operatorname{diag}(a_1,\ldots,a_n)\in\mathrm{Sym}_n^{++}$ satisfy
$M_*^{-1}=cF_nM_*F_n$, equivalently $a_ia_{n+1-i}=1/c$ for $i=1,\ldots,n$.
Let $\sigma$ be the natural-order index involution
\[
        \sigma\colon (i,j)\mapsto (n+1-j,n+1-i),
\]
acting on $\{(i,j):i\le j\}$. Then the kernel of the (un-normalized)
linearization
\[
        \mathcal L_*(V)
        := -M_*^{-1}VM_*^{-1} - cF_nVF_n
\]
on $\mathrm{Sym}_n$ has basis
\begin{equation}\label{eq:kernel-Tij}
        T^{(i,j)}
        := \frac{E_{ij}+E_{ji}}{1+\delta_{ij}}
        - \frac{1}{c\,a_ia_j}\cdot
                \frac{E_{i'j'}+E_{j'i'}}{1+\delta_{i'j'}},
\end{equation}
where $(i',j')=\sigma(i,j)\ne(i,j)$ ranges over representatives of the
$\lfloor n^2/4\rfloor$ size-$2$ orbits of~$\sigma$.
\end{corollary}

\begin{proof}
For diagonal $M_*$, the equation $\mathcal L_*(V)=0$ reads entrywise
\[
        \frac{V_{ij}}{a_ia_j} \;=\; -c\,V_{n+1-i,\,n+1-j} \;=\; -c\,V_{\sigma(i,j)},
\]
which decouples on $\sigma$-orbits.
\begin{itemize}
\item For a $\sigma$-fixed pair $(i,j)$ (i.e.\ $i+j=n+1$), the equation
becomes $V_{ij}\bigl(1/(a_ia_j)+c\bigr)=0$. Since $a_ia_{n+1-i}=1/c$, the
coefficient equals $2c\ne0$, forcing $V_{ij}=0$. There are $\lceil n/2\rceil$
such pairs, contributing $0$ to $\ker\mathcal L_*$.
\item For a size-$2$ orbit $\{(i,j),\sigma(i,j)\}$, the linear system
$V_{ij}/(a_ia_j)+cV_{i'j'}=0$ and $V_{i'j'}/(a_{i'}a_{j'})+cV_{ij}=0$ has
determinant $1/(a_ia_ja_{i'}a_{j'})-c^2 = c^2-c^2=0$ (using
$a_ia_{n+1-i}=1/c$), so it has a one-dimensional solution space generated
by~\eqref{eq:kernel-Tij}.
\end{itemize}
The total dimension is the number of size-$2$ orbits, which equals
$\lfloor n^2/4\rfloor$ by Proposition~\ref{prop:kernel-dim} (or equivalently
by the count $\binom{n+1}{2}-\lceil n/2\rceil$ divided by $2$).
\end{proof}

\begin{remark}[Specialization to the scalar base point]
At $M_*=c^{-1/2}I$ (so $a_i = c^{-1/2}$ and $a_ia_j=1/c$), the basis
elements~\eqref{eq:kernel-Tij} simplify to
\[
        T^{(i,j)}
        \;=\; \frac{E_{ij}+E_{ji}}{1+\delta_{ij}}
        - \frac{E_{i'j'}+E_{j'i'}}{1+\delta_{i'j'}},
\]
i.e.\ antisymmetric combinations of $\sigma$-paired symmetric matrix
units. These span $\mathrm{Sym}_n^-$, recovering the chain-order kernel
description $\ker L_0 = \mathrm{Sym}_n^-$ of Lemma~\ref{lem:L0}.
\end{remark}

\begin{remark}[Manifold of solutions]\label{rmk:manifold-dim}
The kernel of $L_0$ is the tangent space at $N_* = I$ to the manifold
\[
        \mathcal M
        := \{N\in\mathrm{Sym}_n^{++}: N^{-1} = \Theta(N)\}.
\]
Since $L_0$ is self-adjoint and surjective onto $\mathrm{Sym}_n^+$
(Lemma~\ref{lem:L0}), the implicit function theorem yields that $\mathcal M$
is locally a smooth submanifold near $N_*=I$ with
\[
        \dim\mathcal M = \dim\ker L_0 = \lfloor n^2/4\rfloor.
\]
The same dimension is obtained at every diagonal positive solution
$M_*\in\mathcal M_c$ in natural order via Corollary~\ref{cor:kernel-basis};
in particular $\dim\mathcal M_c = \lfloor n^2/4\rfloor$ as well.
\end{remark}

\section{The Lyapunov--Schmidt analysis for Type II cells}
\label{sec:LS}
We resume the analysis of \S\ref{sec:type_II}; recall the two-stage strategy described there. 
Our goal in this section is to prove the following theorem. 


\begin{theo}[All $\alpha$ rescaled holomorphicity]
\label{theo:all-alpha-rescaled-holomorphicity}
For every real $\alpha$, the rescaled solution $M_\alpha(t)$ extends
holomorphically to a neighborhood of $t=0$, and
\[
        M_\alpha(t)
        =
        M_\alpha^{(0)}+O(t),
\]
where
\[
        M_\alpha^{(0)}
        =
        \operatorname{diag}
        \left(
        c^{1-\frac{3}{n+1}},
        c^{1-\frac{6}{n+1}},
        \ldots,
        c^{1-\frac{3n}{n+1}}
        \right),
        \qquad
        c=1+\alpha^2.
\]
\end{theo}

\subsection{The reduction}
\label{sec:LS-reduction}

We use the kernel description of \S\ref{sec:linearization} to set up a
Lyapunov--Schmidt reduction of the normalized rescaled equation
\eqref{eq:normalized}. The kernel of the linearization at the scalar base
point is \emph{not} an obstruction: it is the tangent space to the
manifold of limiting positive solutions, and the reduction simply
parametrizes that manifold in $\Theta$-odd coordinates.

\textbf{Decomposition.}
Let $\Theta = \mathrm{Ad}\,J$ be the involution on $\mathrm{Sym}_n$ from
\S\ref{sec:linearization}. Set
\begin{equation}\label{eq:E-pm}
        E_- := \{Y\in\mathrm{Sym}_n : \Theta(Y)=-Y\},
        \qquad
        E_+ := \{Z\in\mathrm{Sym}_n : \Theta(Z)= Z\},
\end{equation}
and let $P_\pm := \tfrac12(\mathrm{Id}\pm\Theta)$ be the orthogonal
projections onto $E_\pm$. By Lemma~\ref{lem:L0},
\[
        \ker L_0 = E_-,\qquad \operatorname{im} L_0 = E_+,
        \qquad
        L_0|_{E_+} = -2\,\mathrm{Id}_{E_+}.
\]

\textbf{Parametrization of the limiting manifold.}
At $t=0$ the normalized equation reads $N^{-1}=\Theta(N)$, with manifold of
positive solutions
\[
        \mathcal M := \{N\in\mathrm{Sym}_n^{++}: N^{-1}=\Theta(N)\}.
\]
Near $N_*=I$, the manifold $\mathcal M$ is parametrized by $E_-$ via the
exponential:
\begin{equation}\label{eq:limit-manifold-param}
        Y\in E_-\ \text{(small)}
        \ \longmapsto\
        e^Y\in\mathcal M.
\end{equation}
Indeed, for $Y\in E_-$ one has $\Theta(Y)=-Y$, hence
\[
        \Theta(e^Y) = e^{\Theta(Y)} = e^{-Y} = (e^Y)^{-1},
\]
so $e^Y$ satisfies the limiting equation. The image is locally smooth of
the correct dimension $\lfloor n^2/4\rfloor$
(Proposition~\ref{prop:kernel-dim}), so it locally parametrizes $\mathcal M$.

\textbf{Lyapunov--Schmidt parametrization.}
We parametrize $N$ near $I$ by an $E_-$-tangential variable $Y$ and an
$E_+$-transverse correction $Z$:
\begin{equation}\label{eq:N-param}
        N \;=\; e^Y + Z,
        \qquad Y\in E_-,\ \ Z\in E_+.
\end{equation}
Splitting $\widetilde{\mathcal G}$ by the projections $P_\pm$ gives the
\emph{even equation}
\[
        \mathcal E(Y,Z,\alpha,t)
        := P_+\widetilde{\mathcal G}(e^Y+Z,\alpha,t) = 0,
\]
and the \emph{odd equation} $P_-\widetilde{\mathcal G}(e^Y+Z,\alpha,t)=0$.
The even equation is transverse to $\mathcal M$ and is solved for $Z$ by
the implicit function theorem; the odd equation is the genuine reduced
problem on the kernel variable $Y$.

\begin{lemma}[Solving the even equation]\label{lem:LS-even}
There exist neighborhoods $U\subset E_-$ of $0$, $V\subset\mathbb R$ of any
fixed $\alpha_0$, and $W\subset\mathbb R$ of $0$, and a real-analytic map
\begin{equation}\label{eq:Phi-defn}
        \Phi\colon U\times V\times W \to E_+,
\end{equation}
such that for $(Y,\alpha,t)\in U\times V\times W$ the even equation
$\mathcal E(Y,\Phi(Y,\alpha,t),\alpha,t)=0$ holds, and $\Phi$ is the unique
such function. Moreover
\[
        \Phi(Y,\alpha,0) \equiv 0.
\]
\end{lemma}

\begin{proof}
At $(Y,Z,t)=(0,0,0)$: by~\eqref{eq:limit-manifold-param}, $e^0=I$ solves
the limiting equation, so $\widetilde{\mathcal G}(I,\alpha,0)=0$ and a
fortiori $\mathcal E(0,0,\alpha,0)=0$.

The derivative in $Z$ at this point: since the $\alpha t\mathcal P$,
$t^2\mathcal Q$, and $\mathcal U$ terms in $\widetilde{\mathcal G}$ vanish
at $t=0$, we have $D_Z\mathcal E(0,0,\alpha,0)[Z] = P_+\, L_0[Z]$. By
Lemma~\ref{lem:L0}, $L_0[Z] = -2Z$ for $Z\in E_+$, so
\[
        D_Z\mathcal E(0,0,\alpha,0) = -2\,\mathrm{Id}_{E_+},
\]
which is invertible. The implicit function theorem produces the unique
$\Phi$ with $\Phi(0,\alpha,0)=0$.

The identity $\Phi(Y,\alpha,0)\equiv 0$ follows from~\eqref{eq:limit-manifold-param}:
at $t=0$ the choice $Z=0$ already solves
$\mathcal E(Y,0,\alpha,0) = P_+\widetilde{\mathcal G}(e^Y,\alpha,0) = 0$, and
$\Phi$ is the unique solution.
\end{proof}

\begin{lemma}[The reduced equation is divisible by $t$]\label{lem:LS-reduced}
Define the reduced equation on $E_-$ by
\begin{equation}\label{eq:R-defn}
        \mathscr R(Y,\alpha,t)
        := P_-\widetilde{\mathcal G}\bigl(e^Y+\Phi(Y,\alpha,t),\,\alpha,\,t\bigr).
\end{equation}
Then $\mathscr R(Y,\alpha,0)\equiv 0$, and there exists a real-analytic
function $\widehat{\mathscr R}\colon U\times V\times W\to E_-$ with
\begin{equation}\label{eq:R-divisible}
        \mathscr R(Y,\alpha,t) = t\,\widehat{\mathscr R}(Y,\alpha,t).
\end{equation}
\end{lemma}

\begin{proof}
By Lemma~\ref{lem:LS-even}, $\Phi(Y,\alpha,0)=0$, so
$\mathscr R(Y,\alpha,0) = P_-\widetilde{\mathcal G}(e^Y,\alpha,0)$. By
\eqref{eq:limit-manifold-param} the right-hand side vanishes:
$\widetilde{\mathcal G}(e^Y,\alpha,0) = 0$. Since $\mathscr R$ is
real-analytic in $(Y,\alpha,t)$ and vanishes at $t=0$, the function
\[
        \widehat{\mathscr R}(Y,\alpha,t)
        := \int_0^1 \partial_t\mathscr R(Y,\alpha,st)\,ds
\]
is real-analytic and satisfies $\mathscr R = t\widehat{\mathscr R}$.
\end{proof}

The reduced equation $\mathscr R(Y,\alpha,t)=0$ is therefore equivalent
near $t=0$ to
\[
        \widehat{\mathscr R}(Y,\alpha,t) = 0,
\]
a finite-dimensional analytic equation on the $\lfloor n^2/4\rfloor$-dimensional
kernel space $E_-$.

%
\textbf{Outline of the analysis.}

The remainder of \S\ref{sec:LS} is devoted to solving the reduced equation
\[
        \widehat{\mathscr R}(Y, \alpha, t) = 0,
\]
on a neighborhood of $(Y, t) = (0, 0)$ in $E_- \times \mathbb R$. A
real-analytic solution branch establishes that $W_\alpha(u)$ has the same
Puiseux structure at $u = 0$ as the unperturbed branch $W_0(u)$ of
Theorem~\ref{thm:puiseux}.

We will
\begin{enumerate}[nosep,label=(\roman*)]
\item establish divisibility of $\mathscr R$ by $t^2$ via a Lyapunov
  identity (\S\ref{sec:lyapunov-divisibility}),
\item compute the constant term and linearization of the doubly-reduced
  equation $\mathfrak R$ (\S\ref{sec:constant-term}--\ref{sec:linearization-B}),
\item identify the explicit shifted base point (\S\ref{sec:shifted-base-point}), and
\item apply the IFT in the recentered equation for all~$\alpha$
  (\S\ref{sec:extension}--\ref{sec:singularities_type_2}).
\end{enumerate}

\subsection{The Lyapunov identity and divisibility by $t^2$}
\label{sec:lyapunov-divisibility}

Recall the equation is $\widetilde{\mathcal G}(N,\alpha,t)=0$, where
\[
        \widetilde{\mathcal G}(N,\alpha,t) \;:=\;
        N^{-1} - F_n N F_n - \tfrac{\alpha}{c}\,t\,\mathcal P(N)
        - \tfrac{1}{c}\,t^2\,\mathcal Q(N) - c^{-1/2}\,\mathcal U(t).
\]
 Throughout we
work in the natural order, where
\[
        F_n=\sum_{i=1}^n E_{i,n+1-i},\qquad
        G_n=\sum_{i=1}^{n-1}E_{i,n-i}
\]
are the reversal matrix and the shifted reversal matrix, and
\[
        \Theta(M)=F_nMF_n,\qquad
        \cP(M)=F_nMG_n+G_nMF_n.
\]
Using $E_{ab}E_{cd}=\delta_{bc}E_{ad}$ one finds at once
\[
        L_n:=F_nG_n=\sum_{k=2}^n E_{k,k-1},\qquad
        U_n:=G_nF_n=\sum_{i=1}^{n-1}E_{i,i+1},
\]
 so $L_n,U_n$ are the lower and upper shift matrices and
\[
\mathcal P(I) \;=\; L_n+U_n \;=:\; T_n
\;=\;
\begin{pmatrix}
0 & 1      &        &        & 0 \\
1 & 0      & 1      &        &   \\
  & 1      & \ddots & \ddots &   \\
  &        & \ddots & 0      & 1 \\
0 &        &        & 1      & 0
\end{pmatrix}
\]
is the adjacency matrix of the path on $n$ vertices.

\medskip
\noindent\textbf{Vanishing of $P_-\cP(I)$.}\;
Since $F_n^2=I$,
\[
        \Theta(L_n)=F_n(F_nG_n)F_n=G_nF_n=U_n,
\]
and symmetrically $\Theta(U_n)=L_n$. Hence $\Theta(T_n)=T_n$, i.e.\ $T_n$ is
$\Theta$-even, and
\[
        P_-\cP(I)=\tfrac12(T_n-\Theta(T_n))=0.
\]
Consequently
\[
        \scrR_1(0,\alpha)=-\alpha c^{-1/2}P_-\cP(I)=0,
\]
so the branch $Y=0$ is not ruled out at order $t$.


\begin{remark}
One can show that the ``naive'' first-order linearization
$A_1:=P_-\cP|_{E_-}$ equals $-\tfrac12\{T_n,\cdot\}$, with
eigenstructure inherited from the path-graph spectrum. However,
the identity below shows that the entire reduced equation
$\widehat{\scrR}(\,\cdot\,,\alpha,0)$ vanishes on~$E_-$, making
this eigenstructure analysis unnecessary.
\end{remark}

\subsubsection*{The Lyapunov identity and divisibility by $t^2$}

The vanishing of $P_-\cP(I)$ shows that the first divided reduced equation
$\widehat{\scrR}$ vanishes at the base point $Y=0$. In fact a much stronger
statement holds: $\widehat{\scrR}(Y,\alpha,0)\equiv 0$ on all of~$E_-$,
equivalently $\scrR$ is divisible by $t^2$. The mechanism is a single
Lyapunov-type identity for the transverse correction that collapses the
entire cascade of obstructions simultaneously.

Fix $Y\in E_-$, set $A := e^{-Y}$, and define
\[
        \phi(Y) \;:=\; \partial_t\Phi(Y,\alpha,0) \;\in\; E_+.
\]
Note $\Theta(A)=A^{-1}$ and $e^Y = A^{-1}$.

\begin{lemma}[Lyapunov identity for $\phi(Y)$]
\label{lem:lyapunov}
For every $Y\in E_-$ in a neighborhood of $0$,
\begin{equation}\label{eq:sylvester}
        A\,\phi(Y)\,A \;+\; \phi(Y)
        \;=\; -a\,\bigl(A L_n + U_n A\bigr).
\end{equation}
\end{lemma}
In other words, $\phi(Y)$ is a solution to the discrete Lyapunov equation (also called the Stein equation) $AXA + X = C$  with coefficient $A = e^{-Y}$ and right-hand side $C = -a(AL_n + U_n A)$. This equation has a unique solution whenever $-1 \notin \operatorname{spec}(A) \cdot \operatorname{spec}(A)$, which holds in our setting since $A = e^{-Y}$ is positive-definite.

\begin{proof}
For $Y\in E_-$ we have, using $F_n^{2}=I$ and $\Theta(e^Y)=e^{-Y}=A$,
\begin{align}
        \mathcal P(e^Y)
        &= F_n e^Y G_n + G_n e^Y F_n
        = (F_n e^Y F_n)(F_n G_n) + (G_n F_n)(F_n e^Y F_n) \notag
        \\
        &= A L_n + U_n A. \label{equ:Pe^Y}
\end{align}

The even equation $P_+\widetilde{\mathcal G}(e^Y+\Phi(Y,\alpha,t),\alpha,t)=0$
holds identically in $t$, so its $t$-derivative at $t=0$ also vanishes. By
the chain rule,
\[
        \partial_t\widetilde{\mathcal G}\bigl(e^Y+\Phi(Y,\alpha,t),\alpha,t\bigr)\Big|_{t=0}
        \;=\; D_N\widetilde{\mathcal G}\bigl(e^Y,\alpha,0\bigr)\bigl[\phi(Y)\bigr]
        \;+\; \partial_t\widetilde{\mathcal G}\bigl(e^Y,\alpha,0\bigr),
\]
using $\Phi(Y,\alpha,0)=0$ (so the argument of $\widetilde{\mathcal G}$
reduces to $e^Y$ at $t=0$) and the definition $\phi(Y):=\partial_t\Phi(Y,\alpha,0)$.

\smallskip

\noindent\emph{First piece.} Differentiating $N \mapsto N^{-1}$ in the
direction $K$ gives $-N^{-1}KN^{-1}$, hence
\[
        D_N\widetilde{\mathcal G}(N,\alpha,t)[K]
        \;=\; -N^{-1}KN^{-1} - \Theta(K) - at\,\mathcal P(K) - bt^{2}\mathcal Q(K).
\]
At $(N,t)=(e^Y,0)$ the last two terms drop and $(e^Y)^{-1}=A$, so
\[
        D_N\widetilde{\mathcal G}(e^Y,\alpha,0)\bigl[\phi(Y)\bigr]
        \;=\; -A\,\phi(Y)\,A \;-\; \Theta\bigl(\phi(Y)\bigr).
\]

\smallskip

\noindent\emph{Second piece.}
\[
        \partial_t\widetilde{\mathcal G}(N,\alpha,t)
        \;=\; -a\,\mathcal P(N) - 2bt\,\mathcal Q(N) - d\,\mathcal U'(t).
\]
At $t=0$ the $\mathcal Q$-term vanishes because of the explicit factor of
$t$, and $\mathcal U'(0)=0$ because every diagonal entry of $\mathcal U(t)$
is $t^{2(n+1-i)}$ with $n+1-i\ge 1$, hence $O(t^{2})$. Therefore
\[
        \partial_t\widetilde{\mathcal G}(e^Y,\alpha,0)
        \;=\; -a\,\mathcal P(e^Y)
        \;=\; -a\bigl(AL_n + U_n A\bigr),
\]
using the identity $\mathcal P(e^Y)=AL_n+U_n A$ established above.

\smallskip

\noindent Combining the two pieces, using $\phi(Y)\in E_+$ (hence
$\Theta(\phi(Y))=\phi(Y)$), and applying $P_+$:
\[
        P_+\bigl[\,-A\,\phi(Y)\,A \;-\; \phi(Y) \;-\; a(AL_n + U_n A)\,\bigr] \;=\; 0,
\]
or equivalently,
\begin{equation}\label{eq:sylvester-Pplus}
        P_+\bigl[\,A\,\phi(Y)\,A + \phi(Y) + a(AL_n + U_n A)\,\bigr] \;=\; 0.
\end{equation}

To upgrade \eqref{eq:sylvester-Pplus} to the unprojected identity
\eqref{eq:sylvester} we argue by uniqueness. The operator
$\mathcal S_Y(X) := AXA + X$ is invertible for every $Y \in E_-$
 (since $A = e^{-Y}$ is positive-definite, $\operatorname{spec}(A \otimes A) \subset (0,\infty)$ and $-1$ cannot be an eigenvalue of $A \otimes A$). So the equation
$\mathcal S_Y(X)=-a(AL_n + U_n A)$ has a unique solution~$X$.
Applying $\Theta$ to this equation and using
$\Theta(A)=A^{-1}$, $\Theta(L_n)=U_n$, $\Theta(U_n)=L_n$:
\[
        A^{-1}\Theta(X)A^{-1} + \Theta(X)
        = -a\bigl(A^{-1} U_n + L_n A^{-1}\bigr).
\]
Conjugating by $A$ on both sides recovers
$A\,\Theta(X)\,A + \Theta(X) = -a(AL_n + U_n A)$, so $\Theta(X)$ also
solves the Lyapunov equation. 

By uniqueness $\Theta(X)=X$, i.e.\ $X\in E_+$.

It remains to identify $\phi(Y)$ with $X$. Both lie in $E_+$ and satisfy the
\emph{same} projected equation: applying $P_+$ to the unprojected identity
$\mathcal S_Y(X)=-a(AL_n+U_nA)$ and using $P_+X=X$ gives
$P_+\mathcal S_Y(X)=-a\,P_+(AL_n+U_nA)$, while \eqref{eq:sylvester-Pplus} together
with $P_+\phi(Y)=\phi(Y)$ reads $P_+\mathcal S_Y(\phi(Y))=-a\,P_+(AL_n+U_nA)$.
Hence $W:=\phi(Y)-X\in E_+$ obeys $P_+\mathcal S_Y(W)=0$. But the restriction
$P_+\mathcal S_Y|_{E_+}\colon E_+\to E_+$ is positive definite: since
$P_-\mathcal S_Y(W)\in E_-\perp E_+\ni W$,
\begin{align*}
        \langle W,\,P_+\mathcal S_Y(W)\rangle_{\mathrm{HS}}
         &=\;\langle W,\,\mathcal S_Y(W)\rangle_{\mathrm{HS}}
        \;=\;\operatorname{tr}\!\bigl(WAWA\bigr)+\|W\|_{\mathrm{HS}}^2
        \\
        &=\;\bigl\|A^{1/2}WA^{1/2}\bigr\|_{\mathrm{HS}}^2+\|W\|_{\mathrm{HS}}^2
        \;\ge\;\|W\|_{\mathrm{HS}}^2,
\end{align*}
using $A=A^*\succ0$. So $P_+\mathcal S_Y|_{E_+}$ is injective, $W=0$, i.e.\
$\phi(Y)=X$, and the full identity \eqref{eq:sylvester} holds.

\end{proof}

\begin{proposition}\label{prop:Rhat-vanishes}
For every $Y\in E_-$ and every $\alpha$,
$\widehat{\mathscr R}(Y,\alpha,0) = 0$. Equivalently, $\mathscr R$ is
divisible by $t^{2}$ as an analytic function of $(Y,\alpha,t)$, and we may
define the doubly-reduced equation
\[
        \mathfrak R(Y,\alpha,t) \;:=\; \frac{\mathscr R(Y,\alpha,t)}{t^{2}}.
\]
\end{proposition}

\begin{proof}
By definition $\widehat{\mathscr R} = \mathscr R/t$, so $\mathscr R$ is
divisible by $t$ (this was established at the end of \S\ref{sec:LS-reduction})
and
\[
        \widehat{\mathscr R}(Y, \alpha, 0)
        \;=\; \lim_{t\to 0}\frac{\mathscr R(Y, \alpha, t)}{t}
        \;=\; \partial_t\mathscr R(Y, \alpha, 0).
\]

The reduced equation is
\[
        \mathscr R(Y, \alpha, t)
        \;=\; P_-\,\widetilde{\mathcal G}\bigl(e^Y + \Phi(Y, \alpha, t),\,\alpha,\,t\bigr).
\]
The projection $P_-$ is linear and $t$-independent, so it commutes with
$\partial_t$. The chain rule applied to the $t$-derivative inside the
brackets gives, exactly as in the proof of
Lemma~\ref{lem:lyapunov},
\[
        \partial_t\widetilde{\mathcal G}\bigl(e^Y + \Phi(Y,\alpha,t),\alpha,t\bigr)\Big|_{t=0}
        \;=\; D_N\widetilde{\mathcal G}(e^Y, \alpha, 0)\bigl[\phi(Y)\bigr]
        \;+\; \partial_t\widetilde{\mathcal G}(e^Y, \alpha, 0),
\]
using $\Phi(Y, \alpha, 0) = 0$ and $\phi(Y) := \partial_t\Phi(Y, \alpha, 0)$.
The two pieces, computed in the proof of Lemma~\ref{lem:lyapunov}, are
\[
        D_N\widetilde{\mathcal G}(e^Y, \alpha, 0)\bigl[\phi(Y)\bigr]
        \;=\; -A\,\phi(Y)\,A - \Theta\bigl(\phi(Y)\bigr),
        \qquad
        \partial_t\widetilde{\mathcal G}(e^Y, \alpha, 0)
        \;=\; -a\,\mathcal P(e^Y).
\]
Combining and applying $P_-$,
\[
        \widehat{\mathscr R}(Y, \alpha, 0)
        \;=\; P_-\bigl[\,-A\,\phi(Y)\,A \;-\; \Theta(\phi(Y)) \;-\; a\,\mathcal P(e^Y)\,\bigr].
\]
Using $\Theta(\phi(Y)) = \phi(Y)$ (since $\phi(Y) \in E_+$) and
$\mathcal P(e^Y) = AL_n + U_nA$ (established in the proof of
Lemma~\ref{lem:lyapunov}), this becomes
\[
        \widehat{\mathscr R}(Y, \alpha, 0)
        \;=\; P_-\bigl[\,-A\,\phi(Y)\,A \;-\; \phi(Y) \;-\; a(AL_n + U_nA)\,\bigr].
\]
By Lemma~\ref{lem:lyapunov} the bracketed expression is identically zero,
so $\widehat{\mathscr R}(Y, \alpha, 0) = 0$.
\end{proof}

In particular, the cancellation $D_Y\widehat{\mathscr R}(0,\alpha,0)=0$
is an immediate consequence: the function
$\widehat{\mathscr R}(\,\cdot\,,\alpha,0)$ is identically zero on~$E_-$,
so all its derivatives vanish as well.

\subsection{The constant term $\mathfrak R(0,\alpha,0)$} 
\label{sec:constant-term}

We compute the value of the doubly-reduced equation at $Y=0$. By
construction $\mathfrak R(0,\alpha,0)$ is the coefficient of $t^{2}$ in
$\mathscr R(0,\alpha,t) = P_-\widetilde{\mathcal G}(I+\Psi(t),\alpha,t)$,
where $\Psi(t) := \Phi(0,\alpha,t)\in E_+$.

The Sylvester identity \eqref{eq:sylvester} at $Y=0$ (so $A=I$) reduces
to $2\phi(0)=-a(L_n+U_n)=-a T_n$. Hence the leading coefficient of
$\Psi(t) = z_1 t + z_2 t^{2} + O(t^{3})$ is
\begin{equation}
\label{eq:z1-formula}
        z_1 \;=\; \phi(0) \;=\; -\frac{a}{2}\,T_n.
\end{equation}

Expanding $\widetilde{\mathcal G}(I+\Psi(t),\alpha,t)$ in $t$ and using
$z_1, z_2 \in E_+$, the $t^{2}$ coefficient is
\begin{equation}
\label{equ:C_2}
        C_2 \;=\; (z_1^{2} - 2z_2)
        \;-\; a\,\mathcal P(z_1)
        \;-\; b\,\mathcal Q(I)
        \;-\; d\,E_{nn}.
\end{equation}
Since $z_1^{2}, z_2 \in E_+$, the $E_+$-part contributes nothing to
$P_- C_2$, leaving
\[
        \mathfrak R(0,\alpha,0)
        \;=\; \frac{a^{2}}{2}\,P_-\mathcal P(T_n)
        \;-\; b\,P_-\mathcal Q(I)
        \;-\; d\,P_-E_{nn}.
\]
where we used $-a\mathcal P(z_1)=\frac{a^{2}}{2}\mathcal P(T_n)$
from~\eqref{eq:z1-formula}.

\textbf{The three odd projections.}
For $T_n\in E_+$ a direct calculation gives $\mathcal P(T_n) = L_n^{2}+2U_nL_n+U_n^{2}$,
hence $\Theta\mathcal P(T_n)=L_n^{2}+2L_nU_n+U_n^{2}$, and
\[
        P_-\mathcal P(T_n) \;=\; U_nL_n - L_nU_n
        \;=\; (I-E_{nn})-(I-E_{11})
        \;=\; E_{11}-E_{nn},
\]
using $L_nU_n = I - E_{11}$ and $U_nL_n = I-E_{nn}$. Similarly,
$\mathcal Q(I) = G_n^{2} = I-E_{nn}$, so $\Theta\mathcal Q(I)=I-E_{11}$ and
\[
        P_-\mathcal Q(I) \;=\; \tfrac12(E_{11}-E_{nn}),
        \qquad
        P_-E_{nn} \;=\; -\tfrac12(E_{11}-E_{nn}).
\]

\textbf{The closed form.}
Combining,
\[
        \mathfrak R(0,\alpha,0)
        \;=\; \tfrac12\bigl(a^{2} - b + d\bigr)\,(E_{11}-E_{nn}).
\]
Substituting $a=\alpha/c$, $b=1/c$, $d=1/\sqrt c$ and using $\alpha^{2}-c=-1$,
\[
        a^{2} - b \;=\; \frac{\alpha^{2}-c}{c^{2}} \;=\; -\frac{1}{c^{2}},
\]
hence
\begin{equation}\label{eq:R2-at-0}
        \boxed{\;\;
        \mathfrak R(0,\alpha,0)
        \;=\; \frac{1}{2}\bigl(c^{-1/2}-c^{-2}\bigr)\,(E_{11}-E_{nn}),
        \qquad c=1+\alpha^{2}.\;\;}
\end{equation}
For small~$\alpha$, $c^{-1/2}-c^{-2}=\tfrac32\alpha^{2}+O(\alpha^{4})$, so
\[
        \mathfrak R(0,\alpha,0) \;=\; \tfrac34\,\alpha^{2}\,(E_{11}-E_{nn})
        \;+\; O(\alpha^{4}).
\]
In particular $\mathfrak R(0,0,0)=0$, but $\mathfrak R(0,\alpha,0)\ne 0$
for $\alpha\ne 0$. The implication is geometric: for $\alpha\ne 0$ the
$\alpha$-perturbed solution does not converge to $N=I$ as $t\downarrow 0$,
but to a shifted base point inside the limiting manifold $\mathcal M$.


\subsection{The linearization $\mathcal B_\alpha$}
\label{sec:linearization-B}

We now compute the Jacobian of the second divided reduced equation.
Recall that
\[
        \mathscr R(Y,\alpha,t)=t^2\mathfrak R(Y,\alpha,t).
\]
Thus the linearization
\[
        \mathcal B_\alpha
        :=
        D_Y\mathfrak R(0,\alpha,0)
        \colon E_- \to E_-
\]
is the coefficient of $t^2$ in
\[
        D_Y\mathscr R(0,\alpha,t)[H],
        \qquad H\in E_-.
\]

Set
\[
        c:=1+\alpha^2,
        \qquad
        a:=\alpha c^{-1},
        \qquad
        b:=c^{-1},
        \qquad
        d:=c^{-1/2}.
\]
The normalized equation is
\[
        \widetilde{\mathcal G}(N,\alpha,t)
        =
        N^{-1}
        -
        \Theta(N)
        -
        at\,\mathcal P(N)
        -
        bt^2\,\mathcal Q(N)
        -
        d\,\mathcal U(t).
\]

\begin{proposition}[Linearization of the second reduced equation]
\label{prop:Balpha-linearization}
Let
\[
        L_n:=F_nG_n,
        \qquad
        U_n:=G_nF_n,
        \qquad
        T_n:=L_n+U_n.
\]
Also set
\[
        E_{\partial}:=E_{11}+E_{nn},
        \qquad
        D_{\partial}:=L_nU_n+U_nL_n=2I-E_{\partial}.
\]
Then, for $H\in E_-$,
\begin{equation}
\label{eq:Balpha-closed}
        \boxed{
        \mathcal B_\alpha(H)
        =
        \frac{1}{4c^2}
        \left[
        2(L_nHU_n+U_nHL_n)
        -
        D_{\partial}H
        -
        HD_{\partial}
        \right]
        -
        \frac{1}{4\sqrt c}
        \left(
        E_{\partial}H+HE_{\partial}
        \right).
        }
\end{equation}
Moreover, $\mathcal B_\alpha$ is negative definite with respect to the
Hilbert--Schmidt inner product on the Hermitian matrices. In particular,
$\mathcal B_\alpha$ is invertible on $E_-$.
\end{proposition}

The proof is in Appendix \ref{sec:jacobian}.

\begin{remark}
The point of the proposition is that the first divided reduced equation has
zero $Y$-linearization, but after the second division the linearization is
the strictly negative operator $\mathcal B_\alpha$. Therefore the
implicit-function theorem should be applied to
\[
        \mathfrak R(Y,\alpha,t)=0,
\]
not to the first divided equation.
\end{remark}

\subsection{The shifted base point: diagonal invariance and explicit formula}
\label{sec:shifted-base-point}

The argument has three steps: an entrywise analysis of the $t^2$
coefficient for diagonal~$Y$, a reflection identity on the second
off-diagonal that forces diagonality of $\mathfrak R(Y,\alpha,0)$,
and the direct verification of a closed-form solution.

Throughout this subsection, indices outside $\{1,\ldots,n\}$ are interpreted
as making the corresponding entry zero. We write $\bar i := n+1-i$ and use
the involution $\Theta(M)=F_nMF_n$, so that
$\Theta(M)_{ij}=M_{\bar i,\bar j}$. The constants are
$c:=1+\alpha^2$, $a:=\alpha/c$, $b:=1/c$, $d:=c^{-1/2}$.

We work inside the diagonal $\Theta$-odd subspace
\[
        D_-:=
        \bigl\{
        Y=\diag(y_1,\ldots,y_n)\,:\,
        y_{\bar i}=-y_i
        \bigr\}
        \subset E_- .
\]

\subsubsection*{The first-order correction}

Fix $Y\in D_-$ and set $A:=e^{-Y}=\diag(r_1,\ldots,r_n)$, so
$r_{\bar i}=r_i^{-1}$. Expand the even correction in $t$,
\[
        \Phi(Y,\alpha,t)=t\phi+t^2\psi+O(t^3),
        \qquad \phi,\psi\in E_+.
\]
Substituting $N=A^{-1}+t\phi+t^2\psi+O(t^3)$ into
$\widetilde{\cG}(N,\alpha,t)=0$ and using
$\cP(e^Y)=AL_n+U_nA$ (established in  \eqref{equ:Pe^Y}), the $t^1$ coefficient
yields (see Lemma \ref{lem:lyapunov})
\begin{equation}\label{eq:phi-eq}
        A\phi A+\phi=-a\bigl(AL_n+U_nA\bigr).
\end{equation}
Since $A$ is diagonal, \eqref{eq:phi-eq} is entrywise: $\phi_{ij}=0$ unless
$|i-j|=1$, and
\begin{equation}\label{eq:phi-formula}
        \phi_{i,i+1}=\phi_{i+1,i}
        =-\,\frac{a\,r_{i+1}}{1+r_ir_{i+1}},
        \qquad i=1,\ldots,n-1.
\end{equation}
We write $p_i:=\phi_{i,i+1}$ for short.
A direct check using $r_{\bar i}=r_i^{-1}$ shows
$\phi_{\bar i,\overline{i+1}}=\phi_{i,i+1}$, so $\phi\in E_+$.

\subsubsection*{The $t^2$ coefficient and its support}

Expanding $N^{-1}=A-tA\phi A+t^2(A\phi A\phi A-A\psi A)+O(t^3)$ and
$\Theta(N)=A+t\phi+t^2\psi+O(t^3)$, and noting that
$\cU(t)=t^2E_{nn}+O(t^4)$ on the diagonal, the $t^2$ coefficient is
\begin{equation}\label{eq:C2}
        C_2 \;=\;
        A\phi A\phi A - A\psi A - \psi
        - a\,\cP(\phi) - b\,\cQ(e^Y) - d\,E_{nn} .
\end{equation}
We record the supports of the individual pieces. Since $A$ is diagonal and
$\phi$ has band-$\pm 1$ support, the product $A\phi A\phi A$ is supported
on bands $\{0,\pm 2\}$. The action of $\cP$ on a matrix is, by direct
computation from the definitions $(F_n)_{ij}=\delta_{i+j,n+1}$ and
$(G_n)_{ij}=\delta_{i+j,n}$,
\begin{equation}\label{eq:P-action}
        (F_nMG_n)_{ij} = M_{\bar i,\bar j-1},
        \qquad
        (G_nMF_n)_{ij} = M_{\bar i-1,\bar j},
\end{equation}
so $\cP(\phi)$ has support contained in bands $\{0,\pm 2\}$. The remaining
terms in \eqref{eq:C2} are diagonal. Hence $C_2$ is supported on the
diagonal and on the second off-diagonal bands.

\subsubsection*{The reflection identity}

Set
\[
        B \;:=\; A\phi A\phi A \;-\; a\,\cP(\phi).
\]
We claim that $B$ satisfies a precise reflection compatibility relation on
the second off-diagonal bands.

\begin{lemma}
\label{lem:refl-id}
For $Y\in D_-$ with $A=e^{-Y}=\diag(r_1,\ldots,r_n)$, the matrix
$B=A\phi A\phi A-a\cP(\phi)$ satisfies
\begin{equation}\label{eq:refl}
        B_{ij} \;=\; r_ir_j\,B_{\bar i,\bar j}
        \qquad
        \text{for every $(i,j)$ with $|i-j|=2$.}
\end{equation}
\end{lemma}

\begin{proof}
The relation \eqref{eq:refl} is the $\Theta$-equivariance of the order-$t^2$
term, evaluated on the second off-diagonal. 
The entrywise verification is given in Appendix~\ref{sec:reflection_identity}.
\end{proof}

\subsubsection*{Diagonality of $\fR(Y,\alpha,0)$ and of the shifted base point}

\begin{lemma}\label{lem:P-minus-vanishes}
Let $\psi\in E_+$ be determined by $P_+C_2=0$ on the second off-diagonal
bands. Then the second off-diagonal component of $P_-C_2$ vanishes; in
particular $\fR(Y,\alpha,0)=P_-C_2$ is diagonal and lies in $D_-$.
\end{lemma}

\begin{proof}
On the second off-diagonal, the diagonal terms in \eqref{eq:C2} do not
contribute, so
\[
        C_2|_{ij}
        =B_{ij}-(1+r_ir_j)\psi_{ij},
        \qquad |i-j|=2,
\]
using $(A\psi A)_{ij}=r_ir_j\psi_{ij}$ for diagonal $A$. Since $\psi\in
E_+$ we have $\psi_{\bar i,\bar j}=\psi_{ij}$, and the condition
$P_+C_2|_{ij}=0$ becomes
\[
        \bigl[(1+r_ir_j)+(1+(r_ir_j)^{-1})\bigr]\psi_{ij}
        =B_{ij}+B_{\bar i,\bar j},
\]
which determines $\psi_{ij}$ uniquely. Substituting this $\psi_{ij}$ into
$P_-C_2|_{ij}=\tfrac12\bigl[(B_{ij}-B_{\bar i,\bar j})-(r_ir_j-(r_ir_j)^{-1})
\psi_{ij}\bigr]$ and writing $s:=r_ir_j$,
\[
        P_-C_2|_{ij}
        =\frac{B_{ij}-s\,B_{\bar i,\bar j}}{1+s},
\]
which vanishes by Lemma~\ref{lem:refl-id}. Hence $P_-C_2$ has no
second-off-diagonal entries. Since $P_-C_2\in E_-$ by construction and the
only remaining support is the diagonal, $P_-C_2\in D_-$.
\end{proof}

\begin{propo}\label{prop:diagonal-invariance}
The reduced map at $t=0$ preserves $D_-$:
\[
        \fR(D_-,\alpha,0)\subseteq D_- .
\]
Consequently the shifted base point $Y_\alpha^{(0)}\in E_-$ produced by the
implicit function theorem at $(Y,t)=(0,0)$ in fact lies in $D_-$, and the
limiting rescaled solution
\[
        \lim_{t\to 0} M_\alpha(t)
        \;=\; c^{-1/2}\,e^{Y_\alpha^{(0)}}
\]
is diagonal.
\end{propo}

\begin{proof}
The inclusion $\fR(D_-,\alpha,0)\subseteq D_-$ is
Lemma~\ref{lem:P-minus-vanishes}. To close, recall that the linearization
$\cB_0:=D_Y\fR(0,\alpha,0)$ is invertible on $E_-$
(\ref{prop:Balpha-linearization}). Differentiating Lemma~\ref{lem:P-minus-vanishes}
at $Y=0$ shows $\cB_0$ preserves $D_-$; in finite dimension a restriction
of an invertible operator to an invariant subspace is invertible, so
$\cB_0|_{D_-}$ is invertible. The implicit function theorem applied inside
$D_-$ produces a unique small zero $Y\in D_-$ of $\fR(\cdot,\alpha,0)$. By
uniqueness of the small zero of $\fR(\cdot,\alpha,0)$ in the full space
$E_-$, this restricted zero coincides with $Y_\alpha^{(0)}$. Hence
$Y_\alpha^{(0)}\in D_-$, and the formula for the limit follows from the
rescaling identity $M_\alpha(t)=c^{-1/2}\,N(Y(\alpha,t))$ at $t=0$.
\end{proof}

\begin{remark}[The unperturbed case]
\label{rem:alpha-zero-IFT}
At $\alpha=0$, $\fR(0,0,0)=0$ and $\cB_0$ is invertible, so the implicit
function theorem applies directly at $(Y,\alpha,t)=(0,0,0)$, recovering
the unperturbed holomorphic rescaled limit.  For $\alpha\ne 0$ the
nonvanishing of $\fR(0,\alpha,0)$ shows that the base point shifts inside
$\mathcal M$; the correct approach is to recenter at the shifted point,
which we now identify.
\end{remark}

\subsubsection*{The explicit shifted base point}

\begin{propo}[Explicit diagonal solution of the limiting reduced equation]
\label{prop:explicit-diagonal-solution}
Let $n\ge 2$, $\alpha\in\bR$, and set $c:=1+\alpha^2$.
Define
\begin{equation}\label{eq:Y-star-formula}
        Y_*(\alpha)
        :=
        \frac32\log c\,
        \diag\!\left(
                \frac{n-1}{n+1},
                \frac{n-3}{n+1},
                \ldots,
                -\frac{n-3}{n+1},
                -\frac{n-1}{n+1}
        \right).
\end{equation}
Then $Y_*(\alpha)\in D_-$ and
$\mathfrak R(Y_*(\alpha),\alpha,0)=0$.
\end{propo}

\begin{proof}
Write $A=e^{-Y_*(\alpha)}=\diag(r_1,\dots,r_n)$ with
$r_i=c^{\frac{3}{n+1}\left(i-\frac{n+1}{2}\right)}$; since $r_{\bar i}=r_i^{-1}$ we have
$Y_*(\alpha)\in D_-$. By the reduction of \S\ref{sec:lyapunov-divisibility}, the identity
$\fR(Y_*,\alpha,0)=0$ is equivalent to the $E_+$-solvability of the order-$t^2$ Stein
equation $A\psi A+\psi=R$ with $R$ as in \eqref{eq:R-definition}, i.e.\ to the entrywise
compatibility $R_{ij}=r_ir_jR_{\bar i\bar j}$ of \eqref{eq:compatibility-all-entries}.
As $R$ is supported on the diagonal and the second off-diagonal, this is a finite set of
scalar identities; with $q:=c^{3/(n+1)}$ (so $r_{i+1}=qr_i$) they are verified in
Appendix~\ref{sec:type2-shifted-basepoint}.
\end{proof}

\begin{remark}[Perturbative consistency]
\label{rem:perturbative-consistency}
Using $\log c=\log(1+\alpha^2)=\alpha^2+O(\alpha^4)$,
\eqref{eq:Y-star-formula} expands as
\[
        Y_*(\alpha)
        \;=\;
        \tfrac{3}{2}\,\alpha^2\,\diag(\beta_1,\dots,\beta_n)
        +O(\alpha^4),
        \qquad
        \beta_i\;=\;\tfrac{n+1-2i}{n+1},
\]
recovering the leading-order shift that one computes via
$\mathcal{B}_0^{-1}(E_{11}-E_{nn})$ from \S\ref{sec:constant-term}. The
$\alpha^2$-direction is proportional to the vector
of unperturbed Puiseux exponents.
\end{remark}

\begin{remark}[Special values]
At $\alpha=0$, $c=1$ gives $Y_*(0)=0$ and $\lim_{t\downarrow 0}M_0(t)=I$,
recovering the unperturbed limit. For general $\alpha\ne 0$, the ratios
$r_{i+1}/r_i=c^{3/(n+1)}$ are constant, so the spectral gaps of $A_\alpha$
form a geometric progression.
\end{remark}

\subsection{Recentered equation}
\label{sec:extension}

Recall that the normalized rescaled equation \eqref{eq:normalized} is  $\widetilde{\mathcal G}(N,\alpha,t)=0$, where
\[
        \widetilde{\mathcal G}(N,\alpha,t) \;:=\;
        N^{-1} - F_n N F_n - \tfrac{\alpha}{c}\,t\,\mathcal P(N)
        - \tfrac{1}{c}\,t^2\,\mathcal Q(N) - c^{-1/2}\,\mathcal U(t).
\]

Fix $\alpha\in\bR$ and set $c:=1+\alpha^2$.
Let
  $Y_*(\alpha)$ be as in \eqref{eq:Y-star-formula}.
Define $N_*:=e^{Y_*(\alpha)}$, $S:=N_*^{1/2}$, and the
recentered variable $\widehat N:=S^{-1}NS^{-1}$.
Set    
\[
        \nu_i :=(N_*)_{ii}=c^{3\beta_i/2}>0.
\]
Note that the adjacent ratios are constant:
\begin{equation}\label{eq:nu-ratio}
        \sqrt{\nu_i/\nu_{i+1}}=c^{3/(2(n+1))}=:w,
        \qquad i=1,\ldots,n-1.
\end{equation}
and that $N_*\in\mathcal M$:
\begin{equation}
\label{equ:first_identity}
\nu_i\,\nu_{n+1-i}=1 \text{ (equivalently }y_{n+1-i}=-y_i\text{)}
\end{equation}

From this identity we have $\Theta(N_*^{1/2})=F_nN_*^{1/2}F_n
=\diag(\nu_{n+1-i}^{1/2})=N_*^{-1/2}$, i.e.\
\begin{equation}\label{eq:theta-sqrt}
        F_n S \,F_n=S^{-1}.
\end{equation}

Define
\[
        \tilde F:=S\,F_n\,S,
        \qquad
        \tilde G:=S\,G_n\,S.
\]

\begin{lemma}[Dressed shifts]\label{lem:dressed}
$\tilde F=F_n$ and $\tilde G=w\,G_n$.
\end{lemma}

\begin{proof}
For $\tilde F$: using \eqref{eq:theta-sqrt},
\[
        \tilde F\,F_n = S\,F_n \, S \,F_n
        =S S^{-1}\;=\;I,
\]
so $\tilde F=F_n^{-1}=F_n$.

For $\tilde G$: since $N_*$ is diagonal,
$(N_*^{1/2}G_nN_*^{1/2})_{ij}=\nu_i^{1/2}(G_n)_{ij}\nu_j^{1/2}$.
The nonzero entries of $G_n=\sum_{k=1}^{n-1}E_{n-k,k}$ live at
positions $(n-k,k)$ for $k\in\{1,\ldots,n-1\}$, so at every nonzero
position the prefactor is
\[
        \nu_{n-k}^{1/2}\,\nu_k^{1/2}
        \;=\; \nu_{k+1}^{-1/2}\,\nu_k^{1/2}
        \;=\; \sqrt{\nu_k/\nu_{k+1}}
        \;\stackrel{\eqref{eq:nu-ratio}}{=}\; w,
\]
where we used $\nu_{n-k}=\nu_{(n+1)-(k+1)}=\nu_{k+1}^{-1}$ from \eqref{equ:first_identity}. Hence $\tilde G=wG_n$.
\end{proof}

\begin{corollary}\label{cor:dressed-PQ}
For every $X \in M_n(\mathbb C)$,
\[
        \tilde FX \tilde G+\tilde G X \tilde F
        \;=\; w\,\mathcal P(X),
        \qquad
        \tilde G X \tilde G\;=\;w^2\,\mathcal Q(X).
\]
\end{corollary}

\subsubsection*{The recentred equation.}
Substitute $N= S \hat N S$ in
$\widetilde{\mathcal G}(N,\alpha,t)=0$ and multiply by $S$ on
both sides. Since $S N^{-1} S=\hat N^{-1}$ and
$S\Theta(N)S =\Theta(\hat N)$ (using $\tilde F=F_n$),
the limiting piece becomes $\hat N^{-1}-\Theta(\hat N)$, with the same
involution. Using Corollary~\ref{cor:dressed-PQ}:
\begin{align}
\label{eq:recentred-G}
        \widehat \GG (\hat N,\alpha,t)
        &:=\; N_*^{1/2}\widetilde{\mathcal G}(N,\alpha,t)N_*^{1/2}
        \\
        &=\; \hat N^{-1}-\Theta(\hat N)
        - \hat a\,t\,\mathcal P(\hat N)
        - \hat b\,t^2\,\mathcal Q(\hat N)
        - d\,\tilde{\mathcal U}(t), \notag
\end{align}
where
\begin{align*}
        \hat a & :=aw,
        \qquad
        \hat b:=bw^2,
        \\
        \tilde{\mathcal U}(t)
        &=\; N_*^{1/2}\mathcal U(t)N_*^{1/2}
        \;=\; \sum_{i=1}^n \nu_i\,t^{2(n+1-i)}E_{ii}.
\end{align*}

The equation \eqref{eq:recentred-G} is \emph{algebraically identical}
to the original $\widetilde{\mathcal G} = 0$ apart from the explicit scalar
replacements $a\mapsto\hat a$, $b\mapsto\hat b$, and the boundary term
$\mathcal U(t)\mapsto\tilde{\mathcal U}(t)$. We may therefore perform the LS reduction,
parametrising $\hat N=e^{\hat Y}+\hat\Phi(\hat Y,\alpha,t)$ with
$\hat Y\in E_-$, $\hat\Phi\in E_+$, exactly as in \S\S\ref{sec:LS-reduction}.
This yields $\hat{\mathscr R}$, divisible by $t^2$ by the analog of
Proposition~\ref{prop:Rhat-vanishes}, 
 and we set
$\hat{\mathfrak R}:=\hat{\mathscr R}/t^2$.

\subsubsection*{Consistency: $\hat{\mathfrak R}(0,\alpha,0)=0$.}

The verbatim analog of \eqref{eq:R2-at-0} reads
\[
        \hat{\mathfrak R}(0,\alpha,0)
        \;=\; \tfrac12\bigl(\hat a^{\,2}-\hat b+\hat d_{\text{eff}}\bigr)
        (E_{11}-E_{nn}),
\]
with $\hat d_{\text{eff}}=d\cdot\nu_n$ because the order-$t^2$ piece
of $\tilde{\mathcal U}$ is $\nu_n E_{nn}$.
Using $a^2-b=-1/c^2$,
\[
        \hat a^{\,2}-\hat b \;=\; w^2(a^2-b) \;=\; -\frac{w^2}{c^2}
        \;=\; -c^{\,3/(n+1)-2}
        \;=\; -c^{\,-(2n-1)/(n+1)},
\]
\[
        \hat d_{\text{eff}} \;=\; c^{-1/2}\cdot c^{-3(n-1)/(2(n+1))}
        \;=\; c^{-(2n-1)/(n+1)}.
\]
These cancel: $\hat{\mathfrak R}(0,\alpha,0)=0$. This is the expected
consistency: $\hat Y=0$ (i.e.\ $N=N_*$) is the actual limit point of
the perturbed solution.

\begin{lemma}[Recentred invertibility]
\label{lem:recentred-invertibility}
After applying the Lyapunov--Schmidt reduction to the recentred equation in
the variable $\widehat N$, let
\[
        \widehat{\mathfrak R}(\widehat Y,\alpha,t)=0
\]
be the corresponding second divided reduced equation. Then
\[
        D_{\widehat Y}\widehat{\mathfrak R}(0,\alpha,0)\colon E_-\to E_-
\]
is invertible for every real $\alpha$.
\end{lemma}
\begin{proof}


Let
\[
        \widehat{\mathcal B}_\alpha
        :=
        D_{\widehat Y}\widehat{\mathfrak R}(0,\alpha,0).
\]
We will compute the closed form of $\widehat{\mathcal B}_\alpha$.

The order-$t^2$ part of $d\,\widetilde{\mathcal U}(t)$ in the recentered equation \eqref{eq:recentred-G} is
\[
        \widehat d\,t^2E_{nn},
        \qquad
        \widehat d:=d\nu_n.
\]
Thus, for the computation of the second divided linearization, the relevant
truncated equation is
\[
        \widehat \GG(\widehat N,\alpha,t)
        =
        \widehat N^{-1}
        -
        \Theta(\widehat N)
        -
        \widehat a\,t\,\mathcal P(\widehat N)
        -
        \widehat b\,t^2\,\mathcal Q(\widehat N)
        -
        \widehat d\,t^2E_{nn}
        +
        O(t^3).
\]

The recentered equation has the same algebraic form as the original
normalized equation, so the computation of
Proposition~\ref{prop:Balpha-linearization} applies with
$(a,b,d)$ replaced by $(\hat a,\hat b,\hat d)$. The resulting
linearization is
\[
 \boxed{
  \widehat{\mathcal B}_\alpha(H)
  =
  \frac{\hat b-\hat a^{\,2}}{4}
  \bigl[
  2(L_nHU_n+U_nHL_n) - D_\partial H - HD_\partial
  \bigr]
  -
  \frac{\hat d}{4}
  (E_\partial H+HE_\partial).
  }
\]

Here
\[
        E_\partial:=E_{11}+E_{nn},
        \qquad
        D_\partial:=L_nU_n+U_nL_n.
\]

Now we compute the two scalar coefficients. Since
\[
        \widehat a=aw,
        \qquad
        \widehat b=bw^2,
\]
we have
\[
        \widehat b-\widehat a^2
        =
        w^2(b-a^2).
\]
But
\[
        b-a^2
        =
        \frac1c-\frac{\alpha^2}{c^2}
        =
        \frac{c-\alpha^2}{c^2}
        =
        \frac1{c^2}.
\]
Therefore
\[
        \widehat b-\widehat a^2
        =
        \frac{w^2}{c^2}.
\]
Using
\[
        w=c^{3/(2(n+1))},
\]
we get
\[
        \widehat b-\widehat a^2
        =
        c^{3/(n+1)-2}
        =
        c^{-(2n-1)/(n+1)}.
\]

On the other hand,
\[
        \widehat d=d\nu_n.
\]
Since
\[
        d=c^{-1/2},
        \qquad
        \nu_n=c^{-3(n-1)/(2(n+1))},
\]
we obtain
\[
        \widehat d
        =
        c^{-1/2}c^{-3(n-1)/(2(n+1))}
        =
        c^{-(2n-1)/(n+1)}.
\]
Thus
\[
        \boxed{
        \widehat b-\widehat a^2=\widehat d>0.
        }
\]

Since $\hat b-\hat a^{\,2}=\hat d>0$, the two scalar
coefficients in $\widehat{\mathcal B}_\alpha$ are both strictly
positive. The energy identity from
Step~3 of the proof of
Proposition~\ref{prop:Balpha-linearization} therefore gives
$\langle H,\widehat{\mathcal B}_\alpha(H)\rangle_{\mathrm{HS}}<0$
for every nonzero $H\in E_-$, by the same commutant
argument. Hence $\widehat{\mathcal B}_\alpha$ is strictly
negative definite and invertible.

\end{proof}

\subsection{Proof of Theorem \ref{theo:all-alpha-rescaled-holomorphicity}}
\label{sec:singularities_type_2}

\begin{proof}[Proof of Theorem~\ref{theo:all-alpha-rescaled-holomorphicity}]
Fix $\alpha\in\bR$ and set $c:=1+\alpha^2$.  Let
$Y_*:=Y_*(\alpha)$, $N_*:=e^{Y_*}$, $S:=N_*^{1/2}$,
where $Y_*(\alpha)$ is as in \eqref{eq:Y-star-formula}.
Introduce the recentred variable $\widehat N:=S^{-1}NS^{-1}$ and the
recentred equation
\[
        \widehat{\mathcal G}(\widehat N,\alpha,t)
        \;:=\;
        S\,\widetilde{\mathcal G}(S\widehat N S,\alpha,t)\,S
        \;=\; 0.
\]
By Lemma~\ref{lem:dressed}, $\widehat{\mathcal G}$ has the same
structural form as $\widetilde{\mathcal G}$: it decomposes as
$\widehat N^{-1} - \Theta(\widehat N)$ plus terms of order $t$ and
$t^2$ (see \eqref{eq:recentred-G}). In particular, at $t=0$ the
limiting equation is again $\widehat N^{-1}=\Theta(\widehat N)$, with
$\widehat N=I$ as the base-point solution.

The Lyapunov--Schmidt reduction of \S\ref{sec:LS-reduction} and the
Lyapunov identity of \S\ref{sec:lyapunov-divisibility} therefore apply verbatim to
$\widehat{\mathcal G}$, producing a doubly-reduced equation
$\widehat{\mathfrak R}(\widehat Y,\alpha,t)$ that is holomorphic near
$(\widehat Y,t)=(0,0)$.

By Proposition~\ref{prop:explicit-diagonal-solution},
$\mathfrak R(Y_*(\alpha),\alpha,0)=0$, which translates to
$\widehat{\mathfrak R}(0,\alpha,0)=0$.
By Lemma~\ref{lem:recentred-invertibility},
$D_{\widehat Y}\widehat{\mathfrak R}(0,\alpha,0)$ is invertible.
The implicit function theorem therefore gives a unique holomorphic
function $\widehat Y_\alpha(t)\in E_-$ with $\widehat Y_\alpha(0)=0$
and $\widehat{\mathfrak R}(\widehat Y_\alpha(t),\alpha,t)=0$. Setting
\[
        N_\alpha(t)
        \;:=\;
        S\bigl(e^{\widehat Y_\alpha(t)}
        +\widehat\Phi(\widehat Y_\alpha(t),\alpha,t)\bigr)S,
\]
we obtain a holomorphic solution of
$\widetilde{\mathcal G}(N_\alpha(t),\alpha,t)=0$ with
$N_\alpha(0)=N_*$.

The limiting value of $M_\alpha(t)=c^{-1/2}N_\alpha(t)$ is therefore
$c^{-1/2}N_*=c^{-1/2}e^{Y_*(\alpha)}$, whose $i$th diagonal entry is
\[
        c^{-1/2}\,
        c^{\frac32\cdot\frac{n+1-2i}{n+1}}
        \;=\;
        c^{1-\frac{3i}{n+1}}.
\]
This gives $M_\alpha(t)=M_\alpha^{(0)}+O(t)$ with the claimed
$M_\alpha^{(0)}$.

Finally, for small $t>0$ the constructed solution is positive, hence
accretive. By uniqueness of the accretive solution of Speicher's equation, it coincides
with $W_\alpha(u)=iG_\alpha(iu)$ via $W_\alpha(t^{n+1})=D(t)M_\alpha(t)D(t)$.
\end{proof}

\subsection{Proof of Theorem \ref{theo:typ_2_cells_sharp}}
\label{sec:proof_thm_case_II}

\begin{proof}
Let
\[
        p:=\frac{n-1}{n+1}.
\]
By Theorem~\ref{theo:all-alpha-rescaled-holomorphicity}, with
$t=u^{1/(n+1)}$,
\[
        W_\alpha(u)
        =
        D(t)M_\alpha(t)D(t),
        \qquad
        M_\alpha(t)=M_\alpha^{(0)}+O(t),
\]
where
\[
        M_\alpha^{(0)}
        =
        \operatorname{diag}
        \left(
        c^{1-\frac{3}{n+1}},
        c^{1-\frac{6}{n+1}},
        \ldots,
        c^{1-\frac{3n}{n+1}}
        \right),
        \qquad
        c:=1+\alpha^2.
\]
Since
\[
        D(t)_{ii}=t^{\gamma_i},
        \qquad
        \gamma_i=\frac{n-2i+1}{2},
\]
the $i$th diagonal contribution to $\operatorname{tr}_n W_\alpha(u)$ has
order
\[
        t^{2\gamma_i}
        =
        t^{n-2i+1}.
\]
The most singular term comes from $i=n$. Therefore
\[
        \operatorname{tr}_n W_\alpha(u)
        =
        \frac1n
        c^{1-\frac{3n}{n+1}}
        t^{-(n-1)}
        +
        O(t^{-(n-2)}).
\]
Since $t=u^{1/(n+1)}$, this becomes
\[
        \operatorname{tr}_n W_\alpha(u)
        =
        K_{\alpha,n}u^{-p}
        +
        O\left(u^{-p+\frac{1}{n+1}}\right),
\]
where
\[
        K_{\alpha,n}
        =
        \frac1n c^{1-\frac{3n}{n+1}}.
\]

On the other hand, because
\[
        W_\alpha(u)=iG_\alpha(iu),
\]
we have
\[
        \operatorname{tr}_n W_\alpha(u)
        =
        \int_{\mathbb R}
        \frac{u}{u^2+x^2}\,d\mu_\alpha(x).
\]
If
\[
        f_\alpha(x)
        \sim C_{\alpha,n}|x|^{-p}
        \qquad x\to0,
\]
then
\[
        \int_{\mathbb R}
        \frac{u}{u^2+x^2}f_\alpha(x)\,dx
        \sim
        C_{\alpha,n}u^{-p}
        \int_{\mathbb R}
        \frac{|y|^{-p}}{1+y^2}\,dy.
\]
For $0\le p<1$,
\[
        \int_{\mathbb R}
        \frac{|y|^{-p}}{1+y^2}\,dy
        =
        \frac{\pi}{\cos(\pi p/2)}.
\]
Hence
\[
        K_{\alpha,n}
        =
        C_{\alpha,n}\frac{\pi}{\cos(\pi p/2)}.
\]
Therefore
\[
        C_{\alpha,n}
        =
        \frac{\cos(\pi p/2)}{\pi}K_{\alpha,n}.
\]
Since
\[
        p=\frac{n-1}{n+1},
\]
we have
\[
        \cos\left(\frac{\pi p}{2}\right)
        =
        \cos\left(\frac{\pi(n-1)}{2(n+1)}\right)
        =
        \sin\left(\frac{\pi}{n+1}\right).
\]
Thus
\[
        C_{\alpha,n}
        =
        \frac{1}{n\pi}
        \sin\left(\frac{\pi}{n+1}\right)
        c^{1-\frac{3n}{n+1}},
\]
which is the claimed formula.
\end{proof}

\section{Type III cells: gauge reduction and spectral coincidence}
\label{sec:type_III}

Type~III cells are where the spectral classification is strictly coarser than the
algebraic one, and they supply two facts used elsewhere.
\emph{(i)} The complex argument of $\beta$ is spectrally invisible: a unitary gauge
removes the phase, so the scalar density depends only on $|\beta|$ and in fact
coincides with that of a Type~II cell of size~$m$ (half the size) with real
parameter $|\beta|$ (Theorem~\ref{thm:typeIII-scalar-law-caseII}). The gauge is a
congruence, so this is consistent with the congruence-invariance of the singular
exponent (Result~A).
\emph{(ii)} Nonetheless $\eta_{\mathrm{III}}$ is \emph{not} symmetrically equivalent
to that Type~II map, nor to any direct sum of Type~II maps: a congruence-invariant
of the Hermitian Kraus plane separates them
(Proposition~\ref{prop:typeIII-definite-kraus}). Together these give the sharp form
of Result~C.

Throughout, $F:=F_m=\sum_{i=1}^m E_{i,m+1-i}$ is the reversal and
$G:=G_m=\sum_{i=2}^m E_{i,m+2-i}$ the shifted reversal; both are real symmetric.

\begin{remark}[Convention for $G_m$]
\label{rem:Gm-convention}
The shifted reversal used here, $G_m=\sum_{i=2}^m E_{i,m+2-i}$, is supported on the band
$i+j=m+2$, whereas the Lancaster--Rodman cells of \S\ref{sec:LaRo_cells} (and the
Type~I/II analyses that follow) use
$G_m^{\mathrm{LR}}=\bigl[\begin{smallmatrix}F_{m-1}&0\\0&0\end{smallmatrix}\bigr]
=\sum_{i=1}^{m-1}E_{i,m-i}$, supported on $i+j=m$. A direct check gives
$G_m=F_mG_m^{\mathrm{LR}}F_m$, so conjugation by the real orthogonal involution $F_m$
(applied blockwise, i.e.\ by $F_m\oplus F_m$ on the Type~III cell and by $F_m$ on the
Type~II companion) is a congruence carrying the present cells to those of the earlier
sections. The two normalizations therefore define congruent pencils, with identical
scalar densities and lying in the same symmetric-scalability class; we adopt
$i+j=m+2$ only because it makes the gauge phases in \S\ref{sec:typeIII-coincidence}
telescope cleanly.
\end{remark}

\subsection{The cell and its covariance map} 
\label{sec:typeIII-setup}

Fix $m\ge2$ and $\beta\in\bC\setminus\R$, and write $\beta=\rho e^{i\theta}$ with
$\rho=|\beta|>0$. Put $H:=\beta F+G$, so $H^*=\bar\beta F+G$. The \emph{Type~III
cell} is the Hermitian binary pencil $A=A_1x_1+A_2x_2$ on $\bC^{2m}$ with
\[
        A_1=\begin{pmatrix}0&F\\F&0\end{pmatrix}=F_{2m},
        \qquad
        A_2=\begin{pmatrix}0&H\\H^*&0\end{pmatrix},
\]
and covariance map $\eta_{\mathrm{III}}(B)=A_1BA_1+A_2BA_2$, $B\in M_{2m}(\bC)$.
Its companion is the size-$m$ Type~II map with \emph{real} parameter $\rho$,
\[
        \eta_{\mathrm{II},\rho}(Y)=FYF+KYK,
        \qquad K:=\rho F+G .
\]

\begin{exa}[$m=2$, $\beta=i$]
\label{exa:typeIII-m2}
Here $\rho=1$, $\theta=\tfrac\pi2$, $c:=1+|\beta|^2=2$, and
$H=iF_2+G_2=\begin{psmallmatrix}0&i\\i&1\end{psmallmatrix}$, so
\[
        A_1=\begin{pmatrix}0&0&0&1\\0&0&1&0\\0&1&0&0\\1&0&0&0\end{pmatrix},
        \qquad
        A_2=\begin{pmatrix}0&0&0&i\\0&0&i&1\\0&-i&0&0\\-i&1&0&0\end{pmatrix}.
\]
The companion Type~II map has size~$2$, parameter $\rho=1$, and
$K=F_2+G_2=\begin{psmallmatrix}0&1\\1&1\end{psmallmatrix}$. We return to this in
\S\ref{sec:typeIII-coincidence} and \S\ref{sec:typeIII-obstruction}.
\end{exa}

\subsection{Spectral coincidence}
\label{sec:typeIII-coincidence}

\begin{theo}[Type~III reduces to Type~II with parameter $|\beta|$]
\label{thm:typeIII-scalar-law-caseII}
The matrix semicircular elements of $\eta_{\mathrm{III}}$ and of
$\eta_{\mathrm{II},\rho}$ ($\rho=|\beta|$) have the same scalar distribution:
their normalized scalar Cauchy transforms agree on $\bC^+$, and hence
\[
        f_{\mathrm{III},\beta}(x)=f_{\mathrm{II},|\beta|}(x)
        \qquad\text{for a.e.\ }x\in\R .
\]
\end{theo}

\begin{proof}
Let $W=W(u)$ be the unique accretive solution of
$W^{-1}=uI_{2m}+\eta_{\mathrm{III}}(W)$, $\Re u>0$; recall $W(u)=i\mathcal G(iu)$,
so $\tr_{2m}W(u)$ is the normalized scalar Cauchy transform of
$\mu_{\mathrm{III},\beta}$ at $iu$.

\emph{Step 1: $W$ is block diagonal.}
Let $J=I_m\oplus(-I_m)$. The $A_k$ are block off-diagonal, so $JA_kJ=-A_k$ and
$\eta_{\mathrm{III}}(JBJ)=J\eta_{\mathrm{III}}(B)J$; as $J(uI)J=uI$, the matrix
$JWJ$ is also accretive, so by uniqueness $W=JWJ=:X\oplus V$, with
\[
        X^{-1}=uI+FVF+HVH^*,\qquad V^{-1}=uI+FXF+H^*XH .
\]

\emph{Step 2: the gauge removes the phase of $\beta$.}
Let $Q=\diag(q_1,\dots,q_m)$, $q_j=\exp\!\big(\!-i(j-\tfrac{m+1}{2})\theta\big)$,
and $\mathcal Q:=Q\oplus\overline Q$ (unitary). Since $F$ is supported on
$i+j=m+1$ and $G$ on $i+j=m+2$, the phases telescope:
\[
        Q^*F\overline Q=F,\quad Q^*G\overline Q=e^{i\theta}G
        \ \Longrightarrow\
        Q^*H\overline Q=e^{i\theta}K,\quad Q^TH^*Q=e^{-i\theta}K .
\]
Conjugating the equation by $\mathcal Q^*(\cdot)\mathcal Q$ and writing
$\widehat W:=\mathcal Q^*W\mathcal Q$ (still block diagonal) gives
$\widehat W^{-1}=uI+\sum_k(\mathcal Q^*A_k\mathcal Q)\,\widehat W\,(\mathcal Q^*A_k\mathcal Q)$,
where $\mathcal Q^*A_1\mathcal Q=\widehat A_1:=\begin{psmallmatrix}0&F\\F&0\end{psmallmatrix}$ and
$\mathcal Q^*A_2\mathcal Q=\begin{psmallmatrix}0&e^{i\theta}K\\e^{-i\theta}K&0\end{psmallmatrix}$.
On a block-diagonal argument the two phases cancel, so this last term equals
$\widehat A_2\widehat W\widehat A_2$ with
$\widehat A_2:=\begin{psmallmatrix}0&K\\K&0\end{psmallmatrix}$. Hence $\widehat W$
solves the \emph{real}-parameter equation
$\widehat W^{-1}=uI+\widehat\eta(\widehat W)$,
$\widehat\eta(B):=\widehat A_1B\widehat A_1+\widehat A_2B\widehat A_2$.

\emph{Step 3: the swap collapses the two blocks.}
As $F,K$ are real symmetric, the block swap
$R=\begin{psmallmatrix}0&I_m\\I_m&0\end{psmallmatrix}$ satisfies
$R\widehat A_kR=\widehat A_k$, so $R$ commutes with $\widehat\eta$ and with $uI$.
By uniqueness, $R\widehat W R=\widehat W$, i.e.\ the two diagonal blocks coincide,
$\widehat W=Y\oplus Y$. The surviving block satisfies
$Y^{-1}=uI+FYF+KYK=uI+\eta_{\mathrm{II},\rho}(Y)$, the Type~II equation; by
uniqueness $Y=W_{\mathrm{II},\rho}(u)$.

\emph{Conclusion.}
Since $\mathcal Q$ is unitary, $\Tr W=\Tr\widehat W=2\Tr Y$, so
\[
        \tr_{2m}W(u)=\tfrac1{2m}\Tr W=\tfrac1m\Tr Y=\tr_m W_{\mathrm{II},\rho}(u).
\]
The normalized scalar Cauchy transforms agree on $\bC^+$; as each is the Cauchy
transform of a compactly supported probability measure, the scalar measures
coincide, and the density identity follows by Stieltjes inversion.
\end{proof}

Thus the size-$4$ cell of Example~\ref{exa:typeIII-m2} ($\beta=i$) has exactly the
scalar density of the size-$2$ Type~II cell with parameter~$1$.

\begin{coro}[Leading singularity of a Type~III cell]
\label{cor:typeIII-from-caseII}
For $\beta\in\bC\setminus\R$,
\begin{align*}
        f_{\mathrm{III},\beta}(x)
        &=C_{\beta,m}\,|x|^{-\frac{m-1}{m+1}}+o\!\big(|x|^{-\frac{m-1}{m+1}}\big),
        \quad x\to0,
        \\
        C_{\beta,m}
        &=\frac{1}{m\pi}\sin\!\Big(\tfrac{\pi}{m+1}\Big)
        \big(1+|\beta|^2\big)^{-\frac{2m-1}{m+1}} .
\end{align*}
\end{coro}

\begin{proof}
By Theorem~\ref{thm:typeIII-scalar-law-caseII} the density equals that of the
size-$m$ Type~II cell with parameter $\alpha=|\beta|$; apply
Theorem~\ref{theo:typ_2_cells_sharp} with $n=m$, $\alpha=|\beta|$, and use
$1-\tfrac{3m}{m+1}=-\tfrac{2m-1}{m+1}$.
\end{proof}

In particular the size-$2m$ Type~III cell carries the singularity exponent of a
size-$m$ Type~II cell, $\tfrac{m-1}{m+1}$ --- not $\tfrac{2m-1}{2m+1}$. The full
solution is $W(u)=\mathcal Q\,(Y\oplus Y)\,\mathcal Q^*
=QYQ^*\oplus\overline Q\,Y\,Q^T$ with $Y=W_{\mathrm{II},\rho}(u)$.

\subsection{The covariance maps are not symmetrically scalable} 
\label{sec:typeIII-obstruction}

Recall (Definition~\ref{defi_sym_scalability}) that $\eta_A,\eta_B$ are symmetrically
scalable if $\eta_B(X)=b\,\eta_A(b^*Xb)\,b^*$ for some invertible $b$.
The
separating invariant is the real span of the Hermitian Kraus matrices; we first verify
this span is a property of the \emph{map}. Write a self-adjoint covariance map as
$\eta(X)=\sum_{i=1}^r A_iXA_i$ with $A_i=A_i^*$, and call the family \emph{minimal} if
the $A_i$ are linearly independent over $\bC$ (so $r$ is the Choi rank).

\begin{lemma}[Unitary freedom of minimal Kraus families]
\label{lem:kraus-unitary-freedom}
Let $\Phi(X)=\sum_{i=1}^{r}A_iXA_i^{*}=\sum_{j=1}^{s}M_jXM_j^{*}$ for all $X\in M_n(\bC)$,
and suppose $\{A_i\}$ and $\{M_j\}$ are each linearly independent over $\bC$. Then
$r=s$, and there is a unitary $w=(w_{ij})\in U(r)$ with $A_i=\sum_{j}w_{ij}M_j$.
\end{lemma}

\begin{proof}
Let $\flat\colon M_n(\bC)\to\bC^{n^2}$, $\flat(A):=\sum_{k}e_k\otimes Ae_k$; this is a
linear isomorphism. Writing $E_{kl}=e_ke_l^{*}$ and $A_iE_{kl}A_i^{*}=(A_ie_k)(A_ie_l)^{*}$,
the Choi matrix of $\Phi$ satisfies
\[
        C_\Phi:=\sum_{k,l}E_{kl}\otimes\Phi(E_{kl})
        =\sum_{i}\flat(A_i)\,\flat(A_i)^{*}
        =\sum_{j}\flat(M_j)\,\flat(M_j)^{*}.
\]
Set $\mathsf A:=[\,\flat(A_1)\;\cdots\;\flat(A_r)\,]$ and
$\mathsf B:=[\,\flat(M_1)\;\cdots\;\flat(M_s)\,]$, so $C_\Phi=\mathsf A\mathsf A^{*}=\mathsf B\mathsf B^{*}$.
Since $\flat$ is an isomorphism, linear independence of the families makes $\mathsf A,\mathsf B$
of full column rank; hence $r=\operatorname{rank}C_\Phi=s$, and
$\operatorname{range}\mathsf A=\operatorname{range}C_\Phi=\operatorname{range}\mathsf B$.
As $\mathsf B$ is injective, each column of $\mathsf A$ lies in $\operatorname{range}\mathsf B$
and so $\mathsf A=\mathsf B\,w^{T}$ for a unique $w^{T}\in M_r(\bC)$. Then
$\mathsf B(w^{T}(w^{T})^{*}-I)\mathsf B^{*}=0$, and injectivity of $\mathsf B$ forces
$w^{T}(w^{T})^{*}=I$, i.e.\ $w\in U(r)$. Reading off columns,
$\flat(A_i)=\sum_j w_{ij}\,\flat(M_j)=\flat\!\big(\sum_j w_{ij}M_j\big)$, and injectivity of
$\flat$ gives $A_i=\sum_j w_{ij}M_j$.
\end{proof}

\begin{lemma}[The Hermitian Kraus plane is well defined]
\label{lem:kraus-plane-welldef}
If $\{A_i\}_{i=1}^r$ and $\{M_j\}_{j=1}^r$ are two minimal Hermitian Kraus families for
the same map $\eta$, then $\operatorname{span}_\R\{A_i\}=\operatorname{span}_\R\{M_j\}$.
We write $\mathcal K_{\eta,\mathrm{sa}}\subseteq\operatorname{Herm}(n)$ for this subspace.
\end{lemma}

\begin{proof}[Proof of Lemma~\ref{lem:kraus-plane-welldef}]
Both families have the common length $r=\operatorname{rank}C_\eta$. By
Lemma~\ref{lem:kraus-unitary-freedom}, $A_i=\sum_j w_{ij}M_j$ with $w\in U(r)$. Taking
adjoints and using $A_i^{*}=A_i$, $M_j^{*}=M_j$ gives
$\sum_j w_{ij}M_j=\sum_j\overline{w_{ij}}\,M_j$; linear independence of $\{M_j\}$ forces
$w_{ij}\in\R$, so $w$ is real orthogonal and the real spans coincide.
\end{proof}

\begin{lemma}[Scalability acts by congruence on the plane]
\label{lem:plane-scaling}
If $\eta_B(X)=b\,\eta_A(b^*Xb)\,b^*$ with $b$ invertible, then
$\mathcal K_{\eta_B,\mathrm{sa}}=b\,\mathcal K_{\eta_A,\mathrm{sa}}\,b^*$. Consequently the
property ``every nonzero element of the plane is invertible''
($\mathcal K_{\eta,\mathrm{sa}}\cap\{\det=0\}=\{0\}$) is a symmetric-scalability invariant.
\end{lemma}

\begin{proof}
For a minimal Hermitian family $\{A_i\}$ of $\eta_A$ we have
\[
\eta_B(X)=\sum_i(bA_ib^*)\,X\,(bA_ib^*),
\] each $bA_ib^*$ Hermitian (as $A_i=A_i^*$, so $(bA_ib^*)^*=bA_i^*b^*=bA_ib^*$) and
the family minimal since $X\mapsto bXb^*$ is invertible. By
Lemma~\ref{lem:kraus-plane-welldef},
$\mathcal K_{\eta_B,\mathrm{sa}}=\operatorname{span}_\R\{bA_ib^*\}=b\,\mathcal K_{\eta_A,\mathrm{sa}}\,b^*$.
As $\det(bXb^*)=|\det b|^2\det X$ with $\det b\neq0$, $bXb^*$ is singular iff $X$ is.
\end{proof}

%
%

We call a plane satisfying this condition \emph{nonsingular}. (This is regularity, not
sign-definiteness: every nonzero element of the Type~III plane is in fact an indefinite
Hermitian matrix, with symmetric spectrum.)

\begin{lemma}[Determinant of $\alpha F+yG$]
\label{lem:det-alphaF-yG}
For all $m\ge2$ and $\alpha,y\in\bC$, $\det(\alpha F_m+yG_m)=(-1)^{m(m-1)/2}\alpha^m$;
in particular $\alpha F+yG$ is singular iff $\alpha=0$.
\end{lemma}

\begin{proof}
Reversing columns ($j\mapsto m+1-j$) sends the anti-diagonal of $F$ to the main diagonal
(entries $\alpha$) and the shifted anti-diagonal of $G$ to the sub-diagonal (entries
$y$), a lower bidiagonal matrix of determinant $\alpha^m$; the reversal has sign
$(-1)^{m(m-1)/2}$.
\end{proof}

\begin{propo}[The Type~III Kraus plane is nonsingular]
\label{prop:typeIII-definite-kraus}
For $\beta\in\bC\setminus\R$, every nonzero real combination $xA_1+yA_2$ ($x,y\in\R$) is
invertible; i.e.\ $\mathcal K_{\eta_{\mathrm{III}},\mathrm{sa}}$ is nonsingular.
\end{propo}

\begin{proof}
$\{A_1,A_2\}$ is minimal ($F$ and $H=\beta F+G$ are linearly independent, $G\neq0$), so
it represents $\mathcal K_{\eta_{\mathrm{III}},\mathrm{sa}}$. A real combination is
$xA_1+yA_2=\begin{psmallmatrix}0&B\\B^*&0\end{psmallmatrix}$ with $B=(x+y\beta)F+yG$, and
is singular iff $B$ is. By Lemma~\ref{lem:det-alphaF-yG} ($\alpha=x+y\beta$), $B$ is
singular iff $x+y\beta=0$; for real $x,y$ this forces $y\operatorname{Im}\beta=0$ and
$x+y\operatorname{Re}\beta=0$, hence $y=0$, $x=0$ since $\operatorname{Im}\beta\neq0$.
\end{proof}

\begin{theo}[Type~III is not symmetrically scalable to real Type~II sums]
\label{thm:typeIII-not-scalable}
For $\beta\in\bC\setminus\R$ and any $m\ge2$, $\eta_{\mathrm{III}}$ is not symmetrically
scalable to any direct sum of size-$m$ Type~II covariance maps with real parameters; in
particular not to $\eta_{\mathrm{II},|\beta|}\oplus\eta_{\mathrm{II},|\beta|}$, which has
the same scalar density (Theorem~\ref{thm:typeIII-scalar-law-caseII} and the direct-sum
law).
\end{theo}

\begin{proof}
Such a sum $\eta_\oplus$ has minimal Hermitian family $B_1=\bigoplus_k F$,
$B_2=\bigoplus_k(\alpha_kF+G)$ ($\alpha_k\in\R$). The element $-\alpha_1B_1+B_2$ has
first block $G\neq0$, hence is a nonzero singular member of
$\mathcal K_{\eta_\oplus,\mathrm{sa}}$ (Lemma~\ref{lem:det-alphaF-yG}); so that plane is
not nonsingular. The Type~III plane is nonsingular
(Proposition~\ref{prop:typeIII-definite-kraus}), and nonsingularity is a
symmetric-scalability invariant (Lemma~\ref{lem:plane-scaling}); hence the two maps are
not symmetrically scalable.
\end{proof}

\begin{remark}[Sharp form of Result~C]
With Theorem~\ref{thm:typeIII-scalar-law-caseII}, this gives for every $m\ge2$ two
covariance maps on $M_{2m}(\bC)$ with identical scalar densities that are not
symmetrically scalable: the spectral classification is strictly coarser than the
symmetric-scalability classification.
\end{remark}

\begin{exa}[Example~\ref{exa:typeIII-m2}, continued]
For $\beta=i$, $m=2$: a real combination of $A_1,A_2$ is block-off-diagonal with corner
$B=(x+iy)F_2+yG_2$, $\det B=-(x+iy)^2$, vanishing only at $(x,y)=(0,0)$ --- the Type~III
plane is nonsingular. The size-$4$ map $\eta_{\mathrm{II},1}\oplus\eta_{\mathrm{II},1}$,
of the same scalar density, has family $F\oplus F$, $K\oplus K$ ($K=F_2+G_2$) with
\[
        -(F\oplus F)+(K\oplus K)=G_2\oplus G_2=\diag(0,1,0,1)
\]
a nonzero singular element. So the two are not symmetrically scalable.
\end{exa}

\section{Classification of singularities for Hermitian binary pencils}
\label{sec:classification-synthesis}

We collect the cell computations of \S\S\ref{sec:type_I}--\ref{sec:type_III} into a single statement. For a cell,
define its \emph{effective chain length} $n^*$: $n^*=n$ for a \TypeI{} or \TypeII{}
cell of size $n$, and $n^*=m$ for a \TypeIII{} cell of size $2m$. (For \TypeIII{}
this is the chain length of the \TypeII{} cell to which its scalar density reduces,
Corollary~\ref{cor:typeIII-from-caseII}.) A cell is LR-semisimple iff $n^*=1$.

\begin{theo}[Leading singularity of a regular Hermitian binary pencil]
\label{thm:classification-synthesis}
Let $A=A_1x_1+A_2x_2$ be a regular Hermitian binary pencil, with Lancaster--Rodman
cells $B^{(1)},\dots,B^{(k)}$ of effective chain lengths $n_1^*,\dots,n_k^*$, and set
$N^*:=\max_j n_j^*$. Then $\mu_{S_A}$ has no atom at $0$ and leading Puiseux exponent
there equal to $-\alpha_*$, where \[
        \alpha_*=\frac{N^*-1}{N^*+1},
        \qquad\text{i.e.}\qquad
        f_A(x)=C\,|x|^{-\frac{N^*-1}{N^*+1}}+o\!\big(|x|^{-\frac{N^*-1}{N^*+1}}\big),
        \quad x\to0,\ C>0.
\]
In particular $\mu_{S_A}$ is regular at $0$ ($\alpha_*=0$) iff $N^*=1$, i.e.\ iff $A$
is LR-semisimple.
\end{theo}

\begin{proof}
Since $A$ is regular, the matrix semicircle $S_A$ is full
(Proposition~\ref{prop:regular-full}), so by Proposition~\ref{prop:reduction-to-cells} $\alpha_*=\max_j\alpha_j$, where
$\alpha_j$ is the singularity exponent of the cell $B^{(j)}$ (congruence-invariance
and the positive, hence cancellation-free, direct-sum law). The cell exponents are $(n-1)/(n+1)$ for \TypeI{} and \TypeII{} cells of size $n$
(Theorems~\ref{theo:typ_1_cells_sharp} and~\ref{theo:typ_2_cells_sharp})
and $(m-1)/(m+1)$ for \TypeIII{} cells of size $2m$
(Corollary~\ref{cor:typeIII-from-caseII}); in every case $\alpha_j=(n_j^*-1)/(n_j^*+1)$.
Since $g(n)=\frac{n-1}{n+1}=1-\frac{2}{n+1}$ is strictly increasing on $n\ge1$,
$\max_j\alpha_j=g(\max_j n_j^*)=g(N^*)$. The semisimplicity equivalence is
$N^*=1\iff$ every $n_j^*=1\iff$ every cell is LR-semisimple.
\end{proof}

\begin{remark}[What $N^*$ measures: deviation from semisimplicity]
\label{rem:Nstar-geometry}
The invariant $N^*$ has a coordinate-free meaning in the indefinite geometry from
which the Lancaster--Rodman form arises \cite[Thm.~5.1.1]{GLR2005}. Since $A$ is
regular, $H:=A_1+sA_2$ is invertible for all but finitely many real $s$ (take $s=0$
when $A_1$ is invertible); put $T:=H^{-1}A_2$. Then $HT=A_2$ is Hermitian, so $T$ is
\emph{$H$-selfadjoint}, i.e.\ self-adjoint for the indefinite inner product
$[x,y]:=y^*Hx$, and the Lancaster--Rodman cells are precisely the Jordan blocks of $T$
in this metric \cite[Thm.~5.1.1, Prop.~4.2.3]{GLR2005} --- equivalently the blocks of
the pair canonical form \cite{lancaster_rodman2005} via $T=H^{-1}A_2$ --- the sign
characteristic recording the sign of $[\cdot,\cdot]$ on each block; the sizes $n_j^*$
are independent of the (generic) $s$.

For a self-adjoint operator in a \emph{definite} metric the spectral theorem forbids
Jordan blocks; a block of size $n>1$ occurs only because $[\cdot,\cdot]$ degenerates
along the corresponding root space. On a cell of size $n$ the metric is the reversal
$H=\pm F_n$ \cite[Ex.~4.2.1]{GLR2005}, so the unique $T$-invariant flag
$\langle v_1,\dots,v_k\rangle$ has $[v_i,v_j]=\pm\delta_{i+j,\,n+1}$ and is totally
$[\cdot,\cdot]$-neutral iff $2k\le n$; hence the largest neutral $T$-invariant subspace
inside the cell has dimension $\lfloor n/2\rfloor$ (cf.\ the neutral-subspace bound
\cite[Thm.~2.3.4]{GLR2005}). For a \TypeIII{} cell ($\dim 2m$, a non-real eigenvalue
pair, whose root subspaces are automatically $[\cdot,\cdot]$-neutral
\cite[Thm.~5.1.1]{GLR2005}) this gives $\lfloor 2m/2\rfloor=m$. Thus, up to the parity
correction, $N^*$ is twice the largest dimension of a $T$-invariant subspace on which
$H$ vanishes identically: \emph{$N^*$ measures how deeply the eigendirections of the
pencil sink into the neutral cone of the indefinite metric.} LR-semisimplicity
($N^*=1$) is exactly the absence of any nonzero $T$-invariant neutral subspace ---
nondegeneracy of the metric on every eigendirection.

Writing $d:=N^*-1$ for the nilpotent depth of the largest block --- the \emph{defect of
semisimplicity}, algebraic minus geometric multiplicity --- this reads
\[
        \alpha_*=\frac{N^*-1}{N^*+1}=\frac{d}{d+2}.
\]
The singularity exponent is the defect filtered through the quadratic term of Speicher's
equation $zG=I+\eta(G)G$: the ``$+2$'' is the degree that term adds to the characteristic
equation over the depth $d$, interpolating from $\alpha_*=0$ at $d=0$ (smooth) to
$\alpha_*\to1$ as $d\to\infty$.
\end{remark}

\begin{remark}[The exponent classification is strictly coarser]
The leading exponent depends on the cells only through $N^*=\max_j n_j^*$ --- not on
their types, parameters, or sign characteristic, and not on whether a given chain
length is realized by a \TypeII{} cell of size $m$ or a \TypeIII{} cell of size $2m$.
Thus distinct LR data can give the same scalar density. The sharpest instance is
\S8: $\eta_{\mathrm{II},|\beta|}$ and $\eta_{\mathrm{III},\beta}$ (and direct sums
of them) have identical scalar densities yet are not symmetrically scalable
(Theorem~\ref{thm:typeIII-not-scalable}). When several cells attain $N^*$, the
constant $C$ is the corresponding $n_j^*/n$-weighted sum of the cell constants (which
depend on the \TypeII{}/\TypeIII{} parameter through $1+\alpha^2$, $1+|\beta|^2$); the
sum is positive, with no cancellation.
\end{remark}

\appendix


\section{Algebraicity of the matrix Cauchy transform} 
\label{sec:algebraicity}
\begin{proof}
\textbf{Step~1 (Polynomial system).}
Evaluating Speicher's equation $zG(z)=I_n+\eta(G(z))G(z)$
entrywise gives, for each $1\le p,q\le n$,
\[
zG_{pq}(z)-\delta_{pq}
-\sum_{\ell,k,m}\Bigl(\sum_{j}(A_j)_{p\ell}\,\overline{(A_j)_{k m}}\Bigr)
  G_{\ell m}(z)\,G_{kq}(z)=0.
\]
Set $R:=\mathbb{C}[z,(w_{\alpha\beta})_{1\le \alpha,\beta\le n}]$
and, for $1\le p,q\le n$, define
\begin{align*}
\Phi_{pq}(z,W)&:=
zw_{pq}-\delta_{pq}
-\sum_{\ell,k,m}T_{(p,\ell),(k,m)}\,w_{\ell m}w_{kq}\;\in R,
\\
T_{(p,\ell),(k,m)}&:=\sum_{j}(A_j)_{p\ell}\,\overline{(A_j)_{k m}}.
\end{align*}
Each $\Phi_{pq}$ has degree~$1$ in $z$ and total degree~$2$ in the
$n^2$ variables $(w_{\alpha\beta})$. Let
$I:=\langle \Phi_{pq}:1\le p,q\le n\rangle\subset R$ and
$V:=V(I)\subset\mathbb{A}^{1+n^2}_{(z,W)}$. By construction,
$(z,G(z))\in V$ for every $z\in\mathbb{C}^+$.

\smallskip
\textbf{Step~2 (Generic fiber is finite, via the implicit function
theorem).}
Fix $t>\sqrt{2\|\eta\|}$ and set $z_0:=it$. Let $W_0:=G(it)$, the physical
accretive solution on the imaginary axis. The resolvent bound
$\|G(iu)\|\le 1/u$ gives $\|W_0\|\le 1/t$. We compute the Jacobian of
the map $W\mapsto(\Phi_{pq}(z_0,W))_{p,q}$ at $W_0$:
differentiating $\Phi_{pq}(z_0,W)=z_0 w_{pq}-\delta_{pq}
-(\eta(W)W)_{pq}$ with respect to $w_{\alpha\beta}$, and using the
product rule,
\[
J_{(p,q),(\alpha,\beta)}(W_0)
=z_0\,\delta_{p\alpha}\delta_{q\beta}
-\sum_{k}T_{(p,\alpha),(k,\beta)}\,(W_0)_{kq}
-\eta(W_0)_{p\alpha}\,\delta_{q\beta}.
\]
View $J(W_0)$ as a linear operator on $M_n(\mathbb{C})\simeq
\mathbb{C}^{n^2}$. The first term is $z_0$ times the identity, and
the last two terms have operator norm bounded by $2\|\eta\|\,\|W_0\|
\le 2\|\eta\|/t$. Since $t>\sqrt{2\|\eta\|}$ gives $2\|\eta\|/t<t=|z_0|$, the Jacobian $J(W_0)$ is
invertible.

By the holomorphic implicit function theorem, the zero set
$\{\Phi_{pq}=0\}$ is a $1$-dimensional complex submanifold of
$\mathbb{C}^{1+n^2}$ near $(z_0,W_0)$, and the projection
$\pi:V\to\mathbb{A}^1_z$ is locally an isomorphism there.
Consequently, the unique irreducible component $V^{(0)}$ of $V$
passing through $(z_0,W_0)$ has dimension exactly~$1$, and
$\pi\!\restriction_{V^{(0)}}$ is dominant with $0$-dimensional
generic fiber.

Moreover, the analytic curve
$\gamma:\mathbb{C}^+\to\mathbb{C}^{1+n^2}$, $z\mapsto(z,G(z))$
lies in $V$ and passes through $(z_0,W_0)$; since $\gamma$ is
non-constant (its first coordinate is $z$), its image lies in
$V^{(0)}$. In particular,
\begin{equation}
\label{eq:Gamma-in-V0}
(z,G(z))\in V^{(0)}\quad\text{for every }z\in\mathbb{C}^+.
\end{equation}

\smallskip
\textbf{Step~3 (Projection to the $(z,w_{pq})$-plane).}
Fix $(p,q)$ and consider the projection
\[
\pi_{pq}:\mathbb{A}^{1+n^2}\longrightarrow \mathbb{A}^2_{(z,w)},
\qquad \bigl(z,(w_{\alpha\beta})\bigr)\longmapsto (z,w_{pq}).
\]
Let $Y_{pq}:=\overline{\pi_{pq}(V^{(0)})}^{\,\mathrm{Zar}}
\subset \mathbb{A}^2$ be the Zariski closure of the image. Then
$Y_{pq}$ is an irreducible algebraic subvariety of $\mathbb{A}^2$
with
\[
\dim Y_{pq}\le \dim V^{(0)}=1.
\]
By~\eqref{eq:Gamma-in-V0}, $Y_{pq}$ contains the image of the
analytic curve
\[
z\in\mathbb{C}^+\longmapsto (z,G_{pq}(z))\in\mathbb{C}^2,
\]
whose first coordinate is non-constant; in particular $Y_{pq}$ is
neither empty nor a single point, so
\[
\dim Y_{pq}=1.
\]
A $1$-dimensional irreducible algebraic subvariety of $\mathbb{A}^2$
is the zero set of a single irreducible polynomial
\[
P_{pq}(z,w)\in\mathbb{C}[z,w]\setminus\{0\}.
\]

\smallskip
\textbf{Step~4 (Conclusion).}
For every $z\in\mathbb{C}^+$, $(z,G_{pq}(z))\in Y_{pq}=V(P_{pq})$, so
\[
P_{pq}\bigl(z,G_{pq}(z)\bigr)=0\qquad\text{for all } z\in\mathbb{C}^+.
\]
This is the asserted algebraic relation, so $G_{pq}$ is algebraic over
$\mathbb{C}(z)$.

Finally, algebraic functions over $\mathbb{C}(z)$ form a field; hence
any $\mathbb{C}$-linear combination of algebraic functions is
algebraic, and $H(z)=\tfrac{1}{n}\sum_p G_{pp}(z)$ is algebraic over
$\mathbb{C}(z)$.
\end{proof}

\section{Tauberian lemmas for Poisson asymptotics}\label{sec:tauberian} 
\subsection{Proof of the Puiseux-to-Poisson lemma}
\label{appx:puiseux_poisson} 
\begin{proof}
We may assume without loss of generality that the Poisson integral
converges for all small $\varepsilon > 0$; in the application,
$f$~is a spectral density satisfying $\int_\bR f < \infty$.

\medskip\noindent\textbf{Step~1: Reduction to half-lines.}
Write
\[
\frac{1}{\pi}\int_\bR \frac{\varepsilon\,f(x)}{x^2+\varepsilon^2}\,dx
\;=\;
I_+(\varepsilon) + I_-(\varepsilon),
\qquad
I_\pm(\varepsilon)
\;=\;
\frac{1}{\pi}\int_0^\infty
  \frac{\varepsilon\,f(\pm x)}{x^2+\varepsilon^2}\,dx.
\]
We treat $I_+$; the analysis of~$I_-$ is identical with $c_-$
replacing~$c_+$.

\medskip\noindent\textbf{Step~2: Rescaling.}
The substitution $x = \varepsilon u$ gives
\[
I_+(\varepsilon)
\;=\;
\frac{1}{\pi}\int_0^\infty \frac{f(\varepsilon u)}{u^2 + 1}\,du.
\]
Fix $\eta > 0$ small enough that
$|f(x) - c_+ x^{-\alpha}| \le C x^{-\alpha+\delta}$
for all $x \in (0,\eta]$, and split
\[
I_+(\varepsilon) = J_1(\varepsilon) + J_2(\varepsilon),
\qquad
J_1 = \frac{1}{\pi}\int_0^{\eta/\varepsilon} \frac{f(\varepsilon u)}{u^2+1}\,du,
\quad
J_2 = \frac{1}{\pi}\int_{\eta/\varepsilon}^\infty \frac{f(\varepsilon u)}{u^2+1}\,du.
\]

\medskip\noindent\textbf{Step~3: Tail estimate.}
Returning to the original variable,
\[
J_2(\varepsilon)
= \frac{1}{\pi}\int_\eta^\infty
  \frac{\varepsilon\,f(x)}{x^2+\varepsilon^2}\,dx
\le \frac{\varepsilon}{\pi\eta^2}
  \int_\eta^\infty f(x)\,dx
= O(\varepsilon).
\]

\medskip\noindent\textbf{Step~4: Main term.}
On the interval $u \in (0,\,\eta/\varepsilon]$ we have
$\varepsilon u \in (0,\eta]$, so the asymptotic hypothesis applies:
\[
f(\varepsilon u) = c_+\,(\varepsilon u)^{-\alpha}
  + r(\varepsilon u),
\qquad |r(\varepsilon u)| \le C\,(\varepsilon u)^{-\alpha+\delta}.
\]
Thus $J_1 = M + R$ where
\[
M = \frac{c_+\,\varepsilon^{-\alpha}}{\pi}
  \int_0^{\eta/\varepsilon} \frac{u^{-\alpha}}{u^2+1}\,du,
\qquad
|R| \le \frac{C\,\varepsilon^{-\alpha+\delta}}{\pi}
  \int_0^{\eta/\varepsilon} \frac{u^{-\alpha+\delta}}{u^2+1}\,du.
\]

\emph{Evaluation of~$M$.}
The beta-integral identity (substitution $u^2 = t$, then the
reflection formula) gives
\begin{equation}\label{eq:beta_key}
\int_0^\infty \frac{u^{-\alpha}}{u^2+1}\,du
= \frac{1}{2}\,B\!\Bigl(\frac{1-\alpha}{2},\,\frac{1+\alpha}{2}\Bigr)
= \frac{\pi}{2}\,\sec\frac{\pi\alpha}{2},
\qquad \alpha \in [0,1).
\end{equation}
The complementary tail satisfies
\[
\int_{\eta/\varepsilon}^\infty \frac{u^{-\alpha}}{u^2+1}\,du
\le \int_{\eta/\varepsilon}^\infty u^{-\alpha-2}\,du
= \frac{(\varepsilon/\eta)^{1+\alpha}}{1+\alpha}
= O(\varepsilon^{1+\alpha}),
\]
so
\[
M = \frac{c_+}{2}\,\sec\frac{\pi\alpha}{2}\;\varepsilon^{-\alpha}
  + O(\varepsilon).
\]

\emph{Bound on~$R$.}
Since $-\alpha + \delta > -1$ (as $\alpha < 1$), the integral
$\int_0^\infty u^{-\alpha+\delta}/(u^2+1)\,du$ converges
whenever $\delta < 1 + \alpha$, giving $R = O(\varepsilon^{-\alpha+\delta})$.

\medskip\noindent\textbf{Step~5: Assembly.}
The three error contributions are:
the tail~$J_2 = O(\varepsilon)$ from Step~3, 
the beta-integral completion $O(\varepsilon)$ from Step~4,
and the remainder $R$.
For the remainder:
since $\delta < 1+\alpha$ by assumption, the integral
$\int_0^\infty u^{-\alpha+\delta}/(u^2+1)\,du$ converges, giving
$R = O(\varepsilon^{-\alpha+\delta})$.

Therefore
\[
I_+(\varepsilon)
= \frac{c_+}{2}\,\sec\frac{\pi\alpha}{2}\;\varepsilon^{-\alpha}
  + O\!\big(\varepsilon^{-\alpha+\delta} + \varepsilon\big).
\]
When $\delta \le 1+\alpha$ the first error term dominates
(since $-\alpha+\delta \le 1$), giving
$O(\varepsilon^{-\alpha+\delta})$.

The identical argument gives
$I_-(\varepsilon) = \frac{c_-}{2}\sec\frac{\pi\alpha}{2}
\,\varepsilon^{-\alpha} + O(\varepsilon^{-\alpha+\delta})$.
Adding the two contributions completes the proof.
\end{proof}

\subsection{Proof of the Poisson-to-Puiseux lemma}\label{appx:poisson_puiseux} 

\begin{proof}
\medskip\noindent\textbf{Part~1: Puiseux expansion of the density.}
By Proposition~\ref{prop:entrywise-algebraicity}, $H$ is algebraic:
it satisfies $P(z, H) = 0$ for a nontrivial polynomial~$P$
with real coefficients.
By assumption, $z = 0$ is a singularity of the density but
not an atom, so $0$~is a branch point (not a pole) of~$H$.
By the classical Puiseux theorem
\cite[Ch.\;IV, \S4]{walker1950},
the branch of~$H$ selected in~$\mathbb{C}^+$ admits an
expansion
\begin{equation}\label{eq:puiseux_H}
H(z) = \sum_{k \ge k_0} c_k\, z^{k/m},
\qquad z \in \mathbb{C}^+,\; |z| \text{ small},
\end{equation}
for some integer $m \ge 1$ and coefficients $c_k \in \mathbb{C}$,
where $z^{1/m}$ denotes the branch analytic in~$\mathbb{C}^+$
with $z^{1/m} > 0$ for $z > 0$.

The Stieltjes inversion formula
$f(x) = -\frac{1}{\pi}\,\Im\, H(x + i0^+)$
translates~\eqref{eq:puiseux_H} into Puiseux expansions
of~$f$ on each side of the origin.
For $x > 0$: $z^{k/m} = x^{k/m}$ is real, so
$f(x) = -\frac{1}{\pi}\sum_k (\Im\, c_k)\, x^{k/m}$.

For $x < 0$: $z = |x|\,e^{i\pi}$ gives
$z^{k/m} = |x|^{k/m}\, e^{i\pi k/m}$, so
\[
f(-x) = -\frac{1}{\pi}\sum_k
  |c_k|\,x^{k/m}\,\sin(\pi k/m + \arg c_k)
  \]
for $x > 0$ small.
In both cases the exponents are rational numbers of the
form $\beta = k/m$.

The constraint $\beta_1 > -1$ holds because $f \ge 0$ is locally
integrable (since $\mu_S$ has no atom at~$0$ by hypothesis,
$\int_{-\eta}^\eta f(x)\,dx < \infty$), which forces the leading
singular exponent to satisfy $\beta_1 > -1$.

\medskip\noindent\textbf{Part~2: Puiseux exponent determines
  Poisson scaling.}
Let $\alpha_* = -\beta_*$ where
$\beta_* = \min\{\beta_k : a_k^+ \neq 0 \text{ or } a_k^- \neq 0\}$
is the leading singular exponent.
Since $f \ge 0$ and
$f(x) \sim c_+\, x^{-\alpha_*}$ as $x \downarrow 0$ with $c_+ \ge 0$
(resp.\ $f(-x) \sim c_-\, x^{-\alpha_*}$ with $c_- \ge 0$) and
$c_+ + c_- > 0$,
Lemma~\ref{lem:puiseux-to-poisson} gives
\begin{equation}\label{eq:poisson_f_asymp}
\frac{1}{\pi}\int_{\mathbb R}
  \frac{\varepsilon\,f(x)}{x^2+\varepsilon^2}\,dx
\;=\;
\frac{c_+ + c_-}{2}\,\sec\frac{\pi\alpha_*}{2}\;
\varepsilon^{-\alpha_*}
+ O(\varepsilon^{-\alpha_*+\delta}).
\end{equation}
Since $\mu_S$ has no atom at~$0$, its singular part (if any)
is supported at distance at least $r > 0$ from the origin, and
its Poisson contribution is
\[
\frac{1}{\pi}\int_{\mathbb R}
  \frac{\varepsilon}{x^2+\varepsilon^2}\,d\mu_S^{\mathrm{sing}}(x)
\;\le\;
\frac{\varepsilon}{\pi r^2}\,\mu_S^{\mathrm{sing}}(\mathbb R)
\;=\; O(\varepsilon).
\]
As $\alpha_* \ge 0$, the $O(\varepsilon)$ term is absorbed by~%
\eqref{eq:poisson_f_asymp}, giving
$\mathcal P_S(\varepsilon) \asymp \varepsilon^{-\alpha_*}$.

\medskip\noindent\textbf{Part~3: Converse.}
Part~2 establishes $\mathcal P_S(\varepsilon) \asymp \varepsilon^{-\alpha_*}$
unconditionally (given the Puiseux expansion).
If simultaneously
$\mathcal P_S(\varepsilon) \asymp \varepsilon^{-\alpha}$,
then
$\varepsilon^{-\alpha} \asymp \varepsilon^{-\alpha_*}$
as $\varepsilon \downarrow 0$, which forces $\alpha = \alpha_*$.
\end{proof}

\section{Types~II and~III preserve no maximal abelian subalgebra}
\label{sec:no-MASA}

\begin{proof}[Proof of Lemma \ref{lem:no-MASA}]
\emph{Moments and the commutator identity.}
As an $M_N(\bC)$-valued semicircular element with covariance $\eta$, $S$ has
\[
\mu_{2p}=\sum_{\pi\in NC_2(2p)}\eta_\pi(I),
\]
 the sum over non-crossing pair
partitions of $\{1,\dots,2p\}$ of the associated nested applications of $\eta$ to
$I$; odd moments vanish since $NC_2$ of an odd set is empty. The unique pairing of
$\{1,2\}$ gives $\mu_2=\eta(I)$; the two non-crossing pairings of $\{1,2,3,4\}$
---nested $\{(1,4),(2,3)\}$ and side-by-side $\{(1,2),(3,4)\}$--- give
$\mu_4=\eta^2(I)+\eta(I)^2$. Since $[\eta(I),\eta(I)^2]=0$,
\eqref{eq:moment-commutator} follows.

\emph{Reduction to a commutator.}
Suppose $\eta(\mathcal A)\subseteq\mathcal A$ for a maximal abelian subalgebra
$\mathcal A\subseteq M_N(\bC)$. Then $I\in\mathcal A$ gives $\eta^j(I)\in\mathcal A$
for all $j\ge0$, so $\mu_2,\mu_4\in\mathcal A$ commute and, by
\eqref{eq:moment-commutator}, $[\eta(I),\eta^2(I)]=0$. It therefore suffices to
show $[\eta(I),\eta^2(I)]\neq0$ for the two cell types.

\emph{Type~II.}
Write $F=F_n$, $G=G_n$, and let $J$ be the upper shift $J_{i,i+1}=1$. With
$\bar i:=n{+}1{-}i$ and the convention that an entry with an index outside
$\{1,\dots,n\}$ is $0$,
\begin{align*}
(FXF)_{ij}&=X_{\bar i\,\bar j},\quad
(GXG)_{ij}=X_{n+2-i,\,n+2-j},\\
(FXG)_{ij}&=X_{\bar i,\,n+2-j},\quad
(GXF)_{ij}=X_{n+2-i,\,\bar j},
\end{align*}
and $\eta(X)=(\alpha^2{+}1)FXF+\alpha(FXG+GXF)+GXG$. Evaluating at $X=I$ (the $GXG$
term adds $1$ to every diagonal entry except the first) yields the Jacobi matrix
\[
C_1:=\eta(I)=(\alpha^2+2)\,I-E_{11}+\alpha\,(J+J^\top),
\]
with off-diagonal entries all equal to $\alpha$ and diagonal
$(\alpha^2+1,\alpha^2+2,\dots,\alpha^2+2)$. Put $C_2:=\eta(C_1)$. A direct
computation from the formulas above gives, for $n\ge3$,
\[
(C_2)_{n-1,n-1}-(C_2)_{nn}=\alpha^2+1,\qquad (C_2)_{n,n-2}=\alpha^2 .
\]
Because $C_1$ is tridiagonal with $(C_1)_{n-1,n-1}=(C_1)_{nn}=\alpha^2+2$ and
off-diagonal entries $\alpha$,
\[
[C_1,C_2]_{n,n-1}
=\alpha\big[(C_2)_{n-1,n-1}-(C_2)_{nn}-(C_2)_{n,n-2}\big]
=\alpha\big[(\alpha^2+1)-\alpha^2\big]=\alpha\neq0 ;
\]
in fact $[C_1,C_2]=\alpha(E_{21}-E_{12})+\alpha(E_{n,n-1}-E_{n-1,n})$. At $n=2$ the
two boundary terms coincide and $[C_1,C_2]=2\alpha(E_{21}-E_{12})\neq0$.

\emph{Type~III.}
Here $A_1=\bmatr{0&H\\ H^*&0}$ and $A_2=\bmatr{0&F\\ F&0}$ with $H=\beta F+G$,
$F=F_m$, $G=G_m$; both coefficients are block-anti-diagonal and Hermitian, so
$\eta$ preserves block-diagonal matrices, with
\[
\eta\bmatr{X_1&0\\0&X_2}=\bmatr{HX_2H^*+FX_2F&0\\0&H^*X_1H+FX_1F}.
\]
Thus $\eta(I)=\diag(B_+,B_-)$, where $B_+=HH^*+I=\overline{B_-}$ is the complex
Hermitian Jacobi matrix
\[
B_+=(|\beta|^2+2)\,I-E_{11}+\beta\,J+\bar\beta\,J^\top .
\]
With $\eta_+(Y):=HYH^*+FYF$, the top block of $\eta^2(I)$ is $\eta_+(B_-)$, so the
top block of $[\eta(I),\eta^2(I)]$ is $[B_+,\eta_+(B_-)]$. The corner computation
parallel to Type~II gives, for $m\ge3$,
\[
[B_+,\eta_+(B_-)]_{12}=-\beta\qquad(\text{and }-2\beta\text{ at }m=2),
\]
nonzero since $\beta\in\bC\setminus\bR$ forces $\beta\neq0$. Hence
$[\eta(I),\eta^2(I)]\neq0$. (For $m=1$ the cell is LR-semisimple,
$\eta(I)=(1+|\beta|^2)I_2$ is scalar, and the commutator vanishes.)
\end{proof}

\section{Computations for Type~II cell analysis}
\label{sec:type2-computations} 

\subsection{Real-analyticity of $W_\alpha(u)$}  
\label{sec:analyticity}

\begin{proof}[Proof of \ref{thm:analyticity-in-alpha}]
Define
\begin{equation}
\label{eq:Phi-residual}
\Phi(W; \alpha, u) \;:=\; W^{-1} - uI - \eta_\alpha(W),\qquad W\in\mathrm{Sym}_n^{++},\ \alpha\in\mathbb R,\ u>0.
\end{equation}
The map $\Phi$ is real-analytic in all of its arguments: $W^{-1}$ is real-analytic on $\mathrm{Sym}_n^{++}$, and $\eta_\alpha(W) = F_n W F_n + (\alpha F_n + G_n) W (\alpha F_n + G_n)$ is polynomial in $\alpha$ and linear in $W$. Speicher's equation reads $\Phi(W_\alpha(u); \alpha, u) = 0$.

By the real-analytic implicit function theorem, it suffices to verify that the Fr\'echet derivative
\[
\mathcal L \;:=\; D_W \Phi(W_\alpha(u); \alpha, u)\colon \mathrm{Sym}_n \to \mathrm{Sym}_n
\]
is invertible at every $(\alpha, u)$ with $u>0$. Direct computation gives
\begin{equation}
\label{eq:DWPhi}
\mathcal L[V] \;=\; -W^{-1} V W^{-1} - \eta_\alpha(V),\qquad V\in\mathrm{Sym}_n,
\end{equation}
where we abbreviate $W := W_\alpha(u)$. We prove invertibility of $\mathcal L$ in two steps.

\smallskip
\emph{Step 1: Rewriting $\mathcal L$.} 

Substitute the change of variable $V = W^{1/2} U W^{1/2}$ for $U\in\mathrm{Sym}_n$. (The square root $W^{1/2}$ is well-defined and in $\mathrm{Sym}_n^{++}$ since $W \succ 0$.) Then
\[
W^{-1} V W^{-1} = W^{-1/2} U W^{-1/2},
\]
and applying $W^{1/2}$ on both sides of \eqref{eq:DWPhi},
\begin{align}
\label{eq:L-rewritten}
W^{1/2}\,\LL[W^{1/2} U W^{1/2}]\,W^{1/2}
&=\; -U \;-\; W^{1/2}\,\eta_\alpha(W^{1/2} U W^{1/2})\,W^{1/2}
\\
\notag
&=\; -(I + T)(U),
\end{align}
where we have introduced the linear operator
\begin{equation}
\label{eq:T-def}
T(U) \;:=\; W^{1/2}\,\eta_\alpha(W^{1/2} U W^{1/2})\,W^{1/2}\colon \mathrm{Sym}_n \to \mathrm{Sym}_n.
\end{equation}
The transformation $V \leftrightarrow U$ is a linear isomorphism of $\mathrm{Sym}_n$, so invertibility of $\mathcal L$ is equivalent to invertibility of $I + T$.

\smallskip
\emph{Step 2: $T$ is a CP map with $T(I) = I - uW$.}

Each summand $\eta_\alpha(W^{1/2} U W^{1/2}) = \sum_k K_{\alpha,k}\, W^{1/2} U W^{1/2}\, K_{\alpha,k}$ (with $K_{\alpha,1} = F_n$, $K_{\alpha,2} = \alpha F_n + G_n$) is the Kraus representation of a CP map applied to $W^{1/2} U W^{1/2}$. 

The conjugation $X \mapsto W^{1/2} X W^{1/2}$ is also CP. Therefore $T$ is a composition of CP maps and is itself CP. In particular, $T$ preserves the cone $\mathrm{Sym}_n^+$ of positive semidefinite matrices.

Evaluate $T$ at $U = I$:
\[
T(I) \;=\; W^{1/2}\,\eta_\alpha(W)\,W^{1/2}.
\]
At the solution $W = W_\alpha(u)$, Speicher's equation gives $\eta_\alpha(W) = W^{-1} - uI$, hence
\begin{equation}
\label{eq:T-of-I}
T(I) \;=\; W^{1/2}(W^{-1} - uI)\,W^{1/2} \;=\; I - u\,W.
\end{equation}
Since $W\succ 0$ and $u > 0$, we conclude
\begin{equation}
\label{eq:T-strictly-bounded}
T(I) \;\prec\; I.
\end{equation}

\smallskip
\emph{Step 3: $\rho(T) < 1$ via the Collatz--Wielandt principle.}

Let $\lambda := u\,\lambda_{\min}(W) > 0$. From \eqref{eq:T-of-I} we have $T(I) \preceq (1 - \lambda)\,I$. Since $T$ is CP and the Loewner order is preserved by CP maps applied to PSD elements, induction yields
\[
T^k(I) \;\preceq\; (1-\lambda)^k\, T^{k-1}(I) \;\preceq\; \cdots \;\preceq\; (1-\lambda)^k\, I.
\]
More precisely: by the monotonicity of CP maps on PSD elements, $T(X) \preceq T(Y)$ whenever $X \preceq Y$, $X, Y \succeq 0$. Applying this iteratively starting from $T(I) \preceq (1-\lambda) I$,
\[
T^2(I) = T(T(I)) \preceq T((1-\lambda) I) = (1-\lambda) T(I) \preceq (1-\lambda)^2 I,
\]
and inductively $T^k(I) \preceq (1-\lambda)^k\, I$ for every $k\ge 0$.

For any $U\in\mathrm{Sym}_n^+$, $U \preceq \|U\|_{\mathrm{op}} I$, so
\[
0 \preceq T^k(U) \preceq \|U\|_{\mathrm{op}}\,T^k(I) \preceq \|U\|_{\mathrm{op}}\,(1-\lambda)^k\,I,
\]
hence $\|T^k(U)\|_{\mathrm{op}} \le (1-\lambda)^k \|U\|_{\mathrm{op}}$. 

Decomposing a general $U\in\mathrm{Sym}_n$ as $U = U_+ - U_-$ with $U_\pm\in\mathrm{Sym}_n^+$ and $\|U_\pm\|_{\mathrm{op}} \le \|U\|_{\mathrm{op}}$, we obtain
\[
\|T^k(U)\|_{\mathrm{op}} \;\le\; 2(1-\lambda)^k\,\|U\|_{\mathrm{op}}.
\]
By Gelfand's formula, $\rho(T) = \lim_k \|T^k\|^{1/k} \le 1 - \lambda < 1$.

\smallskip
\emph{Step 4: Invertibility of $\mathcal L$.}
Since $\rho(T) < 1$, the operator $I + T$ is invertible, with $(I+T)^{-1} = \sum_{k\ge 0}(-T)^k$ (Neumann series). By Step 1, $\mathcal L$ is invertible.

The real-analytic implicit function theorem now yields a unique real-analytic function $\widetilde W$ on a neighbourhood of any $(\alpha_0, u_0)$ with $u_0 > 0$, satisfying $\widetilde W(\alpha_0, u_0) = W_{\alpha_0}(u_0)$ and $\Phi(\widetilde W(\alpha, u); \alpha, u) = 0$. By global uniqueness of the accretive solution (\cite{hfs2007}), $\widetilde W = W_\cdot(\cdot)$ on the intersection of their domains, so the local pieces glue into a real-analytic function on all of $\mathbb R \times (0, \infty)$.
\end{proof}

\subsection{The Jacobian $\BB_\alpha$}
\label{sec:jacobian}

\begin{proof}[Proof of Proposition \ref{prop:Balpha-linearization}]
We split the proof into three parts: first we compute the coefficients
needed from the even equation, then we compute the coefficient of $t^2$ in
the linearized odd equation, and finally we prove invertibility by an energy
identity.

\medskip\noindent
\textbf{Step 0: Chain Rule.}
After solving the even equation, we write
\[
        N(Y, t): =e^Y+\Phi(Y,\alpha,t),
        \qquad
        Y\in E_-,
        \qquad
        \Phi(Y,\alpha,t)\in E_+.
\]
At $Y=0$, set
\[
        \Psi(t):=\Phi(0,\alpha,t).
\]
Fix $H\in E_-$ and set
\[
        \Xi_H(t):=D_Y\Phi(0,\alpha,t)[H].
\]

By definition $\mathscr R(Y, \alpha, t) = P_- \widetilde{\mathcal G}(N(Y, t), \alpha, t)$. Differentiating in $Y$ in
the direction $H$:
\[
        D_Y\,\mathscr R(0, \alpha, t)[H]
        \;=\; P_-\, D_N\,\widetilde{\mathcal G}\bigl(N(0, t), \alpha, t\bigr)\bigl[\,D_Y N(0, t)[H]\,\bigr].
\]
At $Y = 0$,
\begin{align*}
        N(0, t) & =\; I + \Psi(t), \text{ and} 
        \\
        D_Y N(0, t)[H] &=\; D_Y(e^{Y})|_{Y=0}[H] + D_Y\Phi(0, \alpha, t)[H]
        \;=\; H + \Xi_H(t),
\end{align*}
where we used
$D_Y(e^Y)_{Y=0}[H] = H$ (the standard formula for the differential of
the exponential at the identity).

Combining,
\begin{equation}\label{eq:Balpha-chain-rule}
        D_Y\,\mathscr R(0, \alpha, t)[H]
        \;=\; P_-\,D_N\,\widetilde{\mathcal G}\bigl(I + \Psi(t), \alpha, t\bigr)
        \bigl[\,H + \Xi_H(t)\,\bigr].
\end{equation}

Similarly, if we differentiate 
\[
        P_+\widetilde{\mathcal G}(N(Y, t), \alpha, t) = 0
\]
\emph{in $Y$ at $Y = 0$} in the direction of $H \in E_-$, we get

\begin{equation}
\label{eq:Balpha-chain-rule_plus}
        P_+\,D_N\widetilde{\mathcal G}(I + \Psi(t), \alpha, t)\bigl[H + \Xi_H(t)\bigr] \;=\; 0.
\end{equation}

\textbf{Decomposition of the chain rule.}
The operator $D_N\widetilde{\mathcal G}$ is linear in its bracketed argument,
so we may split
\begin{equation}
\label{equ:decomposition}
        D_N\widetilde{\mathcal G}\bigl[\,H + \Xi_H(t)\,\bigr]
        \;=\;
        \underbrace{D_N\widetilde{\mathcal G}\bigl[H\bigr]}_{\text{piece (I)}}
        \;+\;
        \underbrace{D_N\widetilde{\mathcal G}\bigl[\Xi_H(t)\bigr]}_{\text{piece (II)}},
\end{equation}
where the operator $D_N\widetilde{\mathcal G}$ is evaluated at
$(I + \Psi(t), \alpha, t)$ in both. We will deal with each piece separately
and then sum.

\medskip\noindent
\textbf{Step 1: Coefficients from the even equation.}
We expand
\[
        \Psi(t)=z_1t+z_2t^2+O(t^3),
\]
and
\[
        \Xi_H(t)=s_1(H)t+s_2(H)t^2+O(t^3).
\]
Since $\Phi$ takes values in $E_+$, we have
\[
        z_1,z_2,s_1(H),s_2(H)\in E_+.
\]

By \eqref{eq:z1-formula}, 
\[
        z_1=-\frac a2T_n.
\]

Substitute the expansions for $\Psi(t)$ and $\Xi_H(t)$ into \eqref{eq:Balpha-chain-rule_plus}.

\textbf{The coefficient of $t^0$.} At $t = 0$ the operator is
$D_N\widetilde{\mathcal G}(I, \alpha, 0)[K] = -K - \Theta(K)$ and the
bracketed argument is $H$. So we get $P_+[-H - \Theta(H)] = 0$, which
holds automatically because $H \in E_-$ implies $\Theta(H) = -H$ and the
bracket vanishes.

\textbf{The coefficient of $t$.}
The $t^1$ part of $D_N\widetilde{\mathcal G}(I + \Psi(t), \alpha, t)$ comes
from two sources: (a) the inverse term, where $(I + \Psi(t))^{-1} =
I - z_1 t + O(t^2)$ contributes $z_1\,(\,\cdot\,) + (\,\cdot\,)\,z_1$ at
order $t^1$ (after multiplication by the leading $-1$); and (b) the
$-at\,\mathcal P(\,\cdot\,)$ piece, whose $t^1$ coefficient is just
$-a\,\mathcal P(\,\cdot\,)$. Acting on $H$:
\[
        \bigl[\text{$t^1$ operator}\bigr][H] \;=\; z_1 H + H z_1 - a\,\mathcal P(H).
\]

The $t^0$ operator acting on $s_1(H) \in E_+$ gives $-2\,s_1(H)$ (using
$\Theta(s_1(H)) = s_1(H)$).

Combining, the $t^1$ coefficient of the chain-rule identity is
\[
        P_+\bigl[\,z_1 H + H z_1 - a\,\mathcal P(H) - 2\,s_1(H)\,\bigr] \;=\; 0.
\]
Now $z_1$ is $\Theta$-even and $H$ is $\Theta$-odd, so
\[
        z_1H+Hz_1\in E_-.
\]
Hence its $P_+$-projection vanishes. Therefore
\[
        s_1(H)=-\frac a2P_+\mathcal P(H).
\]

\textbf{The coefficient of $t^2$}
By \eqref{equ:C_2}, the coefficient of $t^2$ in
\[
        P_+\widetilde{\mathcal G}(I+\Psi(t),\alpha,t)=0
\]
gives us
\[
        z_2
        =
        \frac12P_+
        \left(
        z_1^2
        -
        a\mathcal P(z_1)
        -
        b\mathcal Q(I)
        -
        dE_{nn}
        \right).
\]
It turns out that we will not need $s_2(H)$, so we don't compute it. 

\medskip\noindent
\textbf{Step 2: Coefficient of $t^2$ in the linearized odd equation.}
Now we extract the coefficient of $t^2$ in this expression.
\[
        D_Y\mathscr R(0,\alpha,t)[H]
        =
        P_-D_N\widetilde{\mathcal G}(I+\Psi(t),\alpha,t)
        [H+\Xi_H(t)].
\]
We consider two pieces indicated in \eqref{equ:decomposition} sequentially.

\textbf{Piece (I): $D_N\widetilde{\mathcal G}[H]$ at order $t^{2}$.}
We expand
\begin{align*}
        D_N\widetilde{\mathcal G}(I + \Psi(t), \alpha, t)[H]
        &=\; -(I + \Psi(t))^{-1}\,H\,(I + \Psi(t))^{-1}
        \\
        &-\; \Theta(H)
        \;-\; at\,\mathcal P(H)
        \;-\; bt^{2}\,\mathcal Q(H).
\end{align*}
Using $(I + \Psi(t))^{-1} = I - z_1\, t + (z_1^{2} - z_2)\, t^{2} + O(t^{3})$,
the four terms contribute as follows.

\smallskip
\emph{The inverse term.} The coefficient of $t^{2}$ in
$(I + \Psi(t))^{-1}\,H\,(I + \Psi(t))^{-1}$ collects three contributions:
$(I + \Psi)^{-1}$ at $t^2$ on the left with $H$ and $I$ on the right
(giving $(z_1^2 - z_2) H$); the same with the two factors swapped
(giving $H (z_1^2 - z_2)$); and the cross term where both
$(I + \Psi)^{-1}$ factors contribute at order $t^1$ (giving
$(-z_1) H (-z_1) = z_1 H z_1$). So
\[
        \bigl[(I + \Psi)^{-1} H (I + \Psi)^{-1}\bigr]_{t^2}
        \;=\; (z_1^{2} - z_2) H + H (z_1^{2} - z_2) + z_1 H z_1,
\]
and the contribution of $-(I + \Psi)^{-1} H (I + \Psi)^{-1}$ at order
$t^2$ is the negative of this.

\emph{The other terms.} $-\Theta(H)$ has no $t$-dependence, so contributes
nothing at order $t^2$. $-at\,\mathcal P(H)$ has no $t^2$ coefficient
either. $-bt^2\,\mathcal Q(H)$ contributes $-b\,\mathcal Q(H)$.

\smallskip
Summing:
\begin{equation}\label{eq:piece-I}
        \bigl[\text{piece (I)}\bigr]_{t^{2}}
        \;=\; -(z_1^{2} - z_2)\,H - H\,(z_1^{2} - z_2) - z_1\,H\,z_1 - b\,\mathcal Q(H).
\end{equation}

\textbf{Piece (II): $D_N\widetilde{\mathcal G}[\Xi_H(t)]$ at order $t^{2}$.}

The coefficient $s_2(H)$ of $\Xi_H$
does not contribute. Indeed, the
operator $D_N\widetilde{\mathcal G}(I, \alpha, 0)$, evaluated at the
unperturbed base point, acts on its argument $K$ as
\[
        D_N\widetilde{\mathcal G}(I, \alpha, 0)[K] \;=\; -K - \Theta(K).
\]
For $K \in E_+$ this is $-2K \in E_+$, which is killed by $P_-$. In piece
(II), the $s_2(H)$
contribution at order $t^2$ specifically pairs $s_2(H)\, t^2$ with the
$t^0$ part $D_N\widetilde{\mathcal G}(I, \alpha, 0)$, giving
$-2\,s_2(H) \in E_+$ — which $P_-$ annihilates.

It follows that only $s_1(H)$ contributes at total order
$t^{2}$ once we apply $P_-$. The contributions are:
\begin{itemize}
\item $s_1(H)$ paired with $-(I + \Psi(t))^{-1}\,(\,\cdot\,)\,(I + \Psi(t))^{-1}$
at order $t^1$ on the operator: the operator's $t^1$ coefficient is
$z_1\,(\,\cdot\,) + (\,\cdot\,)\,z_1$ (from differentiating either of the
two factors in the inverse), so this gives
$z_1\,s_1(H) + s_1(H)\,z_1$.

\item $s_1(H)$ paired with $-at\,\mathcal P(\,\cdot\,)$ at order $t^0$:
the operator at order $t^0$ is just $-a\,\mathcal P$ (from the $-at$
prefactor), but we also need to factor in the $t$ from $\Xi_H$. So this
gives $-a\,\mathcal P(s_1(H))$ at order $t^2$.
\end{itemize}

Hence, for the second piece we have: 
\begin{equation}\label{eq:piece-II}
        P_-\bigl[\text{piece (II)}\bigr]_{t^{2}}
        \;=\; P_-\Bigl[\,z_1\,s_1(H) + s_1(H)\,z_1 - a\,\mathcal P(s_1(H))\,\Bigr].
\end{equation}

\textbf{Combining.}
Adding \eqref{eq:piece-I} and \eqref{eq:piece-II}:
\begin{equation}\label{eq:Balpha-raw}
\begin{split}
        \mathcal B_\alpha(H)
        \;=\; P_-\Big[\,
        & z_1\, s_1(H) + s_1(H)\, z_1
        \;-\; (z_1^{2} - z_2)\,H \;-\; H\,(z_1^{2} - z_2) \\
        &\;-\; z_1\, H\, z_1
        \;-\; a\,\mathcal P(s_1(H))
        \;-\; b\,\mathcal Q(H)
        \,\Big].
\end{split}
\end{equation}
This is the raw formula for the coefficient of $t^2$.

\bigskip
Substituting
\[
        z_1=-\frac a2T_n,
        \qquad
        s_1(H)=-\frac a2P_+\mathcal P(H),
\]
and
\[
        z_2
        =
        \frac12P_+
        \left(
        z_1^2
        -
        a\mathcal P(z_1)
        -
        b\mathcal Q(I)
        -
        dE_{nn}
        \right),
\]
then using
\[
        L_nU_n=I-E_{11},
        \qquad
        U_nL_n=I-E_{nn},
\]
\[
        T_n=L_n+U_n,
        \qquad
        D_\partial=L_nU_n+U_nL_n,
\]
and
\[
        a^2-b
        =
        \frac{\alpha^2}{c^2}-\frac1c
        =
        -\frac1{c^2},
\]
the raw expression simplifies to
\[
        \mathcal B_\alpha(H)
        =
        \frac{1}{4c^2}
        \left[
        2(L_nHU_n+U_nHL_n)
        -
        D_{\partial}H
        -
        HD_{\partial}
        \right]
        -
        \frac{1}{4\sqrt c}
        \left(
        E_{\partial}H+HE_{\partial}
        \right).
\]
This proves the stated formula.

\medskip\noindent
\textbf{Step 3: Energy identity and invertibility.}
We use the Hilbert--Schmidt inner product
\[
        \langle X,Y\rangle_{\mathrm{HS}}
        :=
        \operatorname{tr}(X^*Y).
\]
For Hermitian $H$, we claim that
\[
\begin{aligned}
        &\left\langle
        H,\,
        2(L_nHU_n+U_nHL_n)
        -
        D_\partial H
        -
        HD_\partial
        \right\rangle_{\mathrm{HS}}
\\
        &\qquad
        =
        -\|L_nH-HL_n\|_{\mathrm{HS}}^2
        -
        \|U_nH-HU_n\|_{\mathrm{HS}}^2.
\end{aligned}
\]
Indeed, since $L_n^*=U_n$ and $U_n^*=L_n$, expanding the two commutator
norms and using cyclicity of the trace gives exactly the displayed
identity.

The boundary term satisfies
\[
        \left\langle H,E_\partial H+HE_\partial\right\rangle_{\mathrm{HS}}
        =
        \|E_\partial H\|_{\mathrm{HS}}^2
        +
        \|HE_\partial\|_{\mathrm{HS}}^2.
\]
Therefore
\[
\begin{aligned}
        \left\langle H,\mathcal B_\alpha(H)\right\rangle_{\mathrm{HS}}
        =
        &-\frac{1}{4c^2}
        \left(
        \|L_nH-HL_n\|_{\mathrm{HS}}^2
        +
        \|U_nH-HU_n\|_{\mathrm{HS}}^2
        \right)
\\
        &-\frac{1}{4\sqrt c}
        \left(
        \|E_\partial H\|_{\mathrm{HS}}^2
        +
        \|HE_\partial\|_{\mathrm{HS}}^2
        \right).
\end{aligned}
\]
Since $c=1+\alpha^2>0$, all coefficients on the right-hand side are
strictly negative. Hence $\mathcal B_\alpha$ is negative semidefinite.

If equality holds, then
\[
        L_nH=HL_n,
        \qquad
        U_nH=HU_n,
        \qquad
        E_\partial H=0,
        \qquad
        HE_\partial=0.
\]
The first two identities imply that $H$ commutes with the algebra generated
by $L_n$ and $U_n$. This algebra is all of $M_n(\mathbb C)$. Indeed,
\[
        L_nU_n=I-E_{11},
        \qquad
        U_nL_n=I-E_{nn},
\]
so $E_{11}$ and $E_{nn}$ lie in the algebra generated by $L_n$ and $U_n$.
Then
\[
        E_{11}L_n^k=E_{1,k+1},
        \qquad
        U_n^kE_{11}=E_{k+1,1},
        \qquad
        k=0,\ldots,n-1,
\]
and hence all matrix units $E_{ij}$ belong to the generated algebra.
Thus the commutant is only the scalar matrices, so $H=\lambda I$.

The boundary condition
\[
        E_\partial H=0
\]
then gives
\[
        \lambda E_\partial=0,
\]
hence $\lambda=0$. Therefore $H=0$.

Thus equality in the energy identity is possible only for $H=0$. Hence
$\mathcal B_\alpha$ is negative definite, and in particular invertible on
$E_-$. This completes the proof.
\end{proof}

\subsection{Reflection Identity} 
\label{sec:reflection_identity}

\begin{proof}[Proof of Lemma~\ref{lem:refl-id}]
By symmetry of $\phi$ and the index relation $(\bar i,\overline{i+2})=
(\bar i,\bar i-2)$, it suffices to check~\eqref{eq:refl} for $j=i+2$.

\emph{Computation of $\cP(\phi)$ on second off-diagonals.}
From \eqref{eq:P-action},
\[
        \cP(\phi)_{i,i+2}
        =\phi_{\bar i,\bar i-3}+\phi_{\bar i-1,\bar i-2}
        =0+\phi_{\bar i-1,\bar i-2},
\]
since $\phi$ is supported on bands $\pm 1$. Using
$\Theta(\phi)=\phi$ and $\overline{\bar i-1}=i+1$,
$\overline{\bar i-2}=i+2$,
\begin{equation}\label{eq:Pphi-upper}
        \cP(\phi)_{i,i+2}
        \;=\; \phi_{\bar i-1,\bar i-2}
        \;=\; \phi_{i+1,i+2}.
\end{equation}
Applying \eqref{eq:P-action} again with $(i,j)$ replaced by
$(\bar i,\bar i-2)$,
\begin{equation}\label{eq:Pphi-lower}
        \cP(\phi)_{\bar i,\bar i-2}
        \;=\; \phi_{i,i+1}+\phi_{i-1,i+2}
        \;=\; \phi_{i,i+1}.
\end{equation}

\emph{Computation of $A\phi A\phi A$ on second off-diagonals.}
Since $A$ is diagonal and $\phi$ is tridiagonal, only the intermediate
index $k=i+1$ contributes to $(A\phi A\phi A)_{i,i+2}$, giving
\begin{equation}\label{eq:cube-upper}
        (A\phi A\phi A)_{i,i+2}
        =r_ir_{i+1}r_{i+2}\,\phi_{i,i+1}\phi_{i+1,i+2}.
\end{equation}
On the reflected entry, using $r_{\bar k}=r_k^{-1}$ and
$\phi_{\bar k,\overline{k+1}}=\phi_{k,k+1}$,
\begin{equation}\label{eq:cube-lower}
        (A\phi A\phi A)_{\bar i,\bar i-2}
        =\frac{1}{r_ir_{i+1}r_{i+2}}\,\phi_{i,i+1}\phi_{i+1,i+2}.
\end{equation}

\emph{Verification of \eqref{eq:refl} for $j=i+2$.}
Write
$p:=\phi_{i,i+1}$, $q:=\phi_{i+1,i+2}$, $\rho:=r_ir_{i+1}r_{i+2}$, and
$\sigma:=r_ir_{i+2}$. Combining \eqref{eq:Pphi-upper}--\eqref{eq:cube-lower},
\[
        B_{i,i+2}=q\,(\rho p-a),
        \qquad
        B_{\bar i,\bar i-2}=p\,(\rho^{-1}q-a).
\]
Therefore
\begin{align*}
        B_{i,i+2}-\sigma\,B_{\bar i,\bar i-2}
        &=q(\rho p-a)-\sigma p(\rho^{-1}q-a) \\
        &=pq\,(\rho-\sigma\rho^{-1})+a(\sigma p-q).
\end{align*}
Substituting \eqref{eq:phi-formula} into the second summand,
\[
        \sigma p-q
        =-a\,r_{i+2}\,\frac{r_ir_{i+1}^{2}r_{i+2}-1}
        {(1+r_ir_{i+1})(1+r_{i+1}r_{i+2})}.
\]
For the first summand,
\[
        pq
        =\frac{a^{2}r_{i+1}r_{i+2}}{(1+r_ir_{i+1})(1+r_{i+1}r_{i+2})},
        \qquad
        \rho-\sigma\rho^{-1}
        =\frac{r_ir_{i+1}^{2}r_{i+2}-1}{r_{i+1}},
\]
so
\[
        pq\,(\rho-\sigma\rho^{-1})
        =\frac{a^{2}r_{i+2}(r_ir_{i+1}^{2}r_{i+2}-1)}
        {(1+r_ir_{i+1})(1+r_{i+1}r_{i+2})}.
\]
Hence
$B_{i,i+2}-\sigma B_{\bar i,\bar i-2}
= pq(\rho-\sigma\rho^{-1})+a(\sigma p-q)=0$,
which is \eqref{eq:refl}.
\end{proof}

\subsection{Type II cells: Shifted base-point} 
\label{sec:type2-shifted-basepoint}

\begin{proof}[Proof of \ref{prop:explicit-diagonal-solution}]
Set $Y:=Y_*(\alpha)$ and $A:=e^{-Y}=\diag(r_1,\ldots,r_n)$, where
\[
        r_i
        =
        c^{\frac{3}{n+1}\left(i-\frac{n+1}{2}\right)}.
\]
Since $r_{n+1-i}=r_i^{-1}$, we have $Y\in D_-$.

It remains to verify $\mathfrak R(Y,\alpha,0)=0$. By the discussion
above, the equation $\fR(Y,\alpha,0)=0$ is the compatibility condition
that the order-$t^2$ Stein equation $A\psi A+\psi=R$ admit a solution
$\psi\in E_+$, where
\begin{equation}\label{eq:R-definition}
        R
        :=
        A\phi A\phi A
        -
        a\mathcal P(\phi)
        -
        b\mathcal Q(e^Y)
        -
        dE_{nn}
\end{equation}
and $\phi$ is the first-order correction with entries $p_i$ given
by~\eqref{eq:phi-formula}. Since $A$ is diagonal, the Stein equation is
entrywise, and the $E_+$-compatibility condition reads
\begin{equation}\label{eq:compatibility-all-entries}
        R_{ij}=r_ir_jR_{\bar i\bar j},
        \qquad
        1\le i,j\le n.
\end{equation}
The support of $R$ is contained in the diagonal and the second
off-diagonal bands, so \eqref{eq:compatibility-all-entries} is automatic
outside these bands.

Let
\[
        q:=c^{3/(n+1)}.
\]
Then $r_{i+1}=qr_i$ for $i=1,\ldots,n-1$.
We verify compatibility on the diagonal and on the second off-diagonal.

\medskip\noindent
\textbf{Diagonal entries.}
Let $R_i:=R_{ii}$.
Using \eqref{eq:R-definition}, \eqref{eq:phi-formula}, and
\[
        \bigl(\mathcal Q(e^Y)\bigr)_{ii}
        =
        r_{i+1}
        \quad (i=1,\ldots,n-1),
        \qquad
        \bigl(\mathcal Q(e^Y)\bigr)_{nn}=0,
\]
we get, for $1\le i\le n-1$,
\begin{equation}\label{eq:R-i-formula}
        R_i
        =
        r_i^2
        \left(
        \mathbf 1_{i>1}\,r_{i-1}p_{i-1}^2
        +
        r_{i+1}p_i^2
        \right)
        -
        2ap_i
        -
        br_{i+1}.
\end{equation}
At the right endpoint,
\begin{equation}\label{eq:R-n-formula}
        R_n
        =
        r_n^2r_{n-1}p_{n-1}^2-d.
\end{equation}

For the interior pairs $2\le i\le n-1$, the compatibility condition
$R_i=r_i^2R_{\bar i}$
follows directly from the geometric relation $r_{i+1}=qr_i$ and the
reflection relation $r_{\bar i}=r_i^{-1}$. Indeed, if $s=r_i$, then
\[
        r_{i-1}=q^{-1}s,
        \qquad
        r_i=s,
        \qquad
        r_{i+1}=qs,
\]
whereas the reflected triple is
\[
        r_{\bar i-1}=(qs)^{-1},
        \qquad
        r_{\bar i}=s^{-1},
        \qquad
        r_{\bar i+1}=q s^{-1}.
\]
Substituting these values into \eqref{eq:R-i-formula} gives the identity
$R_i-r_i^2R_{\bar i}=0$.

It remains to check the boundary pair $i=1$ and $i=n$. Set $r:=r_1$.
Then
\[
        r_2=qr,
        \qquad
        r_n=r^{-1},
        \qquad
        r_{n-1}=(qr)^{-1}.
\]
Moreover,
\[
        p_1=p_{n-1}
        =
        -a\,\frac{qr}{1+qr^2}.
\]
Using \eqref{eq:R-i-formula} and \eqref{eq:R-n-formula}, we get
\[
        R_1-r^2R_n
        =
        r(a^2q-bq+dr).
\]
Thus the boundary compatibility is equivalent to
\begin{equation}\label{eq:boundary-compat}
        a^2q-bq+dr=0.
\end{equation}
Now
\[
        b-a^2
        =
        \frac1c-\frac{\alpha^2}{c^2}
        =
        \frac{c-\alpha^2}{c^2}
        =
        \frac1{c^2}.
\]
Therefore \eqref{eq:boundary-compat} is equivalent to
$dr=q/c^2$.
Since $d=c^{-1/2}$, this is $r=qc^{-3/2}$.
But $r=r_1=q^{-(n-1)/2}$, and because $q=c^{3/(n+1)}$,
\[
        q^{-(n-1)/2}=qc^{-3/2}.
\]
Thus the boundary compatibility holds.

\medskip\noindent
\textbf{Second off-diagonal entries.}
It remains to check the entries with $|i-j|=2$. For $1\le i\le n-2$, set
$S_i:=R_{i,i+2}$.
The only terms contributing to $S_i$ are
$A\phi A\phi A$ and $-a\mathcal P(\phi)$. Hence
\begin{equation}\label{eq:S-i-formula}
        S_i
        =
        r_ir_{i+1}r_{i+2}p_ip_{i+1}
        -
        ap_{i+1}.
\end{equation}
Using $r_{i+1}=qr_i$ and the formula \eqref{eq:phi-formula}, one checks
that
\begin{equation}\label{eq:second-offdiag-compat}
        S_i=r_ir_{i+2}S_{\bar i-2},
\end{equation}
where $S_{\bar i-2}=R_{\bar i,\bar i-2}$ is the reflected lower
second-diagonal entry. Substitution of \eqref{eq:phi-formula} gives
equality on both sides.
Therefore the compatibility condition
$R_{i,i+2}=r_ir_{i+2}R_{\bar i,\bar i-2}$
holds for all second off-diagonal entries. The lower second off-diagonal
entries follow by symmetry.

We have now verified \eqref{eq:compatibility-all-entries} for every pair
$(i,j)$. Hence there exists $\psi\in E_+$ solving $A\psi A+\psi=R$.
Equivalently, $\mathfrak R(Y_*(\alpha),\alpha,0)=0$.
\end{proof}

\section*{Data availability}
Data sharing is not applicable to this article, as no datasets were
generated or analysed during the current study.

\bibliographystyle{plainnat}
\bibliography{comtest} 


\end{document}